\documentclass{amsart}[12pt]
\def\doctype{}

\usepackage{latexsym,amssymb,amsmath}
\usepackage{tikz}
\usetikzlibrary{calc}

\usepackage{fancyhdr}
\usepackage{hyperref}

\def\cB{\mathcal{B}}

\def\Z{{\mathbb Z}}
\def\lam{\lambda}

\newcommand{\comment}[1]{}

\newcommand{\Cfourfive}[2] 
{
\foreach \a in {#1}
 \foreach \b in {#2}
    {
  \draw[blue] (-.22+\a,\b)--(.22+\a,\b)--(-.22+\a,\b+0.5)--(-.22+\a,\b);
  \draw[orange] (-.22+\a,\b+1)--(.22+\a,\b+1)--(.22+\a,\b+0.5)--(\a,\b+0.5)--(-.22+\a,\b+1);
  \foreach \c in {0,1,2}
   \foreach \d in {-1,0,1}
     \filldraw (-.22*\d+\a,\b+0.5*\c) circle [radius=.05];
   }
}

\newcommand{\Cfourtrio}[2] 
{
\foreach \a in {#1}
 \foreach \b in {#2}
    {
  \draw[blue] (-.22+\a,\b)--(.22+\a,\b)--(-.22+\a,\b+0.5)--(-.22+\a,\b);
  \draw[blue] (-.22+\a,\b+1.5)--(.22+\a,\b+1.5) to[out=290,in=70] (.22+\a,\b+0.5) to[out=90,in=-30] (-.22+\a,\b+1.5);
  \draw[blue] (-.22+\a,\b+1)--(.22+\a,\b+1)--(\a,\b+0.5)--(-.22+\a,\b+1);
  \foreach \c in {0,1,...,3}
   \foreach \d in {-1,0,1}
     \filldraw (-.22*\d+\a,\b+0.5*\c) circle [radius=.05];
   }
}

\newcommand{\Cfivetrio}[2] 
{
\foreach \a in {#1}
 \foreach \b in {#2}
    {
  \draw[orange] (-.22+\a,\b+2)--(.22+\a,\b+2)--(.22+\a,\b+1.5)--(\a,\b+1.5)--(-.22+\a,\b+2);
  \draw[orange] (.22+\a,\b)--(-.22+\a,\b)--(-.22+\a,\b+0.5)--(\a,\b+0.5)--(.22+\a,\b);
  \draw[orange] (.22+\a,\b+1)--(-.22+\a,\b+1)--(-.22+\a,\b+1.5) to[out=-20,in=110] (.22+\a,\b+0.5)--(.22+\a,\b+1);
  \foreach \c in {0,1,...,4}
   \foreach \d in {-1,0,1}
     \filldraw (-.22*\d+\a,\b+0.5*\c) circle [radius=.05];
   }
}

\newcommand{\Cthree}[2] 
{
\foreach \a in {#1}
 \foreach \b in {#2}
    {
  \draw (-.22+\a,\b)--(.22+\a,\b) to[out=120,in=60] (-.22+\a,\b);
  \filldraw (-.22+\a,\b) circle [radius=.05];
  \filldraw (\a,\b) circle [radius=.05];
  \filldraw (.22+\a,\b) circle [radius=.05];
   }
}

\newcommand{\Csix}[2] 
{
\foreach \a in {#1}
 \foreach \b in {#2}
    {
  \draw (-.22+\a,\b)--(.22+\a,\b)--(-.22+\a,\b+0.5)--(.22+\a,\b+0.5)--(-.22+\a,\b);
  \foreach \c in {0,1}
   \foreach \d in {-1,0,1}
     \filldraw (-.22*\d+\a,\b+0.5*\c) circle [radius=.05];
   }
}

\newcommand{\Ctop}[2] 
{
\foreach \a in {#1}
 \foreach \b in {#2}
    {
  \draw (\a,\b)--(.22+\a,\b+1)--(-.22+\a,\b+1)--(.22+\a,\b+0.5)--(-.22+\a,\b+0.5)--(.22+\a,\b);
  \filldraw (-.22+\a,\b+1) circle [radius=.05];
  \filldraw (\a,\b+1) circle [radius=.05];
  \filldraw (.22+\a,\b+1) circle [radius=.05];
  \filldraw (-.22+\a,\b+0.5) circle [radius=.05];
  \filldraw (\a,\b+0.5) circle [radius=.05];
  \filldraw (.22+\a,\b+0.5) circle [radius=.05];
   }
}

\newcommand{\Cbot}[2] 
{
\foreach \a in {#1}
 \foreach \b in {#2}
    {
  \draw (\a,\b+1)--(-.22+\a,\b)--(.22+\a,\b)--(-.22+\a,\b+0.5)--(.22+\a,\b+0.5)--(-.22+\a,\b+1);
  \filldraw (-.22+\a,\b) circle [radius=.05];
  \filldraw (\a,\b) circle [radius=.05];
  \filldraw (.22+\a,\b) circle [radius=.05];
  \filldraw (-.22+\a,\b+0.5) circle [radius=.05];
  \filldraw (\a,\b+0.5) circle [radius=.05];
  \filldraw (.22+\a,\b+0.5) circle [radius=.05];
   }
}

\allowdisplaybreaks
\numberwithin{equation}{section}

\fancypagestyle{titlepage}{
\fancyhead[R]{\doctype}
\fancyhead[CL]{}
\cfoot{\vspace{5pt} \thepage}
}

\let\oldsection\section
\newcommand\boldsection[1]{\oldsection{\bf #1}}
\newcommand\starsection[1]{\oldsection*{\bf #1}}
\makeatletter
\renewcommand\section{\@ifstar\starsection\boldsection}
\makeatother

\newtheoremstyle{theorem}
  {12pt}		  
  {0pt}  
  {\sl}  
  {\parindent}     
  {\bf}  
  {. }    
  { }    
  {}     
\theoremstyle{theorem}
\newtheorem{thm}{Theorem}[section]  
\newtheorem{lemma}[thm]{Lemma}     
\newtheorem{cor}[thm]{Corollary}

\newtheorem{prop}[thm]{Proposition}

\newtheoremstyle{definition}
  {12pt}		  
  {0pt}  
  {}  
  {\parindent}     
  {\bf}  
  {. }    
  { }    
  {}     
\theoremstyle{definition}

\newtheorem{ex}[thm]{Example}

\newcommand\rk{{\sc Remark.} }

\renewcommand{\proofname}{Proof}

\makeatletter
\renewenvironment{proof}[1][\proofname]{\par
  \pushQED{\qed}%
  \normalfont \partopsep=\z@skip \topsep=\z@skip
  \trivlist
  \item[\hskip\labelsep
        \scshape
    #1\@addpunct{.}]\ignorespaces
}{%
  \popQED\endtrivlist\@endpefalse
}
\makeatother

\makeatletter
\renewcommand*\@maketitle{%
  \normalfont\normalsize
  \@adminfootnotes
  \@mkboth{\@nx\shortauthors}{\@nx\shorttitle}%
  \global\topskip42\p@\relax 
  \@settitle
  \ifx\@empty\authors \else {\vskip 1em
\vtop{\centering\shortauthors\@@par}} \fi
  \ifx\@empty\@date \else {\vskip 1em \vtop{\centering\@date\@@par}}\fi 
  \ifx\@empty\@dedicatory
  \else
    \baselineskip18\p@
    \vtop{\centering{\footnotesize\itshape\@dedicatory\@@par}%
      \global\dimen@i\prevdepth}\prevdepth\dimen@i
  \fi
  \@setabstract
  \normalsize
  \if@titlepage
    \newpage
  \else
    \dimen@34\p@ \advance\dimen@-\baselineskip
    \vskip\dimen@\relax
  \fi
} 
\renewcommand*\@adminfootnotes{%
  \let\@makefnmark\relax  \let\@thefnmark\relax
  \ifx\@empty\@subjclass\else \@footnotetext{\@setsubjclass}\fi
  \ifx\@empty\@keywords\else \@footnotetext{\@setkeywords}\fi
  \ifx\@empty\thankses\else \@footnotetext{%
    \def\par{\let\par\@par}\@setthanks}%
  \fi
\thispagestyle{titlepage}
}
\makeatother

\begin{document}

\title[Leaves for packings with block size four]{\large Leaves for packings with block size four}

\author{Yanxun~Chang}
\address{\rm Yanxun~Chang:
Mathematics,
Beijing Jiaotong University,
Beijing,
P.R.~China
}
\email{yxchang@bjtu.edu.cn}

\author{Peter J.~Dukes}
\address{\rm Peter J.~ Dukes:
Mathematics and Statistics,
University of Victoria, Victoria, Canada
}
\email{dukes@uvic.ca}

\author{Tao Feng}
\address{\rm Tao Feng:
Mathematics,
Beijing Jiaotong University,
Beijing,
P.R.~China
}
\email{tfeng@bjtu.edu.cn}

\thanks{Research of Yanxun Chang is supported by NSFC grant 11431003; research of Peter Dukes is supported by NSERC grant 312595--2017; research of Tao Feng is supported by NSFC grant 11471032; research of this paper was also partially supported by 111 Project of China, grant number B16002.}

\date{\today}


\begin{abstract}
We consider maximum packings of edge-disjoint $4$-cliques in the complete graph $K_n$. When $n \equiv 1$ or $4 \pmod{12}$, these are simply block designs.  In other congruence classes, there are necessarily uncovered edges; we examine the possible `leave' graphs induced by those edges.  We give particular emphasis to the case $n \equiv 0$ or $3 \pmod{12}$, when the leave is $2$-regular.  Colbourn and Ling settled the case of Hamiltonian leaves in this case.  We extend their construction and use several additional direct and recursive constructions to realize a variety of $2$-regular leaves.  For various subsets $S \subseteq \{3,4,5,\dots\}$, we establish explicit lower bounds on $n$ to guarantee the existence of maximum packings with any possible leave whose cycle lengths belong to $S$.
\end{abstract}

\maketitle
\hrule

\section{Introduction}
\label{intro}

Let $n,k,t,\lam$ be nonnegative integers with $n \ge k \ge t$.
A $t$-$(n,k,\lam)$ \emph{packing} is a pair $(X,\mathcal{B})$, where $X$ is a set of size $n$, $\mathcal{B}$ is a collection of $k$-subsets of $X$, and such that, for every $t$-subset $T$ of $X$, there are at most $\lam$ elements of $\mathcal{B}$ which contain $T$.  Elements of $\mathcal{B}$ are called \emph{blocks} and elements of $X$ are called \emph{points} or \emph{vertices}.  The survey \cite{MMsurvey} offers more details on the background results to follow.

Packings are relaxations of designs in the sense that if ``at most'' is replaced by ``exactly'' in the definition of a packing, one recovers the definition of a design. Alternatively, designs are packings with the maximum number $\lam \binom{v}{t} / \binom{k}{t}$ of blocks.  

Packings in the case $t=1$ are simply partial partitions of a ($\lam$-fold) $n$-set by $k$-subsets.  The first interesting case for existence is $t=2$, $\lam=1$.  In the language of graph theory, a packing here is equivalent to a set of edge-disjoint $k$-cliques in the complete graph $K_n$ on $n$ vertices.  There is also some geometric significance here: blocks may be interpreted as lines which cover any two distinct points at most once.

The (first) Johnson bound says that the number of blocks in a $2$-$(n,k,1)$ packing satisfies
\begin{equation}
\label{johnson}
|\mathcal{B}| \le \left\lfloor \frac{n}{k} \left\lfloor \frac{n-1}{k-1} \right\rfloor \right\rfloor.
\end{equation}
The \emph{leave} of a packing $(X,\mathcal{B})$ is the graph of `uncovered pairs' $L=(X,E)$, where $\{x,y\} \in E$ if and only if there is no $B \in \mathcal{B}$ containing $\{x,y\}$.
Often, isolated vertices are discarded in leaves.  For instance, the leave of a maximum $2$-$(5,3,1)$ packing (consisting of two edge-disjoint triangles on 5 vertices) is isomorphic to the 4-cycle $C_4$.

The leave $L$ of a $2$-$(n,k,1)$ packing satisfies the congruence conditions
\begin{itemize}
\item
$|E(L)| \equiv \binom{n}{2} \pmod{\binom{k}{2}}$ and
\item
$\deg_L(x) \equiv n-1 \pmod{k-1}$ for each $x \in X$.
\end{itemize}
As a result, we note that equality in (\ref{johnson}) is sometimes not possible.  An improved upper bound on the number of blocks is
\begin{equation}
\label{leave-bound}
|\mathcal{B}| \le \frac{1}{\binom{k}{2}} \left[ \binom{n}{2} - |E(L)| \right],
\end{equation}
where $L$ is a minimum size simple graph satisfying the above conditions.
As an example, the reader can easily check that for $k=3$ and $n \equiv 5 \pmod{6}$, the right side of (\ref{leave-bound}) is one smaller than that in (\ref{johnson}). Here, $L=C_4$ is the (unique) minimum leave.

Let us denote by MP$(n,k)$ a $2$-$(n,k,1)$ packing whose number of blocks achieves equality in (\ref{leave-bound}).  Caro and Yuster, \cite{CY}, identified candidate leaves and used a graph decomposition result of Gustavsson, \cite{Gus} to settle the existence of MP$(n,k)$ for each $k$ and sufficiently large $n$.  Chee et al., \cite{CCLW}, obtained a slightly weaker result independent of \cite{Gus}.  More recently, Barber et al. \cite{BKLO} and Keevash \cite{Keevash} have verified (and generalized) the needed result for MP$(n,k)$ that all sufficiently large dense graphs admit a $K_k$-decomposition provided the necessary divisibility conditions hold.  So, if any candidate leave $L_n$ is chosen, say with bounded degree, its complement in $K_n$
can be decomposed for $n > n_0(k)$.  Unfortunately, no upper bounds are known on $n_0(k)$. And the randomized construction methods in \cite{BKLO,Keevash} give especially huge worst-case guarantees.

For block size 3, a complete existence result is possible.  When $n \equiv 1,3 \pmod{6}$, an MP$(n,3)$ is just a Steiner triple system, and the leave is edgeless.
When $n \equiv 0,2 \pmod{6}$, an MP$(n,3)$ results from deleting one point (and all incident blocks) from a Steiner triple system of order $n+1$.  In this case, the leave is a perfect matching $\frac{n}{2} K_2$.  For each $n \equiv 5 \pmod{6}$, it is known that $K_n$ decomposes into triangles and one $5$-clique; this is also known as a pairwise balanced design PBD$(v,\{3,5^*\})$.  Replacing the block of size 5 by two edge-disjoint triangles produces an MP$(n,3)$ with leave $C_4$.  Finally, deleting a point from the $4$-cycle in such a construction settles the class $n \equiv 4 \pmod{6}$, where the unique leave for MP$(n,3)$ is $K_{1,3} \cup \frac{n-4}{2} K_2$.  A concise summary of the above appears in \cite[Table 40.22]{hcd}.

When $k=4$, existence of MP$(n,4)$ is known except for a few small values of $n$; see \cite[Table 40.23]{hcd}.  However, in contrast to the case $k=3$, there emerge different possibilities for the leave in some of the congruence classes for $n$.  Indeed, when $n \equiv 0,3 \pmod{12}$, the minimum leave can be any $2$-regular spanning graph.  Deleting a point from a $2$-$(n+1,4,1)$ design produces MP$(n,4)$ in which $L$ is $\frac{n}{3} K_3$.  However, relatively little is known about other possible leaves. A special case of the main result of \cite{DFL} realizes the leave $\frac{n}{4} C_4$ for each $n \equiv 0 \pmod{12}$, $n \ge 24$.   Colbourn and Ling \cite{CL} constructed, for all $n \equiv 0,3 \pmod{12}$, $n \ge 15$, an MP$(n,4)$ with Hamiltonian leave $C_n$; such packings are useful in statistics for sampling plans that exclude cyclically adjacent pairs.  

In this paper, we study the possible leaves in a packing MP$(n,4)$, with particular emphasis on $2$-regular leaves, that is, for the congruence classes $n \equiv 0,3 \pmod{12}$.
The next section sets up some background for our constructions.  As a first step, in Section 3, we obtain explicit bounds on $n$ for the existence of MP$(n,4)$ whose leaves contain a mixture of small cycle lengths.  Then, in Section 4, we adapt a construction from \cite{CL} to merge cycles in the leave.  Concerning other congruence classes, a new leave for MP$(31,4)$ is found, leading to an explicit lower bound for existence of each of two non-isomorphic leaves in the case $n \equiv 7,10 \pmod{12}$.  The classes $n \equiv 6,9 \pmod{12}$ are more difficult, but we offer a few preliminary remarks.  A (surprisingly small) number of explicit packings are needed for our results; these are detailed in an appendix.

\section{Background}
\label{background}

\subsection{Group divisible designs}

Let $v$ be a positive integer, and $T$ be an integer partition of $v$.
A \emph{group divisible design} of \emph{type} $T$ with block sizes in $K$,
abbreviated GDD$(T,K)$ or as a $K$-GDD of type $T$,  is a triple $(V,\Pi,\cB)$ such that
\begin{itemize}
\item
$V$ is a set of $v$ points;
\item
$\Pi=\{V_1,\dots,V_u\}$ is a partition of $V$ into \emph{groups} so that
$T=(|V_1|,\dots,|V_u|)$;
\item
$\cB \subseteq \cup_{k \in K} \binom{V}{k}$ is a set of blocks meeting each
group in at most one point; and
\item
any two points from different groups appear together in exactly one
block.
\end{itemize}

Often in this context, exponential notation such as $n^u$ is used to
abbreviate $u$ parts or `groups' of size $n$.  It is also convenient to drop the
brackets for a single block size and write $k$ instead of $\{k\}$.

\begin{lemma}[Brouwer, Schriver and Hanani, \cite{BSH}]
\label{gdd-unif}
There exists a $4$-GDD of type $g^u$ if and only if $3 \mid g(u-1)$ and $12 \mid g^2u(u-1)$, where $u \ge 4$ and $(g,u) \neq (2,4),(6,4)$.
\end{lemma}

A GDD naturally induces a packing in which the group partition is interpreted as a leave.  For small group sizes, these are maximum packings.  Taking $g=2$ and $g=3$ in Lemma~\ref{gdd-unif} gives the following MP$(n,4)$.

\begin{cor}
\label{K2C3}
\begin{enumerate}
\item
For $n \equiv 2,8 \pmod{12}$, $n \ge 14$, there exists an MP$(n,4)$ with leave $\frac{n}{2} K_2$.
\item
For $n \equiv 0,3 \pmod{12}$, there exists an MP$(n,4)$ whose leave is $\frac{n}{3} C_3$.
\end{enumerate}
\end{cor}

Later, we also require some results on 4-GDDs with all but one group of the same size. 

\begin{lemma}[Ge and Ling, \cite{GL}]
\label{GDD15ux1}
For $u \ge 4$, there exists a $4$-GDD of type
$15^u x^1$ if and only if $u \equiv 0 \pmod{4}$, $x \equiv 0 \pmod{3}$, and $x \le \frac{1}{2}(15u-18)$; or $u \equiv 1 \pmod{4}$, $x \equiv 0 \pmod{6}$, and $x \le \frac{1}{2}(15u-15)$; or $u \equiv 3 \pmod{4}$, $x \equiv 3 \pmod{6}$, and $x \le \frac{1}{2}(15u-15)$.
\end{lemma}

\begin{lemma}[Schuster, \cite{Schuster}]
\label{GDD24ux1}
There exists a $4$-GDD of type $24^u x^1$ if and only if $u \ge 4$, $x \equiv 0 \pmod{3}$, and $x \le 12(u-1)$.  There exists a $4$-GDD of type $120^u x^1$ if and only if $u \ge 4$, $x \equiv 0 \pmod{3}$, and $x \le 60(u-1)$.
\end{lemma}

Some additional results on $4$-GDDs can be found in \cite{Forbes,Forbes2,WG} and the handbook survey \cite[IV 4.1]{hcd}.

\subsection{The fundamental construction}

We cite an important recursive construction for designs by R.M. Wilson.  The main idea is to produce
a new GDD from a given one by replacing points with clusters of points (or removing them),
provided each block is replaced by an appropriate ingredient.

\begin{lemma}[Wilson's Fundamental Construction, \cite{ConsUses}]
\label{wfc-lemma}
Suppose there exists a GDD $(V,\Pi,\cB)$, where $\Pi=\{V_1,\dots,V_u\}$.
Let $\omega:V \rightarrow \Z_{\ge 0}$, assigning nonnegative weights to each
point in such a way that for every $B \in \cB$ there exists a
$K$-GDD of type $[\omega(x) : x \in B]$.  Then there exists a $K$-GDD of type
$$\left[\sum_{x \in V_1} \omega(x),\dots,\sum_{x \in V_u}
\omega(x)\right].$$
\end{lemma}

In our application of Lemma~\ref{wfc-lemma} to follow, we take $K=\{4\}$ and use ingredients as above.

\subsection{Transversal designs}

A \emph{transversal design} TD$(k,n)$ is a $\{k\}$-GDD of type $n^k$.
A TD$(k,n)$ is equivalent to $k-2$ mutually orthogonal latin squares of order $n$, where two groups are
reserved to index the rows and columns of the squares.  It follows that there exists a TD$(k,q)$ when $q \ge k-1$ is a prime
power.  From this and some further constructions, it was shown in \cite{CES} that there exist TD$(k,n)$ for all integers $n \ge n_0(k)$.

A \emph{parallel class} in a design is a collection of blocks which partition the points.  A transversal design TD$(k,n)$ with a parallel class is equivalent
to $k-2$ mutually orthogonal idempotent latin squares of order $n$.  If there exists a TD$(k+1,n)$, then there exists a TD$(k,n)$ having a parallel class, and in fact a `resolvable' such TD.
Later, we have occasion to use some specific bounds on existence of transversal designs; we refer the reader to \S III.3.6 in \cite{hcd} for details.

If we delete points from one group of a transversal design TD$(k,n)$, the result is a $\{k-1,k\}$-GDD of type $n^{k-1} x^1$.  Note that this is a special case of Wilson's fundamental construction in which $\omega=1$ or $0$.

\subsection{Graph divisible designs}

Suppose $T$ is a list of (simple, undirected) graphs $G_1,G_2,\dots,G_u$ on disjoint vertex sets whose union is $X$.
A \emph{graph divisible design} of type $T$ and block size $k$ is an edge-decomposition of the join $G_1 + \dots + G_u$ into cliques $K_k$.
In the case when each $G_i$ is edgeless $\overline{K_{g_i}}$, the result is a group divisible design of type $[g_i:i=1,\dots,u]$.  For this reason, similar notation
($k$-GDD of type $T$) was adopted for this more general case.

Graph divisible designs were introduced in \cite{DL}.  As an example of their utility, an explicit construction
for MP$(n,5)$ was shown in the difficult congruence class $n \equiv 13 \pmod{20}$.

Let $M_r$ denote the 1-regular graph on $2r$ vertices.  Graph divisible designs whose `groups' are perfect matchings $M_r$ of equal sizes were considered in \cite{DFL}.
The following existence result was proved.

\begin{thm}[Dukes, Feng and Ling \cite{DFL}]
A $4$-GDD of type $M_r^u$ exists if and only if $u\geq 4$, $r(u-1) \equiv 1 \pmod{3}$ and $2 \mid ru$.
\end{thm}

Taking $r=2$ and observing that the complement of $M_2$ (on four vertices) is $C_4$, one obtains packings MP$(n,4)$ whose leave is a disjoint union of $4$-cycles.

\begin{cor}
\label{C4}
There exists an MP$(n,4)$ with leave $\frac{n}{4} C_4$ for each $n \equiv 0 \pmod{12}$, $n \ge 24$.
\end{cor}

\subsection{Double and holey GDDs}

A \emph{double group divisible design} with block sizes in $K$, or $K$-DGDD, is a quadruple $(V,\Gamma_1,\Gamma_2,\cB)$ where
\begin{itemize}
\item
$V$ is a set of $v$ points;
\item
$\Gamma_1$ is a partition of $V$ into \emph{groups} and $\Gamma_2$ is a partition of $V$ into \emph{holes};
\item
$\cB \subseteq \cup_{k \in K} \binom{V}{k}$ is a set of blocks meeting each
group and each hole in at most one point; and
\item
any two points from different groups and different holes appear together in exactly one block.
\end{itemize}
Of particular importance is the situation where any group and any hole intersect in the same number, say $a$, of points, each group has the same size, say $ag$, and each hole has the same size, say $ah$.  This case is called a (uniform) \emph{holey group divisible design}, or $K$-HGDD; see \cite{GW}.  To reflect the symmetry between groups and holes, we use the notation $a^{g \times h}$ for the type.  In our applications to follow, $K=\{4\}$ and $a=3$.  The following existence theorem is a special case of Ge and Wei's more general result for $4$-HGDDs, a few cases of which were completed in a later paper.

\begin{lemma}[\cite{CWW,GW}]
\label{hgdd}
There exists a $4$-HGDD of type $3^{g \times h}$ if and only if $g,h \ge 4$.
\end{lemma}

In certain cases a $4$-DGDD with different group and hole sizes can be obtained from Wilson's fundamental construction. For this, we start with a TD$(k,n)$ having a parallel class of blocks, and give weight zero or three to points.   Blocks of the parallel class become holes, and other blocks are replaced with $4$-GDDs of type $3^k$ or $3^{k-1}$.  We apply this method later to produce templates for our constructions of packings.

\section{Short cycle lengths}
\label{short}

To begin our analysis of possible $2$-regular leaves in MP$(n,4)$, we consider various mixtures of short cycle lengths.  The $4$-GDDs in Section~\ref{background} play a crucial role as templates.  We also need some small explicit packings.  An important case $n=24$ was settled computationally and detailed in a \href{http://www.math.uvic.ca/~dukes/24-4-1-packings.pdf}{supplementary file}.

\begin{lemma}
\label{all24}
Any possible $2$-regular graph on $24$ vertices is the leave of some MP$(24,4)$.
\end{lemma}

A few other specific small leaves are helpful; these packings can be found in the appendix.

\begin{lemma}
\label{small-leaves}
There exist MP$(n,4)$ with the following leaves:
\vspace{-10pt}
\begin{itemize}
\item
$n=15$: $L=3C_5$ and $C_3 \cup 2C_6$;
\item
$n=27$: $L=3C_4 \cup 3C_5$, $C_3 \cup 4C_6$, $3C_3 \cup 3C_6$, $5C_3 \cup 2C_6$, and $7C_3 \cup C_6$;
\item
$n=36$: $L=C_3 \cup 2C_4 \cup 5C_5$, $2C_3 \cup 6C_5$, and $6C_6$;
\item
$n=39$: $L=C_3 \cup 9C_4$;
\item
$n=48$: $L=C_3 \cup 9C_5$.
\end{itemize}
\end{lemma}

We can now get started realizing more general leaves.

\begin{prop}
\label{bound34}
For all $n \equiv 0,3 \pmod{12}$, $n \ge 144$, any $2$-regular graph with cycle lengths in $\{3,4\}$ is the leave of some MP$(n,4)$.
\end{prop}

\begin{proof}
Write $n = 24u+x$, where $x \in X:=\{0,3,12,39\}$ and $u \ge 5$.  From Lemma~\ref{GDD24ux1}, there exists a $4$-GDD of type $24^u x^1$.
Fill groups of size 24 with packings having leaves $8C_3$, $4C_3 \cup 3C_4$, or $6C_4$ (where Lemma~\ref{all24} is used).  This completely settles the case $x=0$.  The case $x=3$ is similar, where we regard the last group of the GDD as an additional 3-cycle in the leave.  When $x=12$, fill the last group with a packing having leave $4C_3$; the leave $\frac{n}{4} C_4$ is obtained separately from Corollary~\ref{C4}.  When $x = 39$, fill the last group with a packing having leave $13 C_3$ from Corollary~\ref{K2C3}(b), or $C_3 \cup 9 C_4$ from Lemma~\ref{small-leaves}, according to whether more $3$-cycles or $4$-cycles are desired.
\end{proof}

\begin{prop}
\label{bound35}
For all $n \equiv 0,3 \pmod{12}$, $n \ge 132$, any $2$-regular graph with cycle lengths in $\{3,5\}$ is the leave of some MP$(n,4)$.
\end{prop}

\begin{proof}
Write $n=15u+x$, where $x \in X:=\{0,12,24,36,48\}$ and $u \equiv 0$ or $1 \pmod{4}$, $u \ge 8$.  Under these conditions, Lemma~\ref{GDD15ux1} gives a
$4$-GDD of type $15^u x^1$.  Fill the groups of size 15 with packings having leaves $5C_3$ or $3C_5$, the latter from Lemma~\ref{small-leaves}.  We may fill the group of size $x$ with a packing having leave $\frac{x}{3} C_3$ if a majority of 3-cycles is desired.  To obtain leaves with mostly 5-cycles, it remains to check the existence of packings for the orders in $X$ having the minimum possible number of 3-cycles.  In this case, the desired leave is $j C_3 \cup \frac{x-3j}{5} C_5$, where $j \in \{0,1,2,3,4\}$.  In the case $x=0$ there is nothing more to do.  For $x=12$, we simply use a $4$-GDD of type $3^4$.  When $x=24$, we use a packing having leave $3C_3 \cup 3 C_5$, using Lemma~\ref{all24}.  When $x=36$, we use a packing having leave $2 C_3 \cup 6 C_5$, from Lemma~\ref{small-leaves}.  When $x=48$, we use a packing having leave $C_3 \cup 9 C_5$, also from Lemma~\ref{small-leaves}.
\end{proof}

We next consider cycle lengths in $\{3,4,5\}$. For the following constructions, it is helpful to abbreviate a leave of the form $aC_3 \cup bC_4 \cup cC_5$ as an $(a,b)$-\emph{leave}. 
Given $n,a,b$, note that $c$ is uniquely determined. We begin by realizing various $(a,b)$-leaves with small $a$ and $b$.

\begin{lemma}
\label{specific345}
There exists an MP$(n,4)$ with $(a,b)$-leave in each of the following cases:
\vspace{-10pt}
\begin{enumerate}
\item
$n=276$ and $(a,b)=(4,1)$;
\item
$n=288$ and $(a,b) \in \{(0,2),(2,3),(3,1),(4,4)\}$;
\item
$n=300$ and $(a,b) \in \{(1,3),(2,1),(3,4),(4,2)\}$;
\item
$n=312$ and $(a,b) \in \{(1,1),(2,4),(3,2)\}$.
\end{enumerate}
\end{lemma}

\begin{proof}
(a)
From a TD$(6,5)$, delete points from two groups to obtain a $\{4,5,6\}$-GDD of type $5^4 2^1 1^1$. Give every point weight $12$ and replace blocks with $4$-GDDs of types $12^4$, $12^5$, and $12^6$. This produces a $4$-GDD of type $60^4 24^1 12^1$. Fill the first four groups with packings having leave $12C_5$, which can be obtained from a $4$-GDD of type $15^4$.  Fill the group of size $24$ with an MP$(24,4)$ having leave $C_4 \cup 4C_5$, using Lemma~\ref{all24}, and the group of size $12$ with an MP$(12,4)$ having leave $4C_3$.

(b)
Following a similar construction as in (a), we first obtain a $4$-GDD of type $60^4 24^2$. Fill the first four groups with packings having leave $12C_5$ and the two groups of size 24 with $3C_3 \cup 3C_5$, $2C_3 \cup 2C_4 \cup 2C_5$, or $C_4 \cup 4C_5$, where again Lemma~\ref{all24} is used.

(c)
Similar to before, we first obtain a $4$-GDD of type $60^4 36^1 24^1$. Fill the first four groups with leave $12C_5$, the group of size $36$ with leave $2C_3 \cup 6C_5$ or $C_3 \cup 2C_4 \cup 5C_5$, and the group of size $24$ with leave $2C_3 \cup 2C_4 \cup 2C_5$, $C_3 \cup 4C_4 \cup C_5$, or $C_4 \cup 4C_5$. For the existence of the small packings, refer to Lemmas~\ref{all24} and \ref{small-leaves}.

(d)
This time we fill groups of $4$-GDD of type $60^4 48^1 24^1$, using $12C_5$ for the first four groups, $C_3 \cup 9C_5$ for the next, and the three cases just as in (c) for the last group.
\end{proof}

\begin{lemma}
\label{ab-leave}
For all $n \equiv 0,3 \pmod{12}$, $n \ge 936$, there exists an MP$(n,4)$ having any possible $(a,b)$-leave in which $a,b \le 4$ and $3a+4b \equiv n \pmod{5}$.
\end{lemma}

\begin{proof}
Write $n=15u+x$, where $x$ is chosen as in Table~\ref{min345}, and $u \ge 44$ with $u \equiv 0$ or $\pm1\pmod{4}$, the sign being positive or negative according to whether $x$ is even or odd, respectively.  The lower bound on $u$ implies, by Lemma~\ref{GDD15ux1}, existence of a $4$-GDD of type $15^u x^1$ for any of the given values of $x$.

\begin{table}[htbp]
\begin{center}
\begin{tabular}{c|cccccc}
$x$ & 0 & 1 & 2 & 3 & 4 & $b$ \\
\hline
0 & 0 & 24 & 288 & 27 & 96 \\
1 & 3 & 312 & 36 & 300 & 24 \\
2 & 36 & 300 & 24 & 288 & 312 \\
3 & 24 & 288 & 312 & 96 & 300 \\
4 & 12 & 276 & 300 & 24 & 288 \\
$a$
\end{tabular}
\caption{Cases for small $(a,b)$}
\label{min345}
\end{center}
\end{table}
We claim that there is an MP$(x,4)$ with $(a,b)$-leave.  The twelve large entries in the table correspond with cases in Lemma~\ref{specific345}.  The two occurrences of $x=96$ follow from filling groups of a $4$-GDD of type $24^4$ using either $C_4 \cup 4C_5$ or $3C_3 \cup 3C_5$ as the leave.  The remaining entries are handled by
Lemmas~\ref{all24} and \ref{small-leaves}.  After filling groups of size $15$ with MP$(15,4)$ having leave $3C_5$ and the group of size $x$ using an MP$(x,4)$ with $(a,b)$-leave, we obtain an MP$(n,4)$ with $(a,b)$-leave.
\end{proof}

\begin{thm}
\label{bound345}
For all $n \equiv 0,3 \pmod{12}$, $n \ge 3216$, any $2$-regular graph with cycle lengths in $\{3,4,5\}$ is the leave of some MP$(n,4)$.
\end{thm}

\begin{proof}
Suppose we are given nonnegative integers $a,b,c$ with $3a+4b+5c=n$, and we wish to construct an MP$(n,4)$ with leave $aC_3 \cup b C_4 \cup cC_5$.

We first consider $n=120$.
Filling groups of a $4$-GDD of type $15^8$, via Lemma~\ref{gdd-unif}, with packings having leave $3C_5$ or $5C_3$ results in an MP$(120,4)$ with $(a,0)$-leave for each $a$ a multiple of $5$.  If we also fill groups of a $4$-GDD of type $24^5$ in all possible ways using Lemma~\ref{all24}, we obtain
 (after some routine case checking) any possible $(a,b)$-leave for MP$(120,4)$, with the possible exception of $(a,b)$ equal to
$$(2,1), (1,3), (4,2), (7,1).$$
Call an ordered pair $(a,b)$ of nonnegative integers with $3a+4b\equiv 0 \pmod{5}$ `good' if not in this list.  We remark that any good pair can be written as a sum of good pairs
$(a_i,b_i)$ with $3a_i+4b_i \le 120$.

Now, write $n = 120u + x$, where $u \ge 19$ and $936 \le x \le 1047$.  By Lemma~\ref{GDD24ux1}, there exists a $4$-GDD of type $120^u x^1$.  We proceed according to two cases.

{\sc Case 1}: $3a+4b > x+25$.  Fill the group of size $x$ with an MP$(x,4)$ whose leave has cycle lengths in $\{3,4\}$, appealing to Proposition~\ref{bound34}.  This leaves, say, $a'$ $3$-cycles and $b'$ $4$-cycles to allocate to the remaining groups in MP$(120,4)$.  Since $3a'+4b'>25$, it follows that $(a',b')$ is good, and we can get the rest of the needed leave as a combination of the possible leaves for MP$(120,4)$.

{\sc Case 2}: $3a+4b \le x+25$.  We then have $5c = n-3a-4b > 3x -(x+25) > x$, so that there are enough $5$-cycles to cover the group of size $x$.  Let $a_0,b_0$ be the least residues of $a,b$, respectively, (mod $5$).  Note that $3a_0+4b_0 \equiv n \equiv x \pmod{5}$.  It follows by
Lemma~\ref{ab-leave} that there exists an MP$(x,4)$ having $(a_0,b_0)$-leave.  The pair $(a-a_0,b-b_0)$ is good, since each component is a multiple of $5$. Hence we may fill the groups of size $120$ with MP$(120,4)$ so as to realize exactly $a-a_0$ $3$-cycles and $b-b_0$ $4$-cycles.  Taken together, we have constructed an MP$(n,4)$ with leave $aC_3 \cup b C_4 \cup cC_5$.
\end{proof}

In some cases, we can obtain good bounds in situations with other specific cycle lengths.

\begin{ex}
By Lemma~\ref{all24}, an MP$(24,4)$ exists with leave $3C_8$.  By filling a $4$-GDD of type $24^u$, we also obtain MP$(n,4)$ with leave $\frac{n}{8}C_8$ for all $n \equiv 0 \pmod{24}$, $n \ge 96$.
\end{ex}

In the next section, we show how to obtain longer cycles from shorter ones in leaves of MP$(n,4)$.  To this end, we give a result that facilitates a cycle-merging construction.

\begin{prop}
\label{bound36}
For all $n \equiv 0,3 \pmod{12}$, $n \ge 120$, any $2$-regular graph with cycle lengths in $\{3,6\}$ is the leave of some MP$(n,4)$.
\end{prop}

\begin{proof}
Write $n = 24u+x$, where $x \in X:=\{0,3,15,36\}$ and $u \ge 4$.  From Lemma~\ref{GDD24ux1}, there exists a $4$-GDD of type $24^u x^1$.
Fill groups of size 24 with packings having leaves $8C_3$, $6C_3 \cup C_6$, $4C_3 \cup 2C_6$, $2C_3 \cup 3C_6$ or $4C_6$, using Lemma~\ref{all24}.  This settles the cases $x=0,3$.  For $x=15$, we additionally fill the group of size 15 so that the leave is either $5C_3$ or $C_3 \cup 2C_6$, the latter from Lemma~\ref{small-leaves}; note here that $3C_3 \cup C_6$, a leave which does not exist on 15 points, is not needed because of the variety of leaves used on the groups of size 24.  For $x=36$, we may fill the group of size $36$ so that the leave is either $12C_3$ or $6C_6$ (Lemma~\ref{small-leaves}), chosen according to whether the desired leave has more cycles of length $3$ or $6$, respectively.
\end{proof}

\section{Merging cycles}
\label{merging}

In \cite[Lemma 3.2]{CL}, a construction was given which has the effect of joining leave cycles.  Although its purpose was to produce Hamiltonian leaves $C_n$, we can easily adapt the construction to merge shorter cycles in the leave.

Suppose we have a 4-HGDD of type $3^{g \times h}$.  Consider a group $G$ of size $3g$ and a hole $H$ of size $3h$, and put $G \cap H=\{a,b,c\}$.  If we fill $G$ with an MP$(3g,4)$
in such a way that $C=(a,b,c,d_1,\dots,d_r,a)$ is a cycle in its leave, and we similarly fill $H$ with an MP$(3h,4)$ so that $C'=(a,c,b,e_1,\dots,e_s,a)$ is a cycle in its leave, then in the resulting packing has the cycle
\begin{equation}
\label{merged-cycle}
b,c,d_1,\dots,d_r,a,e_s,\dots,e_1,b
\end{equation}
in its leave.  The length is the sum of the lengths of $C$ and $C'$ minus 3.  Note that the relative ordering of points $a,b,c$ in the input cycles $C$ and $C'$ is essential, but that such orderings can be freely chosen with appropriate embeddings of the packings into $G$ and $H$, respectively.  We also remark that the above merging can be applied to several cycles.  In a little more detail, if subsequently another group $G^*$ (or hole $H^*$) is filled so as to have a cycle $C^*$ in its leave, then $C^*$ merges similarly with the compound cycle (\ref{merged-cycle}) above if we ensure that $C^*$ runs through $G^* \cap H$ (or $G \cap H^*$) but intersects in exactly one edge.

As a special case, if two groups (or two holes) of the HGDD are filled with cycles of lengths $l_1$ and $l_2$ in their leaves, then, using a connecting $6$-cycle in the other direction, a cycle of length $l_1 + l_2$ is obtained.  An example is shown in Figure~\ref{cycle-merging}, where horizontal `dotted' cycles of lengths 6 and 9 are merged using a vertical `dashed' $C_6$.  Solid edges on the right (left) are covered by blocks in the horizontal (vertical) packing.

\begin{figure}[htbp]
\begin{minipage}{.5\linewidth}
\begin{center}
\begin{tikzpicture}[scale=0.4]

\draw[rounded corners] (-8,-0.5) rectangle (6,4.5);
\draw[rounded corners] (-8,-6.5) rectangle (6,-1.5);
\draw[rounded corners] (2.5,-9) rectangle (5.5,7);

\draw (3.02,0)--(4.02,2);
\draw (3.02,-6)--(4.02,-4);
\draw (3.98,2)--(4.98,4);

\draw[dotted] (-3,0)--(-2,2);
\draw[dotted] (3.02,0) to  [out=35,in=270] (5.02,4);
\draw[dashed] (2.98,0) to  [out=35,in=270] (4.98,4);
\draw[dotted] (-1,4)--(1,2);

\draw[dotted] (0,0)--(1,2);
\draw[dotted] (0,0) to  [out=35,in=270] (2,4);
\draw[dotted] (2,4)--(4,2);

\draw[dotted] (2.98,0)--(3.98,2);
\draw[dotted] (-3,0) to  [out=35,in=270] (-1,4);

\draw[dotted] (2.98,-6)--(3.98,-4);
\draw[dotted] (3.02,-6) to  [out=35,in=270] (5.02,-2);
\draw[dashed] (2.98,-6) to  [out=35,in=270] (4.98,-2);
\draw (3.98,-4)--(4.98,-2);

\draw[dotted] (0,-6)--(1,-4);
\draw[dotted] (0,-6) to  [out=35,in=270] (2,-2);
\draw[dotted] (2,-2)--(4,-4);

\draw[dotted] (1,-4)--(5,-2);
\draw[dotted] (-2,2)--(5,4);

\draw[dashed] (4.02,2)--(5.02,4);
\draw[dashed] (4.02,-4)--(5.02,-2);

\draw[dashed] (3,0) to [out=280,in=80] (3,-6);
\draw[dashed] (4,2) to [out=280,in=80] (4,-4);

\node at (4,6) {$\vdots$};
\node at (4,-7.5) {$\vdots$};
\node at (-5,2) {$\cdots$};
\node at (-5,-4) {$\cdots$};

\foreach \a in {0,1,2}
 \foreach \b in {0,1}
{
  \filldraw (-3+3*\a,-6+6*\b) circle [radius=.2];
  \filldraw (-2+3*\a,-4+6*\b) circle [radius=.2];
  \filldraw (-1+3*\a,-2+6*\b) circle [radius=.2];
}
\end{tikzpicture}
\caption{Cycle merging illustration}
\label{cycle-merging}
\end{center}
\end{minipage}
~
\begin{minipage}{.5\linewidth}
\begin{center}
\begin{tikzpicture}[scale=0.5]
\draw (0,0)--(0,7);
\draw (0,8)--(0,15);
\draw (1,3)--(1,4);
\draw (1,11)--(1,12);
\draw (0,7)--(1,7);
\draw (6,7)--(6,8);
\draw (7,7)--(7,8);
\draw (7,13)--(7,14);
\draw (8,1)--(8,2);
\draw (8,5)--(8,6);

\foreach \a in {2,4,6}
 \foreach \b in {0,3,4,11,12,15}
  \draw (\a,\b)--(\a+1,\b);

\foreach \a in {0,2,4}
 \foreach \b in {0,7,8,15}
  \draw (\a,\b)--(\a+1,\b);

\foreach \a in {2,4}
 \foreach \b in {3,4,7,8,11,12,15}
  \draw (\a,\b)--(\a+1,\b);

\foreach \a in {1,3,5,7}
 \foreach \b in {1,2,5,6,9}
  \draw (\a,\b)--(\a+1,\b);

\foreach \a in {1,3,5}
 \foreach \b in {10,13,14}
  \draw (\a,\b)--(\a+1,\b);

\foreach \a in {1,2,3,4,5,6,7}
 \foreach \b in {0,2,4,6,8,10,12,14}
  \draw (\a,\b)--(\a,\b+1);


\foreach \a in {0,1,...,7}
 \foreach \b in {0,1,...,15}
  \filldraw (\a,\b) circle [radius=.1];

\foreach \b in {1,2,5,6,9}
 \filldraw (8,\b) circle [radius=.1];
\end{tikzpicture}
\caption{A wiggly lattice path}
\label{wiggly-path}
\end{center}
\end{minipage}
\end{figure}

The case of MP$(n,4)$ in which leave cycles are arbitrary multiples of three is a particularly clean application of cycle merging.

\begin{thm}
\label{bound3N}
For all $n \equiv 0,3 \pmod{12}$, $n \ge 5112$, any $2$-regular graph with cycle lengths in $\{3,6,9,\dots\}$ is the leave of some MP$(n,4)$.
\end{thm}

\begin{proof}
Write $n=3(8m+r)$, where $8 \mid m$ and $r \equiv 0,1\pmod{4}$, $40 \le r \le 101$.  We have $m \ge 208 > 2r$.

We claim that there exists a path $P$ in the integer lattice which
\vspace{-10pt}
\begin{itemize}
\item
visits every vertex of $\{1,\dots,8\} \times \{1,\dots,m\}$, and also $r$ extra vertices in the ninth column,
\item
has at most two consecutive horizontal vertices, and
\item
uses vertical runs of only $2,4$, or $8$ vertices.
\end{itemize}
An example of such a path for $m=16$, $r=5$ is shown in Figure~\ref{wiggly-path}.  The example illustrates how in general $P$ can be built from $8 \times 8$ tiles and `detours' to the ninth column. It is sufficient in general to have $8 \mid m$ and $m > 2r$, which hold for our instance of the parameters.

Take an idempotent TD$(9,m)$, which exists for the stated values of $m$ as seen in Tables III.3.83 and 87 of \cite{hcd}.  Delete all but $r$ points from the last group. Without loss of generality, we may assume the resulting $8m+r$ points are naturally labelled by the lattice points of $P$.  Give every point weight three and replace all blocks except for those in one parallel class $\mathcal{C}$ by $4$-GDDs of type $3^8$ or $3^9$.  The result is a $4$-DGDD with group sizes in $\{3m,3r\}$, hole sizes in $\{24,27\}$, and such that every intersection between them has size $0$ or $3$.

Consider a partition of $n$ into summands which are multiples of three that we wish to realize as cycle lengths in the leave.  We begin by cutting up our path $P$ into a disjoint union $Q$ of paths whose lengths are one-third of the required summands.  Groups and holes of the DGDD are filled with MP$(3m,4)$, MP$(3r,4)$, MP$(24,4)$, and MP$(27,4)$, where the cycle lengths in the leaves are chosen according to (thrice) the component sizes of the subgraph of $Q$ induced by the corresponding row or column of the grid.   At each meeting of vertical and horizontal edges in $Q$, we apply a cycle merge.

Note that the conditions on $P$ ensure that only cycles of lengths in $\{3,6\}$ are needed for the holes of size $24$ or $27$, and in the group of size $3r$.  The needed packings MP$(24,4)$ and MP$(27,4)$ exist from Lemmas~\ref{all24} and \ref{small-leaves}.  The needed MP$(3r,4)$ exists in view of Proposition~\ref{bound36} and our lower bound on $r$.  Some groups of size $3m$ in our construction may demand cycle lengths in $\{6,12,24\}$, but such MP$(3m,4)$ are easily seen to exist by filling a $4$-GDD of type $24^{m/8}$ with various MP$(24,4)$ from Lemma~\ref{all24}.
\end{proof}

To obtain arbitrary $2$-regular graphs as leaves in MP$(n,4)$, it is helpful to have two lemmas that mix cycles of length $4$, $5$, and  multiples of three.

\begin{lemma}
\label{short-mult3}
Let $n \equiv 0,3 \pmod{12}$, $n > 10^6$.
Suppose $G=A \cup bC_4 \cup cC_5$, where $A$ is a union of cycles of length divisible by $3$ and $|V(A)| \le 3000$.
Then $G$ is the leave of some MP$(n,4)$.
\end{lemma}

\begin{proof}
Put $a = \frac{1}{3}|V(A)|$, so that $3a+4b+5c=n$ and $a \le 1000$.

We claim that $n = 120u+123m$ for integers $u$ and $m$ satisfying $m \ge 2000$, $m \equiv 0,1 \pmod{4}$, and $123m \le 60(u-1)$.  To see that this is possible, let $m \equiv n/3 \pmod{40}$ with $2000 \le m < 2040$.  Then, with $u=\frac{1}{120}(n-123m)$, we have $60(u-1) = \frac{1}{2}(n-123m)-60 > \frac{1}{2}(10^6-123 \times 2040)-60 >
123m$.
By Lemma~\ref{GDD24ux1}, there exists a $4$-GDD of type $120^u (123m)^1$.

Next, we claim that there exist nonnegative integers $b_0$ and $c_0$ satisfying $4b_0 + 5c_0 = 3(m-a)$, where $b \equiv b_0 \pmod{5}$ and $c \equiv c_0 \pmod{4}$.  For this, observe that $3(m-a)=n-120(m+u)-3a \equiv 4b+5c \pmod{20}$ so that some multiple of $5$ may be subtracted from $b$ and some multiple of $4$ subtracted from $c$ to get the desired $b_0,c_0$.

From a $4$-GDD of type $15^8 3^1$ (Lemma~\ref{GDD15ux1}), there exists an MP$(123,4)$ with leave $24C_5 \cup C_3$.  Similarly, from a $4$-GDD of type $24^5 3^1$ (Lemma~\ref{GDD24ux1}), there exists an MP$(123,4)$ with leave $30C_4 \cup C_3$.
And, as seen in the proof of Theorem~\ref{bound345},  there exist
MP$(120,4)$ with any possible leave having cycle lengths in $\{4,5\}$.

We begin our construction with a $4$-HGDD of type $3^{m \times 41}$ (Lemma~\ref{hgdd}). Fill holes of size $123$ with MP$(123,4)$ having leave either $24C_5 \cup C_3$ or $30C_4 \cup C_3$, where in the first $a$ holes, the unique $C_3$ is placed in the first group, and in the last $m-a$ holes the unique $C_3$ occurs in the last group.   Fill the groups with MP$(3m,4)$ having the following leaves:
\vspace{-10pt}
\begin{itemize}
\item
in the first group, leave $A \cup (m-a)C_3$, where $A$ is placed on the first $a$ holes;
\item
in the last group, leave $aC_3 \cup b_0 C_4 \cup c_0 C_5$, where $aC_3$ is placed on the first $a$ holes;
\item
in all other groups, leave $mC_3$, from a $4$-GDD of type $3^m$.
\end{itemize}
The filling strategy is shown in Figure~\ref{fig-short-mult3}.
It results in an MP$(123m,4)$ having leave $A \cup b_1 C_4 \cup c_1 C_5$, where $b_1 \equiv b \pmod{5}$ and $c_1 \equiv c \pmod{5}$.  Either $b_1=b_0$ if the leave $24C_5 \cup C_3$ is used to fill holes, or $c_1=c_0$ if the leave $30C_4 \cup C_3$ is used.  By choosing this ingredient according to which of $b$ or $c$ is larger, it is possible to ensure that both $b_1 \le b$ and $c_1 \le c$.  Finally, if we fill groups of a $4$-GDD of type $120^u (123m)^1$ with MP$(120,4)$ having cycle lengths in $\{4,5\}$ and the above MP$(123m,4)$, we may obtain the leave $A \cup b C_4 \cup c C_5$, as desired.
\end{proof}

\begin{figure}[htbp]
\begin{minipage}{.5\linewidth}
\begin{center}
\begin{tikzpicture}[scale=0.5]
\draw (0,0)--(0,14)--(8,14)--(8,0)--(0,0);
\draw[line width = 0.5mm] (0,9)--(1,9)--(1,14)--(0,14)--(0,9);
\draw (7,0)--(7,9);
\draw (0,9)--(8,9);
\foreach \a in {1,2,...,8}
 \draw (0,\a)--(7,\a);
\foreach \a in {10,11,...,13}
 \draw (1,\a)--(8,\a);
\node at (0.5,11.5) {$A$};
\node[rotate=90] at (7.5,3.5) {$b_0C_4 \cup c_0 C_5$};
\node at (3.5,0.5) {$24C_5$ or $30C_4$};
\node at (3.5,1.6) {$\vdots$};
\node at (4.5,12.6) {$\vdots$};
\node at (4.5,13.5) {$24C_5$ or $30C_4$};

\node at (4,-1) {$4$-HGDD of type $3^{m \times 41}$};
\end{tikzpicture}
\caption{$|V(A)|$ small}
\label{fig-short-mult3}
\end{center}
\end{minipage}
~
\begin{minipage}{.49\linewidth}
\begin{center}
\begin{tikzpicture}[scale=0.5]
\draw (0,0)--(0,14)--(8,14)--(8,0)--(0,0);
\draw (2,4)--(8,4);
\draw[dotted] (0,4)--(2,4);
\foreach \a in {5,6,...,14}
 \draw (0,\a)--(8,\a);
\foreach \a in {1,2,3}
 \draw[dotted] (0,\a)--(8,\a);
\node at (-1,2) {$A$};
\draw (-0.5,2)--(0,2);
\node at (4,13.5) {$b_1 C_4 \cup c_1 C_5$};
\node at (4,12.6) {$\vdots$};
\node at (4,5.5) {$b_t C_4 \cup c_t C_5$};
\node at (3.75,4.5) {$A_{t+1} \cup b_{t+1} C_4 \cup c_{t+1} C_5$};
\node at (4,3.5) {$A_{t+2}$};
\node at (4,2.6) {$\vdots$};
\node at (4,0.5) {$A_{m}$};
\draw[line width = 0.5mm] (0,0)--(8,0)--(8,4)--(2,4)--(2,5)--(0,5)--(0,0);

\node at (4,-1) {$4$-HGDD of type $3^{m \times 320}$};
\end{tikzpicture}
\caption{$|V(A)|$ large}
\label{fig-long-mult3}
\end{center}
\end{minipage}
\end{figure}

\begin{lemma}
\label{long-mult3}
Let $n \equiv 0 \pmod{3840}$, $n > 10^6$.
Suppose $G=A \cup bC_4 \cup cC_5$, where $A$ is a union of cycles of length divisible by $3$, $|V(A)| \ge 3000$, and $4b+5c \equiv 0 \pmod{60}$.
Then $G$ is the leave of some MP$(n,4)$.
\end{lemma}

\begin{proof}
Put $a = \frac{1}{3}|V(A)|$, so that $3a+4b+5c=n$ and $a \ge 1000$.
Write $n=960m$, where $4 \mid m$.  We have $m > 1000$ from our assumed lower bound on $n$.

Suppose $4b+5c = 960t+u$, where $0 \le u \le 900$ with $60 \mid u$.  Note that $t=\lfloor (n-3a)/960 \rfloor \le m-4$.
Using that $60 \mid 4b+5c$, we can, using multiples of 60, decompose $b=b_1+\dots+b_t+b_{t+1}$ and $c=c_1+\dots+c_t+c_{t+1}$,
where $4b_k+5c_k = 960$ for each $k=1,\dots,t$, and $4b_{t+1}+5c_{t+1}=u$.

We now describe a decomposition of $A$.

{\sc Case 1}:  $u=0$.  We simply `cut up' $A$ at multiples of $960$.  In more detail, suppose the cycle lengths in $A$ are $l_1,\dots,l_h$ with $l_1+\dots+l_h=960(m-t)$. Consider the partial sums $s_0:=0$, $s_j:=l_1+\dots+l_j$ for $j=1,\dots,h$.  Take the largest index $j$ with $s_j < 960$.  Put $A_{t+1}=C_{l_1} \cup \cdots \cup C_{l_{j}} \cup C_{l_{j+1}'}$, where $l_{j+1}' = 960-s_j$, and repeat on the list $l_j-l_j',l_{j+1},\dots,l_h$ to form $A_{t+2}$, continuing until the final list defines $A_m$.

{\sc Case 2}: $u>0$.  Put $a_{t+1}:=\frac{1}{3}(960-u)$ and let $A_{t+1}$ be the graph $a_{t+1}C_3$.  Now, set aside some cycles of length 3 from $A$ or reduce longer cycles in $A$ by a multiple of three, with no such cycle reduced by more than half of its original length, and with the total reduction being $3a_{t+1}$.  In some more detail, if $A = zC_3 \cup C_{l_1} \cup \dots \cup C_{l_h}$, we first reduce it to $A'=(z-a_{t+1})C_3 \cup C_{l_1} \cup \dots \cup C_{l_h}$ if $a_{t+1} \le z$, or otherwise $A'=C_{l_1'} \cup \dots \cup C_{l_h'}$, where $3 \mid l'_j$ and $l_j/2 \le l_j' \le l_j$ for each $j$, and $l_1'+\dots+l_h' =960(m-t-1)$.  Then, follow Case 1 to cut up as needed the resulting cycles so that the pieces $A_{t+2},\dots,A_m$ each have order $960$.

Take a $4$-HGDD of type $3^{m \times 320}$, and fill holes with MP$(960,4)$ having the following leaves, as illustrated in Figure~\ref{fig-long-mult3}:
\vspace{-10pt}
\begin{itemize}
\item
the $k$th hole, $k=1,\dots,t$, gets leave $b_k C_4 \cup c_k C_5$;
\item
the next hole gets leave $A_{t+1} \cup b_{t+1}C_4 \cup c_{t+1} C_5$;
\item
the remaining holes get leaves $A_{t+2},\dots,A_m$.
\end{itemize}
From the lower bound on $a$, there are at least two holes in the latter category.  To complete the construction, we fill groups with MP$(3m,4)$ having cycle lengths in $\{3,6\}$.  It remains to justify that $A$ can be reconstructed from $A_{t+1},\dots,A_m$ by merging cycles from different holes in pairs.
If it was not necessary to reduce any cycles (Case 1 or the situation $a_{t+1} \le z$ in Case 2) then the only merging needed is where cycles were cut up. That is, $A'$ can be formed by linking the last $m-t-1$ holes along a Hamilton path in the grid, with merging in (say) the first and last groups as needed.
If some cycles were reduced, say from length $l_j$ to $l'_j$, we arrange the $C_3$s in the $(t+1)$st hole so that $\frac{1}{3}(l_j-l_j')$ of them  fall into
groups which are traversed by $C_{l_j'}$.  The condition that cycles are reduced by no more than half of their lengths, and the ability to permute points within each group facilitate this alignment.  Since the $A_{t+2},\dots,A_m$ occupy at least two holes, it is possible to align each $C_3$ in $A_{t+1}$ with a reduced cycle in one of these later holes for merging.  (This may be necessary, for instance, when there is demand for a large number of cycles of length $9$.)
As before, merging may be needed in the first and last groups, and we can choose to avoid placing $A_{t+1}$ in those groups since $960-3a_{t+1} \ge 60$.
\end{proof}

\rk
This statement was given so as to roughly match Lemma~\ref{short-mult3} for later use, but in fact much better bounds on $n$ and slightly better bounds on $A$ are possible in Lemma~\ref{long-mult3} with the same methods.

We pause to mention a topic in graph theory loosely connected with our cycle merging methods.  Given a graph $G$ and a spanning sub-forest $F$ of $G$, let $\lam(F)$ denote the multiset of component sizes of $F$.
The set of possible $\lam(F)$ as $F$ varies is connected with the `forest signature table' of $G$, \cite[Section 2.1]{GHN} as well as Stanley's `chromatic symmetric function' of $G$, \cite[Theorem 2.5]{Stanley}.  For our construction of packings with arbitrary cycle lengths divisible by three, we have effectively used that grids or certain sub-graphs of grids have the property that any possible integer partition is realized by $\lam(F)$ for some sub-forest $F$.  Hamilton paths (with some convenient bending conditions) have been all we have needed, except that some caterpillars are used to link $C_3$s with reduced cycles in Case 2 of Lemma~\ref{long-mult3}.

Cycle merging is slightly more delicate when lengths are not multiples of three.
In the construction to follow, we make use of alignments of cycles of lengths 4 and 5, two or three at a time, on a small number of bundles of three vertices.  It is possible to give each cycle two edges internal to some bundle; see Figure~\ref{clusters45}.  If we identify bundles with group/hole intersections in an HGDD, this means that any such cycle can be merged with cycles in other groups.  This is used in the proof of the following result: a longer cycle of length $1 \pmod{3}$ arises from merging some such $C_4$ with a $C_{3t}$, and similarly for length $2 \pmod{3}$ using $C_5$ and $C_{3t}$.

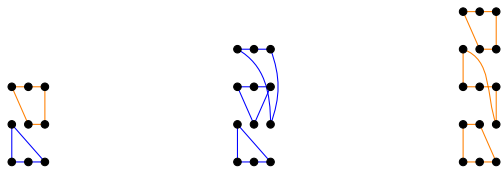
\begin{figure}[htbp]
\begin{center}
\begin{tikzpicture}[scale=1]
\Cfourfive{-3}{0}
\Cfourtrio{0}{0}
\Cfivetrio{3}{0}
\end{tikzpicture}
\end{center}
\caption{Alignment of small clusters of cycles with lengths in $\{4,5\}$}
\label{clusters45}
\end{figure}

\begin{thm}
\label{main}
For all $n \equiv 0,3 \pmod{12}$, $n > 10^7$, any $2$-regular graph of order $n$ is the leave of some packing MP$(n,4)$ of edge-disjoint $K_4$.
\end{thm}

\begin{proof}
Suppose we are given a list of integers $l_1,\dots,l_a \equiv 0\pmod{3}$, $l'_1,\dots,l'_b \equiv 1 \pmod{3}$, and $l''_1,\dots,l''_c \equiv 2 \pmod{3}$ to be realized as cycle lengths of an MP$(n,4)$.

For each length $l'_i>4$, put $p'_i=l_i'-4$.  Similarly, for each $l''_j>5$, put $p''_j=l''_j-5$.  We have $p'_i \equiv p''_j \equiv 0 \pmod{3}$ for each $i,j$.

The outline of our approach is to fill a DGDD so that its groups contain the leave $bC_4 \cup cC_5$ together with some residual cycles of length divisible by three; then, we reconstruct the desired lengths $l_h,l_i',l_j''$ by merging along holes.

Write $n=3(8m+r)$, where $m > r > 10^6/3$ and $1280 \mid m$.  As in the proof of Theorem~\ref{bound3N}, we construct a $4$-DGDD on $n$ points by giving weight three to an (idempotent) transversal design TD$(9,m)$ with one group truncated to have size $r$.  Recall that there are eight groups of size $3m$, one group of size $3r$, $r$ holes of size $27$, and $m-r$ holes of size $24$.

If $4b+5c=n$, then we are done by Proposition~\ref{bound345}.  So, assume in what follows that $4b+5c<n$.
We fill the DGDD according to two main cases.

{\sc Case 1}:  $4b+5c \ge 3r-3000$.
Choose integers $b_0,c_0$ satisfying $b \equiv b_0 \pmod{15}$, $c \equiv c_0 \pmod{12}$, $0 \le b_0 \le b$, $0 \le c_0 \le c$, $3r-3000 \le 4b_0 + 5c_0 \le 3r$, and $4(b-b_0)+5(c-c_0) < 24m$.  Write
$4(b-b_0)+5(c-c_0) = 3mt + u$, where $0 \le t \le 7$ and $0 \le u < 3m$.  Note that since the left side is divisible by $60$, we also have $60 \mid u$.  Now, using multiples of 60 (as $4 \times 15$ or $5 \times 12$), we can write $$4(b-b_0) + 5(c-c_0) = \sum_{k=1}^{t+1} 4b_k + 5c_k,$$
where $4b_k+5c_k = 3m$ for each $k=1,\dots,t$, and $4b_{t+1} + 5c_{t+1} = u$.

Next, we describe a choice of graphs $A_{t+1},\dots,A_8,A_0$ which are disjoint unions of cycles of length divisible by three.  The graph $A_{t+1}$ has $3m-u$ vertices, $A_0$ has $3r-4b_0-5c_0$ vertices, and all others (if any) have $3m$ vertices.  The specific lengths of cycles are $l_h,p_i',p_j''$, except that it may be necessary to make $8-t$ cuts to certain lengths in this list so that each graph has the correct order.

We fill groups of the DGDD as follows:
\begin{itemize}
\item
the $k$th group, $k=1,\dots,t$, gets MP$(3m,4)$ having leave $b_k C_4 \cup c_k C_5$, using Theorem~\ref{bound345};
\item
the next group gets MP$(3m,4)$ with leave $A_{t+1} \cup b_{t+1}C_4 \cup c_{t+1} C_5$, using Lemma~\ref{short-mult3} or \ref{long-mult3};
\item
the next groups up to the $8$th (if any) get MP$(3m,4)$ having leaves $A_{t+2}, \dots A_8$, using Theorem~\ref{bound3N};
\item
the last group is gets MP$(3r,4)$ having leave $A_0 \cup b_0 C_4 \cup c_0 C_5$, using Lemma~\ref{short-mult3}.
\end{itemize}
At this point, we note that the leave in each group can be placed onto the vertices of the DGDD according to any permutation.  Using this, we match each cycle of length $p_i'$ with some $C_4$ from a different group.  Choose a hole $H$ traversed by the cycle $C_{p_i'}$ and demand that its matched $C_4$ uses two edges in the same hole.
In this way, a $C_6$ inside $H$ spanning the two relevant groups facilitates a merge of the cycles and results in a cycle of length $p_i'+4=l_i'$.
Similarly, we match leave cycles $C_{p_j''}$ with $C_5$ in a different group, and set these up for merging to produce a cycle of length $p_j''+5=l_j''$.

{\sc Case 2}: $4b+5c<3r-3000$.  Fill the last group with MP$(3r,4)$ having leave $a_0C_3 \cup bC_4 \cup cC_5$, where $a_0:=r-\frac{1}{3}(4b+5c)$, which exists by Theorem~\ref{bound345}.
Now, similar to the proof of Lemma~\ref{short-mult3}, we remove occurrences of $3$ from the list $l_1,\dots,l_h$ or reduce each length in $l_h,p'_i,p''_j$ by a nonnegative multiple of three up to half of its length so that the total reduction is exactly $3a_0$.  The $2$-regular graph $A'$ with these reduced cycle lengths has exactly $n-3a_0-4b-5c=24m$ vertices and all cycle lengths a multiple of three.  We may realize a leave $A'$ in the first 8 groups of the DGDD, by cutting into multiples of $3m$ and merging (if needed) using one or more $C_6$ in MP$(24,4)$ in (say) the first and last holes.
Similar to Case 1 above, the required lengths can now be reconstructed by additional merging using $C_6$ which run through the last group.
\end{proof}

We give an example to illustrate the method further.

\begin{ex}
Consider $n=14 \times 10^6+7 \equiv 3 \pmod{12}$, and suppose the leave $C_7 \cup 10^6 C_{14}$ is desired.   We can take $m=519680$, $r=509229$ for our DGDD.  We also have $b=1$, $c=10^6$, leading us to case 1 of the proof.  With the choice $b_0=1$, $c_0=305500$, we have $4(b-b_0)+5(c-c_0) = 3mt+u$ for $t=2$ and $u=354420$.  The first two groups are filled so as to have all $C_5$ components in the leave, and the third group has $c_3=u/5=70884$ $C_5$.  The list of residual cycle lengths divisible by three is $3$, $9,\dots,9$.  The leave in the rest of the third group is $2C_3 \cup 133846C_9$, where one $C_3$ is saved for merging with $C_4$ and the other has resulted from cutting a $C_9$.  This $C_3$ can be merged with the leftover $C_6$ in the fourth group, which gets leave $C_6 \cup 173226C_9$.  Groups $5,6,7,8$ are filled similarly with the cutting dictating $C_6$ in groups 5 and 8, and $2C_3$ in groups 6 and 7.  The $C_6$ in group 8 is merged into the ninth group, which gets leave $C_4 \cup 305500C_5 \cup C_3 \cup 20C_9$.  With considerable choice, it is possible to match each $C_5$ with a $C_9$ for merging.
\end{ex}

We remark that our lower bound of $10^7$ in Theorem~\ref{main} is very crude.  Improvements should be possible with some additional work, perhaps based on a more intricate strategy for merging cycles.  Here is another example which shows that the proof method can apply in much smaller cases.

\begin{ex}
Let $n=48048$, and suppose the leave $C_{16015} \cup C_{16016} \cup C_{16017}$ is desired.  Take $m=1800$, $r=1616$ for the DGDD.  In this case, $4b+5c=9$, and we proceed as in case 2 of the proof.  Fill the ninth group so as to have leave $1613 C_3 \cup C_4 \cup C_5$, using Theorem~\ref{bound345}.  We reduce the first two desired cycle lengths by $4$ and $5$, respectively, and reduce $16017$ (a multiple of three) by $3 \times 1613 = 4839$.  We then realize the residual lengths $16011$, $16011$, $11178$, which total $24m$, in the first 8 groups, using Theorem~\ref{bound3N}, by cutting them up as $C_{5400}$, $C_{5400}$,$C_{5211} \cup C_{189}$,  $C_{5400}$,  $C_{5400}$,  $C_{5022} \cup C_{378}$,  $C_{5400}$,  $C_{5400}$ and re-joining them using the first and last holes. The cycles in the ninth group are merged with the residual lengths so as to produce the desired leave.  Note that the cycle of length $16017$ is routed through groups 6, 7, and 8, and additionally takes $1613$ detours of length three into the ninth group.
\end{ex}

We also note that a variety of specific leaves can be obtained with significantly better bounds on $n$.  Here is one such example result which makes use of Lemma~\ref{all24} and a few cycle merges.

\begin{prop}
For all $n \equiv 0,3 \pmod{12}$, $n \ge 7695$ and any integer $l$ with $3 \le l \le n/2$, the graph $C_l \cup C_{n-l}$ is the leave of some MP$(n,4)$.
\end{prop}

\begin{proof}
We first show that the result holds for $24 \mid n$, $n \ge 960$.  For this case, put $n=24m$ and write $l=l_1+l_2+\dots+l_m$ with $l_i \in \{0,3,4,\dots,12\}$ for each $i=1,\dots,m$, where furthermore at most one $l_i$ belongs to $\{3,4,5\}$.
Fill a $4$-HGDD of type $3^{m \times 8}$ so that group $i$ receives an MP$(24,4)$ having leave $C_{l_i} \cup C_{24-l_i}$.  (When $l_i=0$, this is to be interpreted as $C_{24}$.)  The holes are to be filled with MP$(3m,4)$ whose leaves have cycle lengths in $\{3,6\}$, using Proposition~\ref{bound3N}.  Cycles of length six are used to join together the cycles $C_{l_i}$ and (separately) the complementary cycles $C_{24-l_i}$. Note that, by the condition that at most one $l_i$ belongs to $\{3,4,5\}$, it is possible to merge cycles $C_{l_i}$ along an alternating sequence of two holes, so that merging cycles of length six suffice.

For the general case, write $n=3(8m+r)$ where $8 \mid m$, $m \ge 320$, and $r \equiv 0,1 \pmod{4}$, $5 \le r \le 68$.  Write $l=l_1+l_2+\dots+l_8$ with $l_i \in \{0,3,4,\dots,2m\}$.
Construct as in the proof of Theorem~\ref{bound3N} a $4$-DGDD with 8 groups of size $3m$, one group of size $3r$, $r$ holes of size $27$, and $m-r$ holes of size $24$.  Fill groups of size $3m$ with MP$(3m,4)$ having leave $C_{l_i} \cup C_{3m-l_i}$; these packings exist by the first part of the proof.  The group of size $3r$ can be filled with an MP$(3r,4)$ having leave $C_{3r}$; these exist by the main result of \cite{CL} on Hamiltonian $2$-regular leaves. 
Holes are to be filled with MP$(24,4)$ and MP$(27,4)$ having leaves with cycle lengths in $\{3,6\}$ as needed to join together the cycles $C_{l_i}$ across groups to form $C_l$ and (separately) the cycles $C_{3m-l_i}$ along with $C_{3r}$ to form $C_{n-l}$.
\end{proof}

\section{Other congruence classes}

\subsection{Nonempty bounded leaves} Suppose $n \equiv 7,10 \pmod{12}$.
Here, the leave of an MP$(n,4)$ is bounded (non-spanning) since $n \equiv 1 \pmod{3}$ and  its number of edges is $3 \pmod{6}$.
Since we are assuming $\lam=1$, the leave is a simple graph and so at least nine edges is necessary.  The unique MP$(7,4)$ has two blocks intersecting in one point. Its leave is isomorphic to $K_{3,3}$.  For larger orders, use a $4$-GDD of type $1^{n-7} 7^1$ (a design with a hole), which exists, \cite{RS}, for all $n \equiv 7,10 \pmod{12}$, $n \ge 22$.
Filling the group of size $7$ with an MP$(7,4)$  settles the existence problem for MP$(n,4)$ for these  congruence classes.

There is exactly one other graph up to isomorphism with 9 edges and all degrees a multiple of three: this is the `triangular prism' $K_2 \Box K_3$.  In the appendix, we present an MP$(31,4)$ with this leave. Then, proceeding as above, we have a bound for existence of packings with each of the two possible leaves.

\begin{prop}
\label{prism}
For all $n \equiv 7,10 \pmod{12}$, $n \ge 94$, there exists an MP$(n,4)$ with each of the possible leaves $K_{3,3}$ and $K_2 \Box K_3$.
\end{prop}

\begin{proof}
It remains to consider the leave $K_2 \Box K_3$.
Take a $4$-GDD of type $1^{n-31}31^1$, which exists from \cite{RS} for all  $n \ge 3 \times 31 + 1 = 94$.  Fill the group of size $31$ with the example packing shown in the appendix.  The resulting set of blocks gives an MP$(n,4)$ with leave $K_2 \Box K_3$.
\end{proof}

\subsection{Irregular spanning leaves} We now briefly consider $n \equiv 6,9 \pmod{12}$.
In this case, similar to our earlier work, every vertex in the leave has degree $2 \pmod{3}$.  However, the global divisibility condition forces $|E(L)| \equiv n+3 \pmod{6}$.  When coupled with the degree condition, it follows that the target leaves for MP$(n,4)$, $n \equiv 6,9 \pmod{12}$ have two possible degree sequences:
\begin{itemize}
\item
$8,2,2,\dots,2$; or
\item
$5,5,2,2,\dots,2$.
\end{itemize}
The former degree sequence is realized by four cycles identified at a common vertex (and vertex-disjoint unions with 2-regular graphs).  For the latter sequence, the two odd degree vertices must belong to the same connected component, by parity.  There are one, three, or five internally disjoint paths joining these vertices.  To summarize the cases, our leave has one component which is a subdivision of one of the four structures shown in Figure~\ref{leaves69}, and (optionally) cycles as other components.

\begin{figure}[htbp]
\begin{center}
\begin{tikzpicture}[scale=0.5]
\draw (-4,0) to[out=0,in=-45] (-3,1);
\draw (-4,0) to[out=90,in=135] (-3,1);
\draw (-4,0) to[out=90,in=45] (-5,1);
\draw (-4,0) to[out=180,in=225] (-5,1);
\draw (-4,0) to[out=180,in=135] (-5,-1);
\draw (-4,0) to[out=270,in=-45] (-5,-1);
\draw (-4,0) to[out=270,in=225] (-3,-1);
\draw (-4,0) to[out=0,in=45] (-3,-1);
\filldraw (-4,0) circle [radius=.1];

\draw (0,.5) circle [radius=.5];
\draw (0,-.5) circle [radius=.5];
\draw (2,.5) circle [radius=.5];
\draw (2,-.5) circle [radius=.5];
\draw (0,0)--(2,0);
\filldraw (0,0) circle [radius=.1];
\filldraw (2,0) circle [radius=.1];

\draw (5.5,0) circle [radius=.5];
\draw (8.5,0) circle [radius=.5];
\draw (6,0)--(8,0);
\draw (6,0) to[out=30,in=150] (8,0);
\draw (6,0) to[out=-30,in=-150] (8,0);
\filldraw (6,0) circle [radius=.1];
\filldraw (8,0) circle [radius=.1];

\draw (12,0)--(14,0);
\draw (12,0) to[out=30,in=150] (14,0);
\draw (12,0) to[out=-30,in=-150] (14,0);
\draw (12,0) to[out=60,in=120] (14,0);
\draw (12,0) to[out=-60,in=-120] (14,0);
\filldraw (12,0) circle [radius=.1];
\filldraw (14,0) circle [radius=.1];
\end{tikzpicture}
\end{center}
\caption{Possible connected leave types for MP$(n,4)$, $n \equiv 6,9 \pmod{12}$}
\label{leaves69}
\end{figure}
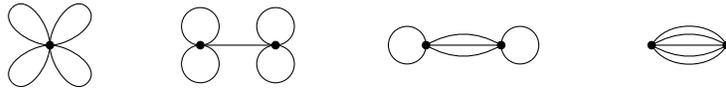
The MP$(6,4)$ consisting of a single block has leave $K_6 \setminus K_4$, which consists of five internally disjoint paths joining two vertices (those not in the block).  The path lengths are as small as possible for simple graphs, namely 1,2,2,2,2.

Filling the groups of a $4$-GDD of type $3^u 6^1$, \cite[IV 4.1]{hcd}, one obtains for $n \equiv 6,9 \pmod{12}$, $n \neq 9,18$, an MP$(n,4)$ having leave $u C_3 \cup (K_6 \setminus K_4)$.  Somewhat more generally, a variety of non-isomorphic leaves with one component equal to $K_6 \setminus K_4$ can be obtained by filling GDDs having one group of size 6 and other group sizes $0$ or $3 \pmod{12}$.  For this, our earlier constructions produce the remaining $2$-regular subgraph of the leave.  Moreover, it seems that our cycle merging technique of Section~\ref{merging} could be adapted to create longer paths and cycles in the non-regular component.  We leave it as an open problem to obtain some explicit bound for the existence of all possible leaves in this more challenging case.

\subsection{Summary}

We conclude with a summary of the status of this problem in Table~\ref{status}, which builds on \cite[Table 40.23]{hcd}.
A bold value indicates that the bound is best possible; $G$ denotes a subdivision of one of the graphs in Figure~\ref{leaves69}.

\begin{table}[htbp]
\begin{tabular}{|c|c|c|}
\hline
$n \equiv$ & possible leaves &  existence for $n \ge$\\
\hline
$1,4 \pmod{12}$ & empty & {\bf 1} \\
$7,10 \pmod{12}$ & $K_{3,3}$ or $K_2 \Box K_3$ & 94 \\
\hline
$2,8 \pmod{12}$ & $\frac{n}{2}K_2$ & {\bf 14} \\
$5,11 \pmod{12}$ & $K_{1,4} \cup \frac{n-5}{2}K_2$ & {\bf 23} \\
\hline
$0,3 \pmod{12}$ & $2$-regular & $10^7$ \\
$6,9 \pmod{12}$ & $2$-regular $\cup \; G$ & ? \\
\hline
\end{tabular}
\medskip
\caption{Bounds for MP$(n,4)$ with arbitrary leaves}
\label{status}
\end{table}

\section*{Appendix: Small examples}

We give the explicit packings MP$(n,4)$ defined on $\{0,1,\ldots,n-1\}$ for small $n$ appearing in Lemma~\ref{small-leaves} and for Proposition~\ref{prism}. When $n \equiv 0,3 \pmod{12}$, we naturally label the cycles in the leave, starting at $0$. For instance, an MP$(15,4)$ with leave $C_3 \cup 2C_6$ is presented with cycles $(0,1,2)$, $(3,4,5,6,7,8)$ and $(9,10,11,12,13,14)$ as its leave. Only `base blocks' are listed below. The set of all blocks is obtained by developing these base blocks under the action of some group $G=\langle \alpha \rangle$, where $\alpha \in \mathcal{S}_n$ is presented as a product of disjoint (permutation) cycles. Base blocks marked with a $*$ generate short orbits.

$n=15$ with leaves $3C_5$ and $C_3 \cup 2C_6$:
\begin{center}
\begin{itemize}
  \item[] $3C_5$: $\alpha=(0,1,2,3,4)(5,6,7,8,9)(10,11,12,13,14)$.
\begin{center}
\begin{tabular}{lllll}
$\{0, 2, 5, 10 \}$,& $\{ 0, 6, 9, 12 \}$,& $\{ 0, 7, 11, 14 \}$.
\end{tabular}
\end{center}
  \item[] $C_3 \cup 2C_6$: $\alpha=(0,1,2)(3,5,7)(4,6,8)(9,11,13)(10$, $12,14)$.
\begin{center}
\begin{tabular}{lllllll}
$\{ 0, 3, 5, 9 \}$,& $\{ 3, 6, 11, 14 \}$,& $\{ 0, 7, 10, 14 \}$,& $\{ 0, 6, 8, 12 \}$,& $\{ 0, 4, 11, 13 \}$.\\
\end{tabular}
\end{center}
\end{itemize}
\end{center}


$n=27$ with leaves $C_3 \cup 4C_6$, $3C_3 \cup 3C_6$, $5C_3 \cup 2C_6$, $7C_3 \cup C_6$, and $3C_4 \cup 3C_5$:
\small
\begin{center}
\begin{itemize}
\item[] $C_3 \cup 4C_6$: $\alpha=(0,1,2)(3,9,15)(4,10,16)(5,11,17)(6,12,18)(7,13,19)(8,14,20)(21,23,25)$ $(22,24,26)$.
    \begin{center}
    \begin{tabular}{lllll}
    $\{ 0, 3, 5, 9 \}$,& $\{ 3, 6, 11, 13 \}$,& $\{ 0, 4, 13, 15 \}$,& $\{ 3, 7, 12, 21 \}$,& $\{ 3, 20, 23, 26 \}$,\\ $\{ 3, 18, 22, 24 \}$,&
    $\{ 3, 14, 16, 25 \}$,& $\{ 0, 7, 19, 26 \}$,& $\{ 4, 7, 10, 24 \}$,& $\{ 5, 8, 13, 20 \}$,\\ $\{ 5, 19, 21, 23 \}$,& $\{ 4, 11, 17, 26 \}$,&
    $\{ 0, 11, 18, 21 \}$,& $\{ 0, 8, 17, 24 \}$,& $\{ 0, 6, 14, 22 \}$,\\ $\{ 0, 12, 16, 23 \}$,& $\{ 4, 6, 8, 12 \}$,& $\{ 0, 10, 20, 25 \}$.
    \end{tabular}
    \end{center}
\item[] $3C_3 \cup 3C_6$: $\alpha=(0,1,2)(3,4,5)(6,7,8)(9,15,21)(10,16,22)(11,17,23)(12,18,24)(13,19,25)$ $(14,20,26)$.
    \begin{center}
    \begin{tabular}{lllll}
    $\{ 0, 3, 9, 15 \}$,& $\{ 0, 6, 21, 24 \}$,& $\{ 3, 8, 12, 21 \}$,& $\{ 6, 13, 15, 22 \}$,& $\{ 9, 13, 17, 23 \}$,\\ $\{ 9, 19, 24, 26 \}$,&
    $\{ 9, 11, 20, 22 \}$,& $\{ 0, 12, 18, 23 \}$,& $\{ 6, 12, 16, 19 \}$,& $\{ 3, 14, 18, 26 \}$,\\ $\{ 3, 16, 22, 24 \}$,& $\{ 0, 10, 20, 25 \}$,&
    $\{ 0, 8, 22, 26 \}$,& $\{ 0, 4, 11, 16 \}$,& $\{ 0, 5, 13, 19 \}$,\\ $\{ 0, 7, 14, 17 \}$,& $\{ 3, 6, 17, 20 \}$,& $\{ 3, 7, 11, 13 \}$.
    \end{tabular}
    \end{center}
\item[] $5C_3 \cup 2C_6$: $\alpha=(0,1,2)(3,6,9)(4,7,10)(5,8,11)(12,13,14)(15,17,19)(16,18,20)(21,23,25)$ $(22,24,26)$.
\begin{center}
\begin{tabular}{llllll}
  $\{ 0, 3, 6, 10 \}$,& $\{ 0, 5, 9, 14 \}$,& $\{ 3, 13, 15, 17 \}$,& $\{ 3, 16, 19, 21 \}$,& $\{ 3, 11, 23, 26 \}$,\\ $\{ 3, 14, 18, 25 \}$,&
  $\{ 3, 20, 22, 24 \}$,& $\{ 4, 8, 13, 22 \}$,& $\{ 0, 11, 19, 22 \}$,& $\{ 5, 8, 19, 23 \}$,\\ $\{ 0, 8, 16, 18 \}$,& $\{ 4, 11, 12, 16 \}$,&
  $\{ 0, 13, 20, 26 \}$,& $\{ 4, 7, 20, 23 \}$,& $\{ 4, 14, 15, 26 \}$,\\ $\{ 0, 4, 17, 24 \}$,& $\{ 0, 7, 15, 21 \}$,& $\{ 0, 12, 23, 25 \}$.
\end{tabular}
\end{center}
\item[]
$7C_3 \cup C_6$: $\alpha=(0,1,2)(3,6,9)(4,7,10)(5,8,11)(12,15,18)(13,16,19)(14,17,20)(21,23,25)(22$, $24,26)$.
\begin{center}
\begin{tabular}{lllll}
 $\{ 0, 3, 6, 10 \}$,& $\{ 0, 5, 9, 18 \}$,& $\{ 3, 13, 15, 19 \}$,& $\{ 3, 11, 14, 21 \}$,& $\{ 3, 16, 20, 23 \}$,\\ $\{ 3, 18, 22, 25 \}$,&
  $\{ 3, 17, 24, 26 \}$,& $\{ 0, 11, 17, 20 \}$,& $\{ 4, 12, 15, 20 \}$,& $\{ 4, 7, 17, 25 \}$,\\ $\{ 0, 14, 16, 26 \}$,& $\{ 0, 7, 12, 22 \}$,&
  $\{ 5, 8, 15, 22 \}$,& $\{ 0, 15, 21, 25 \}$,& $\{ 0, 4, 19, 24 \}$,\\ $\{ 0, 8, 13, 23 \}$,& $\{ 4, 8, 16, 21 \}$,& $\{ 4, 11, 13, 22 \}$.
\end{tabular}
\end{center}
\item[] $3C_4 \cup 3C_5$: $\alpha=(0,4,8)(1,5,9)(2,6,10)(3,7,11)(12,17,22)(13,18,23)(14,19,24)(15,20,25)$ $(16,21,26)$.
\begin{center}
\begin{tabular}{lllll}
  $\{ 0, 2, 4, 9 \}$,& $\{ 2, 6, 11, 12 \}$,& $\{ 0, 6, 14, 16 \}$,& $\{ 2, 14, 17, 20 \}$,& $\{ 2, 13, 15, 19 \}$,\\ $\{ 2, 16, 21, 23 \}$,&
  $\{ 1, 10, 13, 20 \}$,& $\{ 0, 17, 19, 24 \}$,& $\{ 0, 7, 13, 18 \}$,& $\{ 0, 20, 23, 26 \}$,\\ $\{ 0, 12, 21, 25 \}$,& $\{ 0, 11, 15, 22 \}$,&
  $\{ 1, 3, 15, 25 \}$,& $\{ 1, 7, 19, 23 \}$,& $\{ 1, 12, 18, 22 \}$,\\ $\{ 1, 11, 16, 17 \}$,& $\{ 3, 7, 16, 24 \}$,& $\{ 1, 5, 14, 26 \}$.
\end{tabular}
\end{center}
\end{itemize}
\end{center}
\normalsize

$n=36$ with leaves $C_3 \cup 2C_4 \cup 5C_5$, $2C_3 \cup 6C_5$, and $6C_6$:
\small
\begin{center}
\begin{itemize}
\item[] $C_3 \cup 2C_4 \cup 5C_5$: $\alpha=(1,2)(3,5)(4,6)(7,9)(8,10)(11,16)(12,17)(13,18)(14,19)(15,20)(22,25)$ $(23,24)$ $(27,30)(28,29)(32,35)$ $(33,34)$. 
    \begin{center}
    \begin{tabular}{lllll}
    $\{ 0, 21, 26, 31 \}^*$,& $\{ 3, 5, 11, 16 \}^*$,& $\{ 4, 6, 13, 18 \}^*$,& $\{ 7, 9, 14, 19 \}^*$,& $\{ 8, 10, 15, 20 \}^*$,\\
  $\{ 12, 17, 22, 25 \}^*$,&
  $\{ 27, 30, 32, 35 \}^*$,& $\{ 0, 3, 7, 12 \}$,& $\{ 0, 4, 11, 20 \}$,& $\{ 0, 8, 13, 27 \}$,\\ $\{ 0, 14, 22, 32 \}$,& $\{ 0, 23, 28, 33 \}$,&
  $\{ 1, 4, 7, 25 \}$,& $\{ 4, 8, 16, 22 \}$,& $\{ 1, 3, 10, 22 \}$,\\ $\{ 3, 9, 25, 30 \}$,& $\{ 11, 13, 22, 29 \}$,& $\{ 13, 15, 25, 32 \}$,&
  $\{ 14, 23, 25, 26 \}$,& $\{ 15, 22, 30, 33 \}$,\\ $\{ 22, 28, 31, 34 \}$,& $\{ 4, 14, 21, 30 \}$,& $\{ 4, 17, 33, 35 \}$,& $\{ 4, 10, 23, 34 \}$,&
  $\{ 4, 9, 24, 32 \}$,\\ $\{ 4, 12, 27, 29 \}$,& $\{ 4, 15, 19, 31 \}$,& $\{ 1, 6, 26, 29 \}$,& $\{ 7, 11, 28, 32 \}$,& $\{ 7, 16, 26, 34 \}$,\\
  $\{ 7, 18, 21, 33 \}$,& $\{ 7, 15, 17, 29 \}$,& $\{ 7, 20, 24, 30 \}$,& $\{ 1, 9, 18, 31 \}$,& $\{ 1, 13, 20, 33 \}$,\\ $\{ 1, 5, 30, 34 \}$,&
  $\{ 11, 18, 24, 27 \}$,& $\{ 1, 14, 17, 27 \}$,& $\{ 8, 11, 30, 31 \}$,& $\{ 3, 8, 14, 34 \}$,\\ $\{ 11, 17, 19, 34 \}$,& $\{ 3, 13, 19, 28 \}$,&
  $\{ 1, 16, 19, 32 \}$,& $\{ 1, 11, 21, 23 \}$,& $\{ 8, 19, 23, 29 \}$,\\ $\{ 8, 12, 18, 26 \}$,& $\{ 1, 8, 28, 35 \}$,& $\{ 1, 12, 15, 24 \}$,&
  $\{ 3, 17, 24, 31 \}$,& $\{ 8, 17, 21, 32 \}$,\\ $\{ 3, 15, 26, 35 \}$,& $\{ 3, 18, 23, 32 \}$,& $\{ 3, 20, 21, 29 \}$.
    \end{tabular}
    \end{center}
\item[] $2C_3 \cup 6C_5$: $\alpha=(0,1,2)(3,4,5)(6,11,16)(7,12,17)(8,13,18)(9,14,19)(10,15,20)(21,26,31)$ $(22,27,32)$ $(23,28,33)(24$, $29,34)(25,30,35)$.
\begin{center}
\begin{tabular}{lllll}
 $\{ 0, 3, 6, 11 \}$,& $\{ 0, 8, 10, 16 \}$,& $\{ 3, 7, 13, 16 \}$,& $\{ 6, 8, 14, 19 \}$,& $\{ 6, 9, 17, 21 \}$,\\ $\{ 6, 20, 23, 25 \}$,&
  $\{ 6, 22, 24, 28 \}$,& $\{ 6, 30, 31, 33 \}$,& $\{ 6, 27, 32, 35 \}$,& $\{ 6, 26, 29, 34 \}$,\\ $\{ 0, 4, 13, 30 \}$,& $\{ 3, 18, 21, 30 \}$,&
  $\{ 0, 18, 25, 34 \}$,& $\{ 9, 15, 25, 30 \}$,& $\{ 3, 17, 28, 35 \}$,\\ $\{ 0, 12, 22, 35 \}$,& $\{ 7, 9, 24, 35 \}$,& $\{ 7, 18, 28, 31 \}$,&
  $\{ 8, 13, 22, 34 \}$,& $\{ 8, 15, 23, 28 \}$,\\ $\{ 8, 20, 27, 31 \}$,& $\{ 0, 7, 19, 23 \}$,& $\{ 0, 5, 28, 32 \}$,& $\{ 0, 14, 21, 27 \}$,&
  $\{ 3, 9, 23, 29 \}$,\\ $\{ 3, 14, 24, 32 \}$,& $\{ 0, 9, 26, 33 \}$,& $\{ 3, 15, 19, 27 \}$,& $\{ 3, 10, 26, 31 \}$,& $\{ 3, 12, 20, 34 \}$,\\
  $\{ 7, 10, 12, 27 \}$,& $\{ 0, 15, 20, 24 \}$,& $\{ 0, 17, 29, 31 \}$.
\end{tabular}
\end{center}
\item[] $6C_6$: $\alpha=(0,1,2,3,4,5)(6,7,8,9,10,11)(12,13,14,15,16,17)(18,19,20,21,22,23)(24,25,26$, $27,28,29)(30,31,32,33,34,35)$. 
    \begin{center}
    \begin{tabular}{lllll}
    $\{ 0, 3, 6, 9 \}^*$,& $\{ 12, 15, 18, 21 \}^*$,& $\{ 24, 27, 30, 33 \}^*$,& $\{ 0, 2, 7, 12 \}$,& $\{ 0, 8, 20, 24 \}$,\\ $\{ 0, 13, 15, 26 \}$,&
    $\{ 0, 21, 28, 30 \}$,& $\{ 0, 23, 25, 29 \}$,& $\{ 0, 17, 19, 33 \}$,& $\{ 0, 18, 22, 34 \}$,\\ $\{ 0, 10, 27, 32 \}$,& $\{ 0, 14, 31, 35 \}$,&
    $\{ 6, 14, 26, 33 \}$,& $\{ 6, 16, 21, 24 \}$,& $\{ 6, 12, 22, 27 \}$,\\ $\{ 6, 8, 19, 32 \}$,& $\{ 6, 13, 20, 31 \}$,& $\{ 6, 15, 25, 35 \}$.
    \end{tabular}
    \end{center}
\end{itemize}
\end{center}
\normalsize

$n=39$ with leave $C_3 \cup 9C_4$: \small $\alpha=(0,1,2)(3,7,11,15,19,23,27,31,35)(4,8,12,16,20,24,28,32,36)(5$, $9,13,17,21,25,29,33,37)(6,10,14,18,22,26,30,34,38)$. 
\begin{center}
\begin{tabular}{lllll}
 $\{ 0, 3, 15, 27 \}^*$,& $\{ 0, 12, 24, 36 \}^*$, &  $\{ 0, 14, 26, 38 \}^*$, & $\{ 0, 4, 7, 11 \}$,& $\{ 0, 5, 8, 21 \}$, \\ 
$\{ 0, 6, 10, 13 \}$,& $\{ 3, 8, 10, 28 \}$,& $\{ 4, 12, 21, 34 \}$,& $\{ 4, 14, 25, 30 \}$,& $\{ 3, 5, 12, 16 \}$,\\ $\{ 3, 24, 29, 38 \}$,& $\{ 3, 18, 20, 26 \}$,&  $\{ 3, 9, 11, 30 \}$,& $\{ 3, 14, 19, 33 \}$,& $\{ 3, 13, 21, 25 \}$.
\end{tabular}
\end{center}
\normalsize

$n=48$ with leave $C_3 \cup 9C_5$: \small $\alpha=(0,1,2)(3,8,13,4,9,14,5,10,15,6,11,16,7,12,17)(18,23,28,19,24$, $29,20,25,30,21,26,31,22,27,32)(33,38,43,34,39,44,35,40,45,36,41,46,37,42,47)$.
\begin{center}
\begin{tabular}{llllll}
  $\{ 0, 3, 8, 14 \}$,& $\{ 0, 18, 23, 29 \}$,& $\{ 0, 33, 38, 44 \}$,& $\{ 3, 5, 18, 33 \}$,& $\{ 3, 10, 32, 34 \}$,& $\{ 3, 12, 23, 37 \}$,\\
  $\{ 3, 30, 35, 42 \}$,& $\{ 3, 29, 31, 43 \}$,& $\{ 3, 20, 27, 45 \}$,& $\{ 3, 28, 38, 41 \}$,& $\{ 3, 24, 40, 44 \}$,& $\{ 3, 22, 26, 39 \}$.
\end{tabular}
\end{center}
\normalsize

$n = 31$ with leave $K_3 \Box K_2$ consisting of edges $\{0,1\}$, $\{0,2\}$, $\{1,2\}$, $\{3,4\}$, $\{3,5\}$, $\{4,5\}$, $\{0,3\}$, $\{1,4\}$ and $\{2,5\}$: 
\small
$\alpha=(0,1,2)(3,4,5)(7,8,9)(10,11,12)(13,14,15)(16,17,18)(19,20,21)(22,23,24)(25,26$, $27)(28,29,30)$.   
\begin{center}
\begin{tabular}{lllll}
$\{ 6, 10, 11, 12 \}^*$,&
$\{ 6, 13, 14, 15 \}^*$,&
$\{ 6, 16, 17, 18 \}^*$,&
$\{ 6, 19, 20, 21 \}^*$,&
$\{ 6, 22, 23, 24 \}^*$,\\
$\{ 6, 25, 26, 27 \}^*$,&
$\{ 6, 28, 29, 30 \}^*$,&
$\{ 0, 4, 6, 7 \}$,&
$\{ 0, 5, 9, 10 \}$,&
$\{ 0, 12, 13, 16 \}$,\\
$\{ 0, 8, 15, 19 \}$,&
$\{ 0, 14, 18, 22 \}$,&
$\{ 0, 20, 24, 25 \}$,&
$\{ 0, 17, 26, 28 \}$,&
$\{ 0, 21, 27, 30 \}$,\\
$\{ 0, 11, 23, 29 \}$,&
$\{ 7, 12, 15, 26 \}$,&
$\{ 10, 15, 20, 30 \}$,&
$\{ 3, 14, 23, 30 \}$,&
$\{ 3, 13, 19, 26 \}$,\\
$\{ 3, 15, 17, 29 \}$,&
$\{ 7, 13, 23, 25 \}$,&
$\{ 3, 12, 27, 28 \}$,&
$\{ 10, 18, 23, 26 \}$,&
$\{ 3, 8, 18, 25 \}$,\\
$\{ 3, 16, 20, 22 \}$,&
$\{ 3, 10, 21, 24 \}$,&
$\{ 7, 8, 22, 30 \}$,&
$\{ 7, 10, 16, 19 \}$,&
$\{ 7, 18, 20, 28 \}$.
\end{tabular}
\end{center}

\end{document}


\title[Leaves for maximum $2$-$(24,4,1)$ packings]{\large Leaves for maximum $2$-$(24,4,1)$ packings}

\author{Yanxun~Chang}
\address{\rm Yanxun~Chang:
Mathematics,
Beijing Jiaotong University,
Beijing,
P.R.~China
}
\email{yxchang@bjtu.edu.cn}

\author{Peter J.~Dukes}
\address{\rm Peter J.~ Dukes:
Mathematics and Statistics,
University of Victoria, Victoria, Canada
}
\email{dukes@uvic.ca}

\author{Tao Feng}
\address{\rm Tao Feng:
Mathematics,
Beijing Jiaotong University,
Beijing,
P.R.~China
}
\email{tfeng@bjtu.edu.cn}

\thanks{Research of Yanxun Chang is supported by NSFC grant 11431003; research of Peter Dukes is supported by NSERC grant 312595--2017; research of Tao Feng is supported by NSFC grant 11471032; research of this paper was also partially supported by 111 Project of China, grant number B16002.}


\maketitle


This is a supplement for Lemma 3.1 in the paper {\it Leaves for packings with block size four}.

\begin{lemma}\label{all24}
Any possible $2$-regular graph on $24$ vertices is the leave of some MP$(24,4)$.
\end{lemma}

\proof The leave $L$ of an MP$(24,4)$ consists of disjoint cycles covering all $24$ vertices. If $L$ contains $u_i$ cycles of size $g_i$, $1\leq i\leq r$, then $L$ is of {\em type} $g_1^{u_1}g_2^{u_2}\cdots g_r^{u_r}$. If $L$ is of type $24^1$, then the packing is a BSEC$(24,4,1)$, \cite{CL}. If $L$ is of type $3^8$, then the packing is a GDD$(3^8,4)$, \cite{hcd}. If $L$ is of type $4^6$, then it is equivalent to a matching divisible design GDD($M_2^6,4)$, \cite{DFL}.

There are 110 ways to partition  24 into integers greater than 2. Given a leave $L$ of any of the 107 other types (apart from $24^1$, $3^8$, $4^6$ considered above), we explicitly construct an MP$(24,4)$ with leave $L$ according to the following conventions. The vertex set is taken to be $\{0,1,\ldots,23\}$. We label the cycles in $L$ naturally. For instance, if $L$ is of type $3^6 6^1$, then it consists of cycles $(3i,3i+1,3i+2)$, $0\leq i\leq 5$, and $(18,19,20,21,22,23)$. Only `base blocks' are listed below. The set of all blocks is obtained by developing these base blocks under the action of some group $G=\langle \alpha \rangle$, where $\alpha \in \mathcal{S}_n$ is presented as a product of disjoint (permutation) cycles. Base blocks marked with a $*$ generate short orbits.

$\cdot$ Type $3^6 6^1$: $\alpha=(0, 3, 6)(1, 4, 7)(2, 5, 8)(9, 12, 15)(10, 13, 16)(11, 14, 17)(18, 20, 22)(19, 21, 23)$.
\begin{center}
\begin{tabular}{llllll}
$\{ 0, 3, 7, 9 \}$,& $\{ 1, 9, 15, 18 \}$,& $\{ 2, 5, 9, 22 \}$,& $\{ 0, 10, 12, 21 \}$,& $\{ 9, 13, 17, 21 \}$& $\{ 2, 12, 17, 19 \}$,\\
$\{ 0, 5, 19, 23 \}$,& $\{ 1, 4, 14, 21 \}$,& $\{ 1, 10, 20, 23 \}$,& $\{ 1, 5, 13, 16 \}$& $\{ 1, 8, 17, 22 \}$,& $\{ 0, 8, 11, 13 \}$,\\
$\{ 0, 14, 17, 18 \}$,& $\{ 0, 16, 20, 22 \}$.
\end{tabular}
\end{center}

$\cdot$ Type $3^5 9^1$: $\alpha=(1)$.
\begin{center}
\begin{tabular}{llllll}
$\{ 0, 3, 6, 9 \}$,& $\{ 0, 4, 7, 12 \}$,& $\{ 0, 5, 8, 15 \}$,& $\{ 0, 10, 13, 17 \}$,& $\{ 0, 11, 14, 19 \}$,& $\{ 0, 16, 20, 22 \}$,\\
$\{ 0, 18, 21, 23 \}$,& $\{ 1, 9, 12, 15 \}$,& $\{ 2, 9, 14, 22 \}$,& $\{ 4, 9, 18, 20 \}$,& $\{ 8, 9, 16, 21 \}$,& $\{ 7, 9, 13, 23 \}$,\\
$\{ 5, 9, 17, 19 \}$,& $\{ 2, 10, 12, 21 \}$,& $\{ 3, 12, 16, 19 \}$,& $\{ 2, 7, 15, 19 \}$,& $\{ 6, 13, 19, 21 \}$,& $\{ 1, 5, 14, 21 \}$,\\
$\{ 4, 11, 15, 21 \}$,& $\{ 3, 7, 17, 21 \}$,& $\{ 7, 14, 16, 18 \}$,& $\{ 1, 7, 10, 20 \}$,& $\{ 5, 7, 11, 22 \}$,& $\{ 2, 5, 13, 18 \}$,\\
$\{ 3, 13, 15, 22 \}$,& $\{ 8, 11, 13, 20 \}$,& $\{ 1, 4, 13, 16 \}$,& $\{ 6, 10, 15, 18 \}$,& $\{ 5, 10, 16, 23 \}$,& $\{ 5, 6, 12, 20 \}$,\\
$\{ 11, 12, 17, 23 \}$,& $\{ 1, 3, 11, 18 \}$,& $\{ 1, 6, 17, 22 \}$,& $\{ 1, 8, 19, 23 \}$,& $\{ 2, 3, 20, 23 \}$,& $\{ 4, 6, 14, 23 \}$,\\
$\{ 3, 8, 10, 14 \}$,& $\{ 8, 12, 18, 22 \}$,& $\{ 2, 4, 8, 17 \}$,& $\{ 2, 6, 11, 16 \}$,& $\{ 4, 10, 19, 22 \}$,& $\{ 14, 15, 17, 20 \}$.
\end{tabular}
\end{center}

$\cdot$ Type $3^5 4^1 5^1$: $\alpha=(1)$.
\begin{center}
\begin{tabular}{llllll}
$\{ 0, 3, 6, 9 \}$,& $\{ 0, 4, 7, 12 \}$,& $\{ 0, 5, 8, 15 \}$,& $\{ 0, 10, 13, 17 \}$,& $\{ 0, 11, 14, 19 \}$,& $\{ 0, 16, 21, 23 \}$,\\
$\{ 0, 18, 20, 22 \}$,& $\{ 1, 9, 12, 15 \}$,& $\{ 2, 9, 14, 21 \}$,& $\{ 5, 9, 17, 20 \}$,& $\{ 4, 9, 18, 23 \}$,& $\{ 7, 9, 13, 22 \}$,\\
$\{ 8, 9, 16, 19 \}$,& $\{ 1, 7, 14, 18 \}$,& $\{ 5, 6, 14, 22 \}$,&$\{ 4, 10, 14, 16 \}$,& $\{ 8, 14, 17, 23 \}$,& $\{ 3, 14, 15, 20 \}$,\\
$\{ 7, 11, 16, 20 \}$,& $\{ 1, 5, 13, 16 \}$,& $\{ 6, 12, 16, 18 \}$,& $\{ 2, 3, 16, 22 \}$,& $\{ 1, 8, 10, 22 \}$,& $\{ 1, 6, 20, 23 \}$,\\
$\{ 2, 10, 12, 20 \}$,&$\{ 4, 8, 13, 20 \}$,& $\{ 2, 8, 11, 18 \}$,& $\{ 3, 8, 12, 21 \}$,& $\{ 5, 11, 12, 23 \}$,& $\{ 2, 5, 7, 19 \}$,\\$\{ 2, 13, 15, 23 \}$,& $\{ 2, 4, 6, 17 \}$,& $\{ 3, 7, 10, 23 \}$,& $\{ 3, 13, 18, 19 \}$,& $\{ 1, 3, 11, 17 \}$,& $\{ 1, 4, 19, 21 \}$,\\
$\{ 4, 11, 15, 22 \}$,& $\{ 6, 11, 13, 21 \}$,& $\{ 6, 10, 15, 19 \}$,& $\{ 5, 10, 18, 21 \}$,& $\{ 7, 15, 17, 21 \}$,& $\{ 12, 17, 19, 22 \}$.
\end{tabular}
\end{center}

$\cdot$ Type $3^4 12^1$: $\alpha=(0,3,6,9)(1,4,7,10)(2,5,8$, $11)(12,15,18,21)(13,16,19,22)(14,17,20,23)$.

\begin{center}
\begin{tabular}{llllll}
$\{ 0, 3, 6, 9 \}^*$, & $\{ 12, 14, 18, 20 \}^*$, & $\{ 1, 5, 7, 11 \}^*$, & $\{ 13, 16, 19, 22 \}^*$, & $\{ 0, 4, 12, 15 \}$, & $\{ 1, 8, 18, 23 \}$, \\
$\{ 0, 7, 19, 21 \}$, & $\{ 1, 4, 17, 22 \}$, & $\{ 0, 10, 13, 17 \}$, & $\{ 0, 5, 18, 22 \}$, & $\{ 0, 8, 14, 16 \}$, & $\{ 0, 11, 20, 23 \}$, \\
$\{ 2, 5, 16, 21 \}$.
\end{tabular}
\end{center}

$\cdot$ Type $3^4 6^2$: $\alpha=(1)$.
\begin{center}
\begin{tabular}{llllll}
$\{ 0, 3, 6, 9 \}$, & $\{ 0, 4, 7, 12 \}$, & $\{ 0, 5, 8, 14 \}$, & $\{ 0, 10, 13, 16 \}$, & $\{ 0, 11, 15, 18 \}$, & $\{ 0, 17, 20, 22 \}$, \\
$\{ 0, 19, 21, 23 \}$, & $\{ 1, 3, 11, 21 \}$, & $\{ 8, 11, 12, 20 \}$, & $\{ 3, 12, 14, 18 \}$, & $\{ 2, 3, 16, 20 \}$, & $\{ 3, 7, 10, 22 \}$, \\
$\{ 3, 13, 15, 23 \}$, & $\{ 3, 8, 17, 19 \}$, & $\{ 4, 8, 16, 23 \}$, & $\{ 8, 10, 18, 21 \}$, & $\{ 1, 8, 9, 13 \}$, & $\{ 2, 8, 15, 22 \}$, \\
$\{ 7, 13, 18, 20 \}$, & $\{ 9, 14, 20, 23 \}$, & $\{ 1, 4, 10, 20 \}$, & $\{ 5, 6, 15, 20 \}$, & $\{ 10, 12, 15, 19 \}$, & $\{ 1, 7, 15, 17 \}$, \\
$\{ 4, 9, 15, 21 \}$, & $\{ 1, 6, 12, 23 \}$, & $\{ 2, 5, 12, 21 \}$, & $\{ 5, 10, 17, 23 \}$, & $\{ 2, 9, 17, 18 \}$, & $\{ 4, 6, 18, 22 \}$, \\
$\{ 1, 5, 16, 18 \}$, & $\{ 1, 14, 19, 22 \}$, & $\{ 5, 7, 9, 19 \}$, & $\{ 5, 11, 13, 22 \}$, & $\{ 6, 13, 17, 21 \}$, & $\{ 2, 4, 13, 19 \}$, \\
$\{ 4, 11, 14, 17 \}$, & $\{ 6, 11, 16, 19 \}$, & $\{ 2, 6, 10, 14 \}$, & $\{ 2, 7, 11, 23 \}$, & $\{ 7, 14, 16, 21 \}$, & $\{ 9, 12, 16, 22 \}$.
\end{tabular}
\end{center}

$\cdot$ Type $3^4 5^1 7^1$: $\alpha=(1)$.
\begin{center}
\begin{tabular}{llllll}
$\{ 0, 3, 6, 9 \}$, & $\{ 0, 4, 7, 12 \}$, & $\{ 0, 5, 8, 14 \}$, & $\{ 0, 10, 13, 17 \}$, & $\{ 0, 11, 15, 19 \}$, & $\{ 0, 16, 20, 22 \}$, \\
$\{ 0, 18, 21, 23 \}$, & $\{ 1, 3, 15, 21 \}$, & $\{ 4, 8, 10, 21 \}$, & $\{ 5, 11, 12, 21 \}$, & $\{ 7, 13, 19, 21 \}$, & $\{ 6, 14, 17, 21 \}$, \\
$\{ 2, 9, 16, 21 \}$, & $\{ 4, 16, 17, 19 \}$, & $\{ 3, 8, 16, 23 \}$, & $\{ 1, 8, 9, 19 \}$, & $\{ 3, 11, 13, 20 \}$, & $\{ 2, 3, 12, 17 \}$, \\
$\{ 3, 10, 14, 19 \}$, & $\{ 3, 7, 18, 22 \}$, & $\{ 6, 12, 19, 22 \}$, & $\{ 2, 5, 19, 23 \}$, & $\{ 8, 11, 17, 22 \}$, & $\{ 2, 8, 13, 18 \}$, \\
$\{ 8, 12, 15, 20 \}$, & $\{ 9, 12, 14, 23 \}$, & $\{ 1, 10, 12, 18 \}$, & $\{ 5, 7, 10, 16 \}$, & $\{ 2, 10, 15, 22 \}$, & $\{ 6, 10, 20, 23 \}$, \\
$\{ 2, 4, 6, 11 \}$, & $\{ 2, 7, 14, 20 \}$, & $\{ 1, 7, 11, 23 \}$, & $\{ 1, 4, 14, 22 \}$, & $\{ 5, 9, 13, 22 \}$, & $\{ 7, 9, 15, 17 \}$, \\
$\{ 4, 9, 18, 20 \}$, & $\{ 1, 5, 17, 20 \}$, & $\{ 1, 6, 13, 16 \}$, & $\{ 4, 13, 15, 23 \}$, & $\{ 5, 6, 15, 18 \}$, & $\{ 11, 14, 16, 18 \}$.
\end{tabular}
\end{center}

$\cdot$ Type $3^4 4^1 8^1$: $\alpha=(0,3,6,9)(1,4,7,10)(2,5$, $8,11)(12,13,14,15)(16,18,20,22)(17,19,21,23)$.
\begin{center}
\begin{tabular}{llllll}
$\{ 0, 6, 12, 14 \}^*$,& $\{ 1, 7, 16, 20 \}^*$,& $\{ 2, 8, 17, 21 \}^*$,& $\{ 0, 3, 7, 19 \}$,& $\{ 0, 5, 13, 21 \}$,& $\{ 0, 10, 20, 23 \}$,\\
$\{ 2, 15, 18, 23 \}$,& $\{ 0, 8, 11, 18 \}$,& $\{ 0, 15, 16, 22 \}$,& $\{ 1, 4, 11, 13 \}$,\\ $\{ 1, 15, 19, 21 \}$,& $\{ 1, 5, 14, 18 \}$.
\end{tabular}
\end{center}

$\cdot$ Type $3^4 4^3$: $\alpha=(0,3,6)(1,4,7)(2,5,8)(12$, $16,20)(13,17,21)(14,18,22)(15,19,23)$.
\begin{center}
\begin{tabular}{llllll}
$\{ 0, 3, 6, 9 \}^*$,&
$\{ 1, 4, 7, 10 \}^*$,&
$\{ 2, 5, 8, 11 \}^*$,&
$\{ 1, 5, 9, 12 \}$,&
$\{ 9, 13, 15, 18 \}$,&
$\{ 0, 4, 13, 17 \}$,\\
$\{ 1, 11, 15, 17 \}$,&
$\{ 1, 16, 18, 23 \}$,&
$\{ 1, 8, 14, 22 \}$,&
$\{ 0, 7, 15, 16 \}$,&
$\{ 0, 10, 12, 22 \}$,&
$\{ 0, 8, 18, 21 \}$,\\
$\{ 0, 5, 19, 23 \}$,&
$\{ 0, 11, 14, 20 \}$,&
$\{ 2, 12, 16, 21 \}$,&
$\{ 2, 10, 17, 23 \}$.
\end{tabular}
\end{center}

$\cdot$ Type $3^3 15^1$: $\alpha=(0,3,6)(1,4,7)(2,5,8)(9,14,19)(10,15,20)(11,16,21)(12,17,22)(13,18,23)$.

\begin{center}
\begin{tabular}{llllll}
$\{ 0, 3, 7, 9 \}$,& $\{ 1, 9, 12, 19 \}$,& $\{ 0, 11, 14, 16 \}$,& $\{ 2, 9, 15, 18 \}$,& $\{ 2, 11, 19, 23 \}$,& $\{ 2, 10, 12, 14 \}$,\\
$\{ 0, 5, 8, 12 \}$,& $\{ 1, 5, 10, 21 \}$,& $\{ 1, 8, 16, 23 \}$,& $\{ 0, 15, 21, 23 \}$,\\ $\{ 1, 4, 18, 20 \}$,& $\{ 0, 13, 18, 22 \}$,\\
$\{ 0, 10, 17, 20 \}$,& $\{ 1, 11, 17, 22 \}$.
\end{tabular}
\end{center}

$\cdot$ Type $3^3 7^1 8^1$: $\alpha=(1)$.

\begin{center}
\begin{tabular}{llllll}
$\{ 0, 3, 6, 9 \}$,& $\{ 0, 4, 7, 11 \}$,& $\{ 0, 5, 8, 13 \}$,& $\{ 0, 10, 12, 16 \}$,& $\{ 0, 14, 18, 20 \}$,& $\{ 0, 15, 17, 22 \}$,\\
$\{ 0, 19, 21, 23 \}$,& $\{ 1, 3, 15, 19 \}$,& $\{ 2, 3, 12, 20 \}$,& $\{ 3, 11, 18, 22 \}$,& $\{ 3, 8, 14, 16 \}$,& $\{ 3, 7, 17, 21 \}$,\\
$\{ 3, 10, 13, 23 \}$,& $\{ 2, 4, 8, 19 \}$,& $\{ 1, 8, 12, 22 \}$,& $\{ 8, 10, 17, 20 \}$,& $\{ 8, 9, 18, 23 \}$,& $\{ 8, 11, 15, 21 \}$,\\
$\{ 10, 14, 19, 22 \}$,& $\{ 7, 9, 12, 19 \}$,& $\{ 5, 11, 16, 19 \}$,& $\{ 6, 13, 17, 19 \}$,& $\{ 9, 11, 13, 20 \}$,& $\{ 4, 9, 16, 22 \}$,\\
$\{ 2, 7, 13, 22 \}$,& $\{ 5, 6, 20, 22 \}$,& $\{ 1, 4, 20, 23 \}$,& $\{ 7, 15, 16, 20 \}$,& $\{ 5, 7, 14, 23 \}$,& $\{ 1, 7, 10, 18 \}$,\\
$\{ 1, 13, 16, 21 \}$,& $\{ 1, 5, 9, 17 \}$,& $\{ 1, 6, 11, 14 \}$,& $\{ 2, 11, 17, 23 \}$,& $\{ 6, 12, 15, 23 \}$,& $\{ 2, 5, 10, 15 \}$,\\
$\{ 2, 9, 14, 21 \}$,& $\{ 2, 6, 16, 18 \}$,& $\{ 4, 6, 10, 21 \}$,& $\{ 5, 12, 18, 21 \}$,& $\{ 4, 12, 14, 17 \}$,& $\{ 4, 13, 15, 18 \}$.
\end{tabular}
\end{center}

$\cdot$ Type $3^3 6^1 9^1$: $\alpha=(0,3,6)(1,4,7)(2,5,8)(9,11$, $13)(10,12,14)(15,18,21)(16,19,22)(17,20,23)$.
\begin{center}
\begin{tabular}{llllll}
$\{ 0, 3, 7, 9 \}$,& $\{ 1, 9, 13, 15 \}$,& $\{ 0, 11, 16, 20 \}$,& $\{ 2, 13, 17, 20 \}$,& $\{ 0, 12, 18, 23 \}$,& $\{ 0, 5, 10, 17 \}$,\\
$\{ 0, 15, 19, 21 \}$,& $\{ 0, 8, 14, 22 \}$,& $\{ 1, 5, 8, 21 \}$,& $\{ 2, 9, 12, 21 \}$,& $\{ 1, 14, 18, 20 \}$,& $\{ 1, 10, 12, 22 \}$,\\
$\{ 1, 4, 17, 19 \}$,& $\{ 2, 11, 19, 22 \}$.
\end{tabular}
\end{center}

$\cdot$ Type $3^3 5^1 10^1$: $\alpha=(1)$.

\begin{center}
\begin{tabular}{llllll}
$\{ 0, 3, 6, 9 \}$,& $\{ 0, 4, 7, 11 \}$,& $\{ 0, 5, 8, 14 \}$,& $\{ 0, 10, 12, 16 \}$,& $\{ 0, 13, 15, 18 \}$,& $\{ 0, 17, 20, 22 \}$,\\
$\{ 0, 19, 21, 23 \}$,& $\{ 1, 3, 7, 12 \}$,& $\{ 2, 3, 15, 20 \}$,& $\{ 3, 8, 13, 23 \}$,& $\{ 3, 11, 17, 19 \}$,& $\{ 3, 10, 14, 21 \}$,\\
$\{ 3, 16, 18, 22 \}$,& $\{ 7, 10, 15, 17 \}$,& $\{ 2, 7, 13, 21 \}$,& $\{ 7, 14, 19, 22 \}$,& $\{ 7, 9, 16, 23 \}$,& $\{ 5, 7, 18, 20 \}$,\\
$\{ 4, 13, 14, 17 \}$,& $\{ 8, 12, 17, 21 \}$,& $\{ 1, 5, 9, 17 \}$,& $\{ 2, 6, 17, 23 \}$,& $\{ 5, 11, 15, 23 \}$,& $\{ 5, 6, 16, 21 \}$,\\
$\{ 2, 5, 12, 19 \}$,& $\{ 5, 10, 13, 22 \}$,& $\{ 9, 11, 18, 21 \}$,& $\{ 9, 12, 14, 20 \}$,& $\{ 1, 6, 14, 18 \}$,& $\{ 1, 4, 15, 21 \}$,\\
$\{ 1, 10, 20, 23 \}$,& $\{ 4, 6, 10, 19 \}$,& $\{ 4, 12, 18, 23 \}$,& $\{ 2, 8, 10, 18 \}$,& $\{ 2, 11, 14, 16 \}$,& $\{ 4, 8, 16, 20 \}$,\\
$\{ 2, 4, 9, 22 \}$,& $\{ 6, 12, 15, 22 \}$,& $\{ 1, 8, 11, 22 \}$,& $\{ 1, 13, 16, 19 \}$,& $\{ 6, 11, 13, 20 \}$,& $\{ 8, 9, 15, 19 \}$.
\end{tabular}
\end{center}

$\cdot$ Type $3^3 4^1 11^1$: $\alpha=(1)$.

\begin{center}
\begin{tabular}{llllll}
$\{ 0, 3, 6, 9 \}$,&  $\{ 0, 4, 7, 11 \}$,&  $\{ 0, 5, 8, 13 \}$,&  $\{ 0, 10, 12, 15 \}$,&  $\{ 0, 14, 17, 19 \}$,&  $\{ 0, 16, 20, 22 \}$,\\
$\{ 0, 18, 21, 23 \}$,&  $\{ 1, 3, 14, 21 \}$,&  $\{ 4, 8, 10, 21 \}$,&  $\{ 2, 12, 16, 21 \}$,&  $\{ 5, 6, 19, 21 \}$,&  $\{ 7, 9, 17, 21 \}$,\\
  $\{ 11, 13, 15, 21 \}$,&  $\{ 3, 11, 16, 19 \}$,&  $\{ 8, 14, 16, 23 \}$,&  $\{ 1, 5, 7, 16 \}$,&  $\{ 4, 9, 16, 18 \}$,&  $\{ 6, 10, 13, 16 \}$,\\
  $\{ 1, 9, 13, 22 \}$,&  $\{ 4, 12, 13, 19 \}$,&  $\{ 1, 10, 19, 23 \}$,&  $\{ 7, 15, 19, 22 \}$,&  $\{ 2, 8, 9, 19 \}$,&  $\{ 9, 11, 14, 20 \}$,\\
  $\{ 5, 9, 15, 23 \}$,&  $\{ 4, 6, 14, 22 \}$,&  $\{ 5, 12, 14, 18 \}$,&  $\{ 2, 7, 10, 14 \}$,&  $\{ 5, 10, 17, 20 \}$,&  $\{ 2, 5, 11, 22 \}$,\\
  $\{ 6, 11, 17, 23 \}$,&  $\{ 1, 8, 11, 18 \}$,&  $\{ 1, 6, 12, 20 \}$,&  $\{ 1, 4, 15, 17 \}$,&  $\{ 2, 4, 20, 23 \}$,& $\{ 2, 6, 15, 18 \}$,\\
  $\{ 2, 3, 13, 17 \}$,&  $\{ 7, 13, 18, 20 \}$,&  $\{ 3, 7, 12, 23 \}$,&  $\{ 3, 10, 18, 22 \}$,&  $\{ 3, 8, 15, 20 \}$,&  $\{ 8, 12, 17, 22 \}$.
\end{tabular}
\end{center}

$\cdot$ Type $3^3 4^1 5^1 6^1$: $\alpha=(1)$.

\begin{center}
\begin{tabular}{llllll}
$\{ 0, 3, 6, 9 \}$,& $\{ 0, 4, 7, 11 \}$,& $\{ 0, 5, 8, 13 \}$,& $\{ 0, 10, 12, 15 \}$,& $\{ 0, 14, 16, 18 \}$,& $\{ 0, 17, 20, 22 \}$,\\
  $\{ 0, 19, 21, 23 \}$,& $\{ 1, 3, 16, 21 \}$,& $\{ 6, 13, 15, 21 \}$,& $\{ 8, 11, 17, 21 \}$,& $\{ 4, 10, 14, 21 \}$,& $\{ 2, 9, 18, 21 \}$,\\
  $\{ 5, 7, 12, 21 \}$,& $\{ 11, 15, 18, 20 \}$,& $\{ 5, 6, 18, 22 \}$,& $\{ 4, 8, 12, 18 \}$,& $\{ 1, 7, 17, 18 \}$,& $\{ 3, 10, 13, 18 \}$,\\
  $\{ 2, 3, 11, 22 \}$,& $\{ 3, 7, 14, 20 \}$,& $\{ 3, 8, 15, 19 \}$,& $\{ 3, 12, 17, 23 \}$,& $\{ 7, 9, 15, 22 \}$,& $\{ 1, 5, 15, 23 \}$,\\
  $\{ 2, 4, 15, 17 \}$,& $\{ 6, 10, 17, 19 \}$,& $\{ 5, 9, 14, 17 \}$,& $\{ 5, 11, 16, 19 \}$,& $\{ 2, 5, 10, 20 \}$,& $\{ 7, 10, 16, 23 \}$,\\
  $\{ 2, 7, 13, 19 \}$,& $\{ 2, 6, 12, 16 \}$,& $\{ 2, 8, 14, 23 \}$,& $\{ 1, 8, 10, 22 \}$,& $\{ 8, 9, 16, 20 \}$,& $\{ 4, 13, 16, 22 \}$,\\
  $\{ 4, 6, 20, 23 \}$,& $\{ 1, 4, 9, 19 \}$,& $\{ 1, 6, 11, 14 \}$,& $\{ 1, 12, 13, 20 \}$,& $\{ 9, 11, 13, 23 \}$,& $\{ 12, 14, 19, 22 \}$.
\end{tabular}
\end{center}

$\cdot$ Type $3^3 5^3$: $\alpha=(0,3,6)(1,4,7)(2,5,8)(9,14,19)(10,15,20)(11,16,21)(12,17,22)(13,18,23)$.

\begin{center}
\begin{tabular}{llllll}
$\{ 0, 3, 7, 9 \}$, & $\{ 1, 9, 11, 19 \}$, & $\{ 0, 12, 14, 17 \}$, & $\{ 9, 15, 18, 21 \}$, & $\{ 2, 5, 15, 19 \}$, & $\{ 2, 9, 17, 23 \}$, \\
$\{ 0, 10, 20, 22 \}$, & $\{ 0, 8, 15, 23 \}$, & $\{ 0, 11, 13, 18 \}$, & $\{ 0, 5, 16, 21 \}$, & $\{ 1, 8, 16, 22 \}$, & $\{ 1, 5, 12, 23 \}$, \\
$\{ 1, 10, 17, 21 \}$, & $\{ 1, 4, 18, 20 \}$.
\end{tabular}
\end{center}

$\cdot$ Type $3^3 4^2 7^1$: $\alpha=(1)$.

\begin{center}
\begin{tabular}{llllll}
 $\{ 0, 3, 6, 9 \}$,& $\{ 0, 4, 7, 11 \}$,& $\{ 0, 5, 8, 13 \}$,& $\{ 0, 10, 12, 15 \}$,& $\{ 0, 14, 16, 17 \}$,& $\{ 0, 19, 21, 23 \}$,\\
  $\{ 0, 18, 20, 22 \}$,& $\{ 1, 3, 14, 19 \}$,& $\{ 5, 11, 14, 22 \}$,& $\{ 7, 9, 14, 20 \}$,& $\{ 2, 6, 14, 21 \}$,& $\{ 4, 10, 14, 23 \}$,\\
  $\{ 8, 12, 14, 18 \}$,& $\{ 3, 12, 13, 21 \}$,& $\{ 3, 7, 16, 18 \}$,& $\{ 2, 3, 10, 20 \}$,& $\{ 3, 15, 17, 22 \}$,& $\{ 3, 8, 11, 23 \}$,\\
  $\{ 1, 4, 8, 20 \}$,& $\{ 8, 9, 15, 19 \}$,& $\{ 8, 10, 17, 21 \}$,& $\{ 2, 8, 16, 22 \}$,& $\{ 1, 9, 13, 22 \}$,& $\{ 7, 10, 19, 22 \}$,\\
  $\{ 4, 6, 12, 22 \}$,& $\{ 4, 9, 16, 21 \}$,& $\{ 6, 10, 13, 18 \}$,& $\{ 1, 5, 10, 16 \}$,& $\{ 6, 11, 16, 19 \}$,& $\{ 1, 6, 15, 23 \}$,\\
  $\{ 5, 7, 15, 21 \}$,& $\{ 1, 7, 12, 17 \}$,& $\{ 2, 7, 13, 23 \}$,& $\{ 2, 4, 15, 18 \}$,& $\{ 4, 13, 17, 19 \}$,& $\{ 11, 13, 15, 20 \}$,\\
  $\{ 1, 11, 18, 21 \}$,& $\{ 2, 9, 11, 17 \}$,& $\{ 2, 5, 12, 19 \}$,& $\{ 5, 6, 17, 20 \}$,& $\{ 5, 9, 18, 23 \}$,& $\{ 12, 16, 20, 23 \}$.
\end{tabular}
\end{center}

$\cdot$ Type $3^2 18^1$: $\alpha=(0,1,2)(3,4,5)(6,12,18)(7,13,19)(8,14,20)(9,15,21)(10,16,22)(11,17,23)$.

\begin{center}
\begin{tabular}{llllll}
$\{ 0, 3, 6, 12 \}$,& $\{ 0, 7, 18, 21 \}$,& $\{ 3, 9, 16, 18 \}$,& $\{ 6, 10, 16, 20 \}$,& $\{ 6, 8, 11, 17 \}$,& $\{ 6, 14, 19, 21 \}$,\\
$\{ 0, 5, 8, 14 \}$,& $\{ 3, 8, 21, 23 \}$,& $\{ 0, 13, 20, 22 \}$,& $\{ 0, 11, 16, 19 \}$,& $\{ 0, 4, 10, 17 \}$,& $\{ 3, 13, 17, 19 \}$,\\
$\{ 0, 9, 15, 23 \}$,& $\{ 3, 7, 10, 15 \}$.
\end{tabular}
\end{center}

$\cdot$ Type $3^2 9^2$: $\alpha=(0, 1, 2)(3, 4, 5)(6, 9, 12)(7,$ $ 10, 13)(8, 11, 14)(15, 18, 21)(16, 19, 22)(17, 20, 23)$.

\begin{center}
\begin{tabular}{llllll}
$\{ 0, 3, 6, 9 \}$,& $\{ 0, 7, 12, 21 \}$,& $\{ 3, 12, 15, 19 \}$,& $\{ 6, 13, 17, 20 \}$,& $\{ 6, 19, 21, 23 \}$,& $\{ 6, 8, 11, 16 \}$,\\
$\{ 0, 5, 17, 19 \}$,& $\{ 0, 10, 16, 22 \}$,& $\{ 3, 7, 14, 16 \}$,& $\{ 3, 10, 13, 21 \}$,& $\{ 0, 8, 13, 23 \}$,& $\{ 3, 11, 18, 23 \}$,\\
$\{ 0, 14, 15, 18 \}$,& $\{ 0, 4, 11, 20 \}$.
\end{tabular}
\end{center}

$\cdot$ Type $3^2 8^1 10^1$: $\alpha=(1)$.

\begin{center}
\begin{tabular}{llllll}
   $\{ 0, 3, 6, 8 \}$,& $\{ 0, 4, 7, 10 \}$,& $\{ 0, 5, 9, 12 \}$,& $\{ 0, 11, 13, 14 \}$,& $\{ 0, 16, 18, 20 \}$,& $\{ 0, 15, 17, 22 \}$,\\
  $\{ 0, 19, 21, 23 \}$,& $\{ 1, 3, 15, 19 \}$,& $\{ 2, 3, 11, 16 \}$,& $\{ 6, 9, 11, 19 \}$,& $\{ 7, 11, 18, 22 \}$,& $\{ 3, 9, 20, 22 \}$,\\
  $\{ 3, 7, 14, 21 \}$,& $\{ 3, 12, 18, 23 \}$,& $\{ 3, 10, 13, 17 \}$,& $\{ 7, 13, 16, 19 \}$,& $\{ 1, 7, 9, 23 \}$,& $\{ 4, 9, 14, 16 \}$,\\
  $\{ 2, 9, 17, 21 \}$,& $\{ 9, 13, 15, 18 \}$,& $\{ 5, 7, 17, 20 \}$,& $\{ 2, 7, 12, 15 \}$,& $\{ 4, 11, 17, 23 \}$,& $\{ 1, 8, 11, 20 \}$,\\
  $\{ 5, 11, 15, 21 \}$,& $\{ 4, 6, 15, 20 \}$,& $\{ 2, 13, 20, 23 \}$,& $\{ 10, 12, 14, 20 \}$,& $\{ 8, 10, 15, 23 \}$,& $\{ 5, 6, 16, 23 \}$,\\
  $\{ 1, 6, 12, 17 \}$,& $\{ 6, 10, 18, 21 \}$,& $\{ 2, 6, 14, 22 \}$,& $\{ 2, 5, 10, 19 \}$,& $\{ 2, 4, 8, 18 \}$,& $\{ 8, 12, 16, 21 \}$,\\
  $\{ 1, 10, 16, 22 \}$,& $\{ 1, 4, 13, 21 \}$,& $\{ 1, 5, 14, 18 \}$,& $\{ 4, 12, 19, 22 \}$,& $\{ 5, 8, 13, 22 \}$,& $\{ 8, 14, 17, 19 \}$.
\end{tabular}
\end{center}

$\cdot$ Type $3^2 7^1 11^1$: $\alpha=(1)$.

\begin{center}
\begin{tabular}{llllll}
   $\{ 0, 3, 6, 8 \}$,& $\{ 1, 4, 6, 10 \}$,& $\{ 2, 5, 6, 13 \}$,& $\{ 6, 9, 11, 15 \}$,& $\{ 6, 14, 17, 19 \}$,& $\{ 6, 16, 20, 22 \}$,\\
  $\{ 6, 18, 21, 23 \}$,& $\{ 0, 4, 11, 13 \}$,& $\{ 3, 7, 11, 17 \}$,& $\{ 1, 11, 19, 22 \}$,& $\{ 2, 11, 14, 21 \}$,& $\{ 8, 11, 16, 18 \}$,\\
  $\{ 5, 11, 20, 23 \}$,& $\{ 8, 10, 14, 22 \}$,& $\{ 5, 8, 12, 19 \}$,& $\{ 4, 8, 17, 21 \}$,& $\{ 2, 8, 15, 23 \}$,& $\{ 1, 8, 13, 20 \}$,\\
  $\{ 1, 5, 15, 17 \}$,& $\{ 12, 13, 17, 22 \}$,& $\{ 2, 10, 17, 20 \}$,& $\{ 0, 9, 17, 23 \}$,& $\{ 2, 3, 9, 22 \}$,& $\{ 2, 7, 12, 18 \}$,\\
  $\{ 2, 4, 16, 19 \}$,& $\{ 1, 3, 14, 18 \}$,& $\{ 3, 12, 15, 20 \}$,& $\{ 4, 9, 18, 20 \}$,& $\{ 4, 12, 14, 23 \}$,& $\{ 0, 7, 14, 20 \}$,\\
  $\{ 0, 5, 18, 22 \}$,& $\{ 1, 7, 16, 23 \}$,& $\{ 1, 9, 12, 21 \}$,& $\{ 5, 7, 10, 21 \}$,& $\{ 5, 9, 14, 16 \}$,& $\{ 0, 10, 12, 16 \}$,\\
  $\{ 0, 15, 19, 21 \}$,& $\{ 7, 9, 13, 19 \}$,& $\{ 3, 10, 19, 23 \}$,& $\{ 3, 13, 16, 21 \}$,& $\{ 4, 7, 15, 22 \}$,& $\{ 10, 13, 15, 18 \}$.
\end{tabular}
\end{center}

$\cdot$ Type $3^2 6^1 12^1$: $\alpha=(0, 1, 2)(3, 4, 5)(6, 8, 10)(7, 9, 11)(12, 16, 20)(13, 17, 21)(14, 18, 22)(15, 19, 23)$.

\begin{center}
\begin{tabular}{llllll}
   $\{ 0, 3, 12, 16 \}$,& $\{ 0, 14, 17, 20 \}$,& $\{ 6, 8, 12, 14 \}$,& $\{ 6, 16, 21, 23 \}$,& $\{ 7, 9, 20, 23 \}$,& $\{ 3, 11, 18, 20 \}$,\\
  $\{ 0, 9, 18, 22 \}$,& $\{ 3, 10, 14, 21 \}$,& $\{ 3, 15, 19, 22 \}$,& $\{ 0, 5, 6, 19 \}$,& $\{ 3, 7, 13, 17 \}$,& $\{ 0, 4, 8, 11 \}$,\\
  $\{ 0, 7, 15, 21 \}$,& $\{ 0, 10, 13, 23 \}$.
\end{tabular}
\end{center}

$\cdot$ Type $3^2 6^3$: $\alpha=(0,1,2)(3,4,5)(6,12,18)(7,13,19)(8,14,20)(9,15,21)(10,16,22)(11,17,23)$.

\begin{center}
\begin{tabular}{llllll}
 $\{ 0, 3, 6, 12 \}$,& $\{ 0, 7, 18, 21 \}$,& $\{ 6, 15, 17, 22 \}$,& $\{ 6, 16, 19, 21 \}$,& $\{ 6, 8, 10, 20 \}$,& $\{ 3, 8, 17, 18 \}$,\\
  $\{ 3, 9, 15, 23 \}$,& $\{ 0, 9, 14, 19 \}$,& $\{ 0, 5, 8, 15 \}$,& $\{ 3, 13, 16, 20 \}$,& $\{ 0, 11, 13, 22 \}$,& $\{ 0, 4, 10, 16 \}$,\\
  $\{ 0, 17, 20, 23 \}$,& $\{ 3, 7, 11, 19 \}$.
\end{tabular}
\end{center}

$\cdot$ Type $3^2 5^1 6^1 7^1$: $\alpha=(1)$.

\begin{center}
\begin{tabular}{llllll}
  $\{ 0, 3, 6, 8 \},$& $\{ 0, 4, 7, 11 \},$& $\{ 0, 5, 9, 13 \},$& $\{ 0, 10, 12, 15 \},$& $\{ 0, 14, 16, 17 \},$& $\{ 0, 19, 21, 23 \},$\\
  $\{ 0, 18, 20, 22 \},$& $\{ 1, 4, 10, 13 \},$& $\{ 2, 6, 13, 21 \},$& $\{ 3, 13, 16, 18 \},$& $\{ 11, 13, 20, 23 \},$& $\{ 7, 13, 17, 19 \},$\\
  $\{ 8, 13, 15, 22 \},$& $\{ 4, 15, 18, 21 \},$& $\{ 2, 3, 15, 23 \},$& $\{ 7, 9, 15, 20 \},$& $\{ 5, 11, 15, 17 \},$& $\{ 1, 6, 15, 19 \},$\\
  $\{ 6, 9, 11, 18 \},$& $\{ 4, 6, 17, 22 \},$& $\{ 5, 6, 14, 23 \},$& $\{ 6, 12, 16, 20 \},$& $\{ 2, 4, 14, 20 \},$& $\{ 3, 10, 17, 20 \},$\\
  $\{ 1, 5, 8, 20 \},$& $\{ 2, 8, 12, 17 \},$& $\{ 1, 9, 17, 21 \},$& $\{ 1, 7, 16, 23 \},$& $\{ 1, 12, 14, 18 \},$& $\{ 1, 3, 11, 22 \},$\\
  $\{ 2, 10, 11, 19 \},$& $\{ 2, 5, 7, 18 \},$& $\{ 2, 9, 16, 22 \},$& $\{ 4, 8, 16, 19 \},$& $\{ 4, 9, 12, 23 \},$& $\{ 3, 9, 14, 19 \},$\\
  $\{ 3, 7, 12, 21 \},$& $\{ 7, 10, 14, 22 \},$& $\{ 5, 10, 16, 21 \},$& $\{ 5, 12, 19, 22 \},$& $\{ 8, 10, 18, 23 \},$& $\{ 8, 11, 14, 21 \}$.
\end{tabular}
\end{center}

$\cdot$ Type $3^2 5^2 8^1$: $\alpha=(1)$.

\begin{center}
\begin{tabular}{llllll}
   $\{ 0, 3, 6, 8 \}$,& $\{ 0, 4, 7, 11 \}$,& $\{ 0, 5, 9, 13 \}$,& $\{ 0, 10, 12, 16 \}$,& $\{ 0, 14, 18, 20 \}$,& $\{ 0, 15, 17, 22 \}$,\\
  $\{ 0, 19, 21, 23 \}$,& $\{ 1, 3, 15, 19 \}$,& $\{ 7, 9, 16, 19 \}$,& $\{ 3, 11, 14, 16 \}$,& $\{ 3, 9, 20, 23 \}$,& $\{ 3, 7, 12, 22 \}$,\\
  $\{ 2, 3, 17, 21 \}$,& $\{ 3, 10, 13, 18 \}$,& $\{ 1, 5, 7, 18 \}$,& $\{ 7, 10, 15, 21 \}$,& $\{ 7, 14, 17, 23 \}$,& $\{ 2, 7, 13, 20 \}$,\\
  $\{ 6, 9, 12, 17 \}$,& $\{ 2, 4, 9, 15 \}$,& $\{ 1, 9, 14, 22 \}$,& $\{ 9, 11, 18, 21 \}$,& $\{ 6, 13, 15, 23 \}$,& $\{ 8, 15, 16, 18 \}$,\\
  $\{ 5, 12, 15, 20 \}$,& $\{ 2, 12, 18, 23 \}$,& $\{ 4, 6, 18, 22 \}$,& $\{ 2, 5, 16, 22 \}$,& $\{ 2, 6, 11, 19 \}$,& $\{ 2, 8, 10, 14 \}$,\\
  $\{ 8, 13, 19, 22 \}$,& $\{ 1, 11, 13, 17 \}$,& $\{ 4, 13, 16, 21 \}$,& $\{ 1, 6, 16, 20 \}$,& $\{ 4, 8, 17, 20 \}$,& $\{ 4, 12, 14, 19 \}$,\\
  $\{ 1, 4, 10, 23 \}$,& $\{ 5, 8, 11, 23 \}$,& $\{ 1, 8, 12, 21 \}$,& $\{ 5, 6, 14, 21 \}$,& $\{ 5, 10, 17, 19 \}$,& $\{ 10, 11, 20, 22 \}$.
\end{tabular}
\end{center}

$\cdot$ Type $3^2 5^1 13^1$: $\alpha=(1)$.

\begin{center}
\begin{tabular}{llllll}
    $\{ 0, 3, 6, 8 \}$,& $\{ 0, 4, 7, 11 \}$,& $\{ 0, 5, 9, 13 \}$,& $\{ 0, 10, 12, 15 \}$,& $\{ 0, 14, 17, 19 \}$,& $\{ 0, 16, 20, 22 \}$,\\
  $\{ 0, 18, 21, 23 \}$,& $\{ 1, 4, 8, 10 \}$,& $\{ 5, 8, 17, 23 \}$,& $\{ 2, 8, 14, 20 \}$,& $\{ 8, 12, 16, 21 \}$,& $\{ 8, 13, 18, 22 \}$,\\
  $\{ 8, 11, 15, 19 \}$,& $\{ 4, 12, 14, 22 \}$,& $\{ 4, 6, 16, 18 \}$,& $\{ 2, 4, 13, 19 \}$,& $\{ 4, 15, 17, 21 \}$,& $\{ 4, 9, 20, 23 \}$,\\
  $\{ 6, 13, 15, 23 \}$,& $\{ 2, 10, 16, 23 \}$,& $\{ 1, 3, 14, 23 \}$,& $\{ 7, 12, 19, 23 \}$,& $\{ 1, 5, 16, 19 \}$,& $\{ 6, 9, 19, 22 \}$,\\
  $\{ 3, 10, 19, 21 \}$,& $\{ 3, 7, 13, 16 \}$,& $\{ 9, 11, 14, 16 \}$,& $\{ 5, 6, 14, 21 \}$,& $\{ 7, 10, 14, 18 \}$,& $\{ 2, 7, 9, 21 \}$,\\
  $\{ 1, 11, 13, 21 \}$,& $\{ 10, 13, 17, 20 \}$,& $\{ 5, 10, 11, 22 \}$,& $\{ 3, 11, 18, 20 \}$,& $\{ 2, 6, 11, 17 \}$,& $\{ 3, 9, 12, 17 \}$,\\
  $\{ 1, 6, 12, 20 \}$,& $\{ 5, 7, 15, 20 \}$,& $\{ 2, 5, 12, 18 \}$,& $\{ 1, 7, 17, 22 \}$,& $\{ 1, 9, 15, 18 \}$,& $\{ 2, 3, 15, 22 \}$.
\end{tabular}
\end{center}

$\cdot$ Type $3^2 4^3 6^1$: $\alpha=(0,1,2)(3,4,5)(6,10,14)(7,11,15)(8,12,16)(9,13,17)(18,20,22)(19,21,23)$.

\begin{center}
\begin{tabular}{llllll}
$\{ 0, 3, 6, 10 \}$,& $\{ 6, 12, 16, 18 \}$,& $\{ 0, 11, 14, 18 \}$,& $\{ 6, 11, 13, 22 \}$,& $\{ 3, 17, 20, 22 \}$,& $\{ 0, 4, 12, 20 \}$,\\
$\{ 0, 15, 19, 22 \}$,& $\{ 0, 7, 13, 16 \}$,& $\{ 3, 7, 15, 23 \}$,& $\{ 3, 9, 11, 16 \}$,& $\{ 0, 5, 8, 23 \}$,& $\{ 0, 9, 17, 21 \}$,\\
$\{ 6, 8, 19, 21 \}$,& $\{ 3, 13, 14, 21 \}$.
\end{tabular}
\end{center}

$\cdot$ Type $3^2 4^2 10^1$: $\alpha=(1)$.

\begin{center}
\begin{tabular}{llllll}
  $\{ 0, 3, 6, 8 \}$,& $\{ 1, 4, 6, 10 \}$,& $\{ 2, 5, 6, 12 \}$,& $\{ 6, 11, 13, 14 \}$,& $\{ 6, 16, 18, 20 \}$,& $\{ 6, 15, 17, 22 \}$,\\
  $\{ 6, 19, 21, 23 \}$,& $\{ 2, 8, 10, 16 \}$,& $\{ 1, 8, 19, 22 \}$,& $\{ 8, 12, 14, 21 \}$,& $\{ 8, 11, 17, 23 \}$,& $\{ 4, 8, 15, 18 \}$,\\
  $\{ 5, 8, 13, 20 \}$,& $\{ 1, 3, 14, 17 \}$,& $\{ 1, 11, 15, 20 \}$,& $\{ 1, 7, 12, 18 \}$,& $\{ 1, 9, 13, 21 \}$,& $\{ 1, 5, 16, 23 \}$,\\
  $\{ 2, 13, 18, 23 \}$,& $\{ 3, 7, 13, 15 \}$,& $\{ 2, 9, 15, 19 \}$,& $\{ 2, 3, 11, 22 \}$,& $\{ 3, 10, 18, 21 \}$,& $\{ 3, 12, 16, 19 \}$,\\
  $\{ 3, 9, 20, 23 \}$,& $\{ 0, 5, 15, 21 \}$,& $\{ 10, 12, 15, 23 \}$,& $\{ 0, 4, 7, 23 \}$,& $\{ 5, 7, 11, 19 \}$,& $\{ 0, 9, 11, 18 \}$,\\
  $\{ 4, 11, 16, 21 \}$,& $\{ 0, 13, 16, 22 \}$,& $\{ 4, 13, 17, 19 \}$,& $\{ 2, 4, 14, 20 \}$,& $\{ 2, 7, 17, 21 \}$,& $\{ 5, 9, 10, 17 \}$,\\
  $\{ 5, 14, 18, 22 \}$,& $\{ 7, 9, 14, 16 \}$,& $\{ 0, 10, 14, 19 \}$,& $\{ 0, 12, 17, 20 \}$,& $\{ 4, 9, 12, 22 \}$,& $\{ 7, 10, 20, 22 \}$.
\end{tabular}
\end{center}

$\cdot$ Type $3^2 4^2 5^2$: $\alpha=(1)$.

\begin{center}
\begin{tabular}{llllll}
$\{ 0, 3, 6, 8 \}$,& $\{ 1, 4, 6, 10 \}$,& $\{ 2, 5, 6, 12 \}$,& $\{ 6, 11, 13, 14 \}$,& $\{ 6, 16, 18, 19 \}$,& $\{ 6, 15, 21, 23 \}$,\\
  $\{ 6, 17, 20, 22 \}$,& $\{ 2, 8, 10, 16 \}$,& $\{ 4, 8, 14, 20 \}$,& $\{ 8, 11, 18, 23 \}$,& $\{ 8, 12, 19, 21 \}$,& $\{ 1, 8, 15, 22 \}$,\\
  $\{ 5, 8, 13, 17 \}$,& $\{ 1, 14, 17, 19 \}$,& $\{ 2, 3, 14, 21 \}$,& $\{ 2, 7, 17, 23 \}$,& $\{ 2, 11, 15, 20 \}$,& $\{ 2, 13, 18, 22 \}$,\\
  $\{ 2, 4, 9, 19 \}$,& $\{ 0, 4, 13, 23 \}$,& $\{ 4, 12, 15, 17 \}$,& $\{ 4, 7, 18, 21 \}$,& $\{ 4, 11, 16, 22 \}$,& $\{ 3, 9, 11, 17 \}$,\\
  $\{ 0, 10, 17, 21 \}$,& $\{ 1, 5, 11, 21 \}$,& $\{ 0, 7, 11, 19 \}$,& $\{ 3, 13, 15, 19 \}$,& $\{ 5, 10, 19, 22 \}$,& $\{ 3, 7, 12, 22 \}$,\\
  $\{ 0, 9, 14, 22 \}$,& $\{ 0, 5, 15, 18 \}$,& $\{ 0, 12, 16, 20 \}$,& $\{ 10, 12, 14, 23 \}$,& $\{ 1, 9, 12, 18 \}$,& $\{ 3, 10, 18, 20 \}$,\\
  $\{ 1, 3, 16, 23 \}$,& $\{ 1, 7, 13, 20 \}$,& $\{ 9, 13, 16, 21 \}$,& $\{ 5, 7, 14, 16 \}$,& $\{ 5, 9, 20, 23 \}$,& $\{ 7, 9, 10, 15 \}$.
\end{tabular}
\end{center}

$\cdot$ Type $3^2 4^1 14^1$: $\alpha=(1)$.

\begin{center}
\begin{tabular}{llllll}
   $\{ 0, 3, 6, 8 \}$,& $\{ 1, 4, 6, 10 \}$,& $\{ 2, 5, 6, 12 \}$,& $\{ 6, 11, 14, 16 \}$,& $\{ 6, 13, 15, 18 \}$,& $\{ 6, 17, 20, 22 \}$,\\
  $\{ 6, 19, 21, 23 \}$,& $\{ 1, 8, 11, 21 \}$,& $\{ 8, 12, 15, 20 \}$,& $\{ 4, 8, 18, 23 \}$,& $\{ 8, 14, 19, 22 \}$,& $\{ 2, 8, 13, 17 \}$,\\
  $\{ 5, 8, 10, 16 \}$,& $\{ 0, 7, 13, 21 \}$,& $\{ 5, 14, 18, 21 \}$,& $\{ 3, 10, 13, 22 \}$,& $\{ 2, 10, 15, 21 \}$,& $\{ 3, 9, 16, 21 \}$,\\
  $\{ 4, 12, 17, 21 \}$,& $\{ 1, 13, 16, 23 \}$,& $\{ 4, 11, 13, 19 \}$,& $\{ 5, 9, 13, 20 \}$,& $\{ 2, 3, 7, 19 \}$,& $\{ 0, 12, 16, 19 \}$,\\
  $\{ 9, 10, 17, 19 \}$,& $\{ 1, 5, 15, 19 \}$,& $\{ 0, 10, 14, 20 \}$,& $\{ 7, 10, 12, 18 \}$,& $\{ 1, 3, 18, 20 \}$,& $\{ 3, 12, 14, 23 \}$,\\
  $\{ 3, 11, 15, 17 \}$,& $\{ 0, 5, 17, 23 \}$,& $\{ 5, 7, 11, 22 \}$,& $\{ 4, 7, 16, 20 \}$,& $\{ 2, 4, 9, 14 \}$,& $\{ 0, 9, 11, 18 \}$,\\
  $\{ 0, 4, 15, 22 \}$,& $\{ 1, 7, 14, 17 \}$,& $\{ 1, 9, 12, 22 \}$,& $\{ 2, 11, 20, 23 \}$,& $\{ 2, 16, 18, 22 \}$,& $\{ 7, 9, 15, 23 \}$.
\end{tabular}
\end{center}

$\cdot$ Type $3^2 4^1 5^1 9^1$: $\alpha=(1)$.

\begin{center}
\begin{tabular}{llllll}
 $\{ 0, 3, 6, 8 \}$,& $\{ 1, 4, 6, 10 \}$,& $\{ 2, 5, 6, 12 \}$,& $\{ 6, 11, 13, 15 \}$,& $\{ 6, 14, 17, 19 \}$,& $\{ 6, 16, 20, 22 \}$,\\
  $\{ 6, 18, 21, 23 \}$,& $\{ 1, 8, 12, 18 \}$,& $\{ 2, 8, 14, 16 \}$,& $\{ 8, 11, 19, 23 \}$,& $\{ 4, 8, 15, 20 \}$,& $\{ 8, 10, 13, 21 \}$,\\
  $\{ 5, 8, 17, 22 \}$,& $\{ 1, 13, 19, 22 \}$,& $\{ 2, 15, 19, 21 \}$,& $\{ 3, 7, 12, 19 \}$,& $\{ 4, 9, 16, 19 \}$,& $\{ 0, 5, 10, 19 \}$,\\
  $\{ 10, 12, 16, 23 \}$,& $\{ 1, 3, 16, 21 \}$,& $\{ 0, 11, 16, 18 \}$,& $\{ 5, 7, 13, 16 \}$,& $\{ 5, 11, 14, 21 \}$,& $\{ 5, 9, 15, 18 \}$,\\
  $\{ 1, 5, 20, 23 \}$,& $\{ 1, 7, 14, 15 \}$,& $\{ 1, 9, 11, 17 \}$,& $\{ 0, 12, 15, 22 \}$,& $\{ 0, 7, 9, 21 \}$,& $\{ 0, 13, 17, 20 \}$,\\
  $\{ 0, 4, 14, 23 \}$,& $\{ 2, 4, 13, 18 \}$,& $\{ 4, 12, 17, 21 \}$,& $\{ 3, 9, 13, 23 \}$,& $\{ 9, 12, 14, 20 \}$,& $\{ 2, 9, 10, 22 \}$,\\
  $\{ 3, 10, 15, 17 \}$,& $\{ 3, 14, 18, 22 \}$,& $\{ 2, 3, 11, 20 \}$,& $\{ 2, 7, 17, 23 \}$,& $\{ 4, 7, 11, 22 \}$,& $\{ 7, 10, 18, 20 \}$.
\end{tabular}
\end{center}

$\cdot$ Type $3^2 4^1 6^1 8^1$: $\alpha=(1)$.

\begin{center}
\begin{tabular}{llllll}
  $\{ 0, 3, 6, 8 \}$,& $\{ 1, 4, 6, 10 \}$,& $\{ 2, 5, 6, 12 \}$,& $\{ 6, 11, 14, 16 \}$,& $\{ 6, 13, 15, 18 \}$,& $\{ 6, 17, 20, 22 \}$,\\
  $\{ 6, 19, 21, 23 \}$,& $\{ 1, 8, 11, 21 \}$,& $\{ 8, 12, 15, 23 \}$,& $\{ 2, 8, 10, 17 \}$,& $\{ 5, 8, 14, 18 \}$,& $\{ 4, 8, 16, 20 \}$,\\
  $\{ 8, 13, 19, 22 \}$,& $\{ 1, 5, 15, 20 \}$,& $\{ 1, 7, 13, 17 \}$,& $\{ 1, 3, 18, 23 \}$,& $\{ 1, 9, 14, 22 \}$,& $\{ 1, 12, 16, 19 \}$,\\
  $\{ 0, 5, 16, 22 \}$,& $\{ 10, 12, 18, 22 \}$,& $\{ 3, 7, 11, 22 \}$,& $\{ 2, 4, 15, 22 \}$,& $\{ 2, 3, 9, 19 \}$,& $\{ 3, 15, 17, 21 \}$,\\
  $\{ 3, 12, 14, 20 \}$,& $\{ 3, 10, 13, 16 \}$,& $\{ 7, 9, 15, 16 \}$,& $\{ 0, 11, 15, 19 \}$,& $\{ 4, 14, 17, 19 \}$,& $\{ 2, 16, 18, 21 \}$,\\
  $\{ 2, 7, 14, 23 \}$,& $\{ 2, 11, 13, 20 \}$,& $\{ 4, 9, 11, 18 \}$,& $\{ 0, 10, 14, 21 \}$,& $\{ 0, 7, 18, 20 \}$,& $\{ 0, 9, 12, 17 \}$,\\
  $\{ 0, 4, 13, 23 \}$,& $\{ 5, 9, 13, 21 \}$,& $\{ 5, 7, 10, 19 \}$,& $\{ 4, 7, 12, 21 \}$,& $\{ 5, 11, 17, 23 \}$,& $\{ 9, 10, 20, 23 \}$.
\end{tabular}
\end{center}

$\cdot$ Type $3^2 4^1 7^1 7^1$: $\alpha=(1)$.

\begin{center}
\begin{tabular}{llllll}
  $\{ 0, 3, 6, 8 \}$,& $\{ 1, 4, 6, 10 \}$,& $\{ 2, 5, 6, 12 \}$,& $\{ 6, 11, 14, 16 \}$,& $\{ 6, 13, 15, 17 \}$,& $\{ 6, 19, 21, 23 \}$,\\
  $\{ 6, 18, 20, 22 \}$,& $\{ 2, 8, 10, 17 \}$,& $\{ 8, 12, 14, 18 \}$,& $\{ 5, 8, 16, 21 \}$,& $\{ 8, 11, 19, 22 \}$,& $\{ 1, 8, 15, 23 \}$,\\
  $\{ 4, 8, 13, 20 \}$,& $\{ 2, 3, 7, 19 \}$,& $\{ 2, 9, 16, 22 \}$,& $\{ 2, 11, 15, 20 \}$,& $\{ 2, 13, 18, 23 \}$,& $\{ 2, 4, 14, 21 \}$,\\
  $\{ 1, 13, 16, 19 \}$,& $\{ 0, 4, 16, 18 \}$,& $\{ 3, 16, 17, 20 \}$,& $\{ 7, 12, 16, 23 \}$,& $\{ 4, 9, 11, 23 \}$,& $\{ 3, 10, 14, 23 \}$,\\
  $\{ 0, 5, 20, 23 \}$,& $\{ 3, 11, 13, 21 \}$,& $\{ 0, 7, 11, 17 \}$,& $\{ 0, 9, 14, 19 \}$,& $\{ 0, 10, 13, 22 \}$,& $\{ 0, 12, 15, 21 \}$,\\
  $\{ 3, 9, 15, 18 \}$,& $\{ 1, 3, 12, 22 \}$,& $\{ 9, 10, 12, 20 \}$,& $\{ 1, 5, 11, 18 \}$,& $\{ 5, 14, 17, 22 \}$,& $\{ 1, 7, 14, 20 \}$,\\
  $\{ 1, 9, 17, 21 \}$,& $\{ 4, 7, 15, 22 \}$,& $\{ 5, 10, 15, 19 \}$,& $\{ 4, 12, 17, 19 \}$,& $\{ 5, 7, 9, 13 \}$,& $\{ 7, 10, 18, 21 \}$.
\end{tabular}
\end{center}

$\cdot$ Type $3^1 21^1$: $\alpha=(0,1,2)(3,10,17)(4,11,18)(5,12,19)(6,13,20)(7,14,21)(8,15,22)(9,16,23)$.

\begin{center}
\begin{tabular}{llllll}
  $\{ 0, 3, 5, 10 \}$,& $\{ 3, 6, 8, 12 \}$,& $\{ 0, 8, 17, 23 \}$,& $\{ 3, 7, 14, 22 \}$,& $\{ 3, 13, 18, 20 \}$,& $\{ 3, 11, 16, 21 \}$,\\
  $\{ 0, 9, 13, 22 \}$,& $\{ 4, 12, 15, 22 \}$,& $\{ 0, 11, 15, 20 \}$,& $\{ 0, 4, 18, 21 \}$,& $\{ 4, 16, 19, 23 \}$,& $\{ 5, 7, 13, 16 \}$,\\
  $\{ 0, 7, 12, 19 \}$,& $\{ 0, 6, 14, 16 \}$.
\end{tabular}
\end{center}

$\cdot$ Type $3^1 10^1 11^1$: $\alpha=(1)$.

\begin{center}
\begin{tabular}{llllll}
   $\{ 0, 3, 5, 7 \}$,& $\{ 1, 5, 9, 11 \}$,& $\{ 2, 5, 8, 13 \}$,& $\{ 5, 10, 12, 15 \}$,& $\{ 5, 14, 17, 19 \}$,& $\{ 5, 16, 20, 22 \}$,\\
  $\{ 5, 18, 21, 23 \}$,& $\{ 2, 3, 9, 14 \}$,& $\{ 2, 6, 15, 18 \}$,& $\{ 2, 10, 19, 22 \}$,& $\{ 2, 12, 16, 21 \}$,& $\{ 2, 4, 17, 23 \}$,\\
  $\{ 2, 7, 11, 20 \}$,& $\{ 7, 9, 19, 23 \}$,& $\{ 4, 6, 9, 16 \}$,& $\{ 9, 13, 15, 21 \}$,& $\{ 7, 15, 17, 22 \}$,& $\{ 9, 12, 17, 20 \}$,\\
  $\{ 0, 9, 18, 22 \}$,& $\{ 1, 4, 7, 21 \}$,& $\{ 7, 12, 14, 18 \}$,& $\{ 7, 10, 13, 16 \}$,& $\{ 0, 6, 12, 23 \}$,& $\{ 1, 8, 12, 22 \}$,\\
  $\{ 4, 12, 13, 19 \}$,& $\{ 4, 11, 14, 22 \}$,& $\{ 0, 11, 13, 17 \}$,& $\{ 3, 6, 13, 22 \}$,& $\{ 1, 13, 18, 20 \}$,& $\{ 4, 8, 10, 18 \}$,\\
  $\{ 0, 4, 15, 20 \}$,& $\{ 0, 10, 14, 21 \}$,& $\{ 0, 8, 16, 19 \}$,& $\{ 1, 14, 16, 23 \}$,& $\{ 3, 10, 20, 23 \}$,& $\{ 3, 8, 17, 21 \}$,\\
  $\{ 8, 11, 15, 23 \}$,& $\{ 6, 11, 19, 21 \}$,& $\{ 1, 3, 15, 19 \}$,& $\{ 1, 6, 10, 17 \}$,& $\{ 3, 11, 16, 18 \}$,& $\{ 6, 8, 14, 20 \}$.
\end{tabular}
\end{center}

$\cdot$ Type $3^1 7^1 14^1$: $\alpha=(1)$.

\begin{center}
\begin{tabular}{llllll}
   $\{ 0, 3, 5, 7 \}$,& $\{ 1, 5, 8, 10 \}$,& $\{ 2, 5, 9, 12 \}$,& $\{ 5, 11, 14, 16 \}$,& $\{ 5, 13, 15, 18 \}$,& $\{ 5, 17, 20, 22 \}$,\\
  $\{ 5, 19, 21, 23 \}$,& $\{ 0, 4, 15, 20 \}$,& $\{ 0, 12, 14, 23 \}$,& $\{ 10, 12, 16, 20 \}$,& $\{ 7, 14, 18, 20 \}$,& $\{ 1, 9, 20, 23 \}$,\\
  $\{ 2, 11, 13, 20 \}$,& $\{ 3, 6, 8, 20 \}$,& $\{ 3, 11, 17, 23 \}$,& $\{ 2, 6, 15, 23 \}$,& $\{ 4, 7, 13, 23 \}$,& $\{ 8, 16, 18, 23 \}$,\\
  $\{ 0, 8, 13, 17 \}$,& $\{ 8, 11, 15, 21 \}$,& $\{ 4, 8, 12, 19 \}$,& $\{ 2, 8, 14, 22 \}$,& $\{ 3, 10, 14, 19 \}$,& $\{ 2, 3, 18, 21 \}$,\\
  $\{ 2, 4, 10, 17 \}$,& $\{ 2, 7, 16, 19 \}$,& $\{ 1, 7, 11, 22 \}$,& $\{ 1, 6, 12, 18 \}$,& $\{ 1, 15, 17, 19 \}$,& $\{ 1, 4, 14, 21 \}$,\\
  $\{ 1, 3, 13, 16 \}$,& $\{ 3, 12, 15, 22 \}$,& $\{ 7, 9, 10, 15 \}$,& $\{ 6, 9, 14, 17 \}$,& $\{ 4, 6, 16, 22 \}$,& $\{ 4, 9, 11, 18 \}$,\\
  $\{ 0, 6, 11, 19 \}$,& $\{ 0, 10, 18, 22 \}$,& $\{ 0, 9, 16, 21 \}$,& $\{ 6, 10, 13, 21 \}$,& $\{ 7, 12, 17, 21 \}$,& $\{ 9, 13, 19, 22 \}$.
\end{tabular}
\end{center}

$\cdot$ Type $3^1 6^1 7^1 8^1$: $\alpha=(1)$.

\begin{center}
\begin{tabular}{llllll}
   $\{ 0, 3, 6, 9 \}$,& $\{ 0, 4, 7, 11 \}$,& $\{ 0, 5, 8, 13 \}$,& $\{ 0, 10, 12, 16 \}$,& $\{ 0, 14, 18, 20 \}$,& $\{ 0, 15, 17, 22 \}$,\\
  $\{ 0, 19, 21, 23 \}$,& $\{ 1, 4, 15, 19 \}$,& $\{ 2, 6, 11, 19 \}$,& $\{ 8, 10, 19, 22 \}$,& $\{ 7, 9, 13, 19 \}$,& $\{ 5, 14, 16, 19 \}$,\\
  $\{ 3, 12, 17, 19 \}$,& $\{ 3, 11, 14, 23 \}$,& $\{ 6, 12, 14, 22 \}$,& $\{ 7, 10, 14, 21 \}$,& $\{ 1, 8, 9, 14 \}$,& $\{ 2, 4, 14, 17 \}$,\\
  $\{ 5, 7, 17, 20 \}$,& $\{ 9, 11, 17, 21 \}$,& $\{ 6, 8, 17, 23 \}$,& $\{ 1, 10, 13, 17 \}$,& $\{ 1, 6, 16, 20 \}$,& $\{ 6, 10, 15, 18 \}$,\\
  $\{ 4, 6, 13, 21 \}$,& $\{ 4, 10, 20, 23 \}$,& $\{ 2, 3, 5, 10 \}$,& $\{ 3, 13, 20, 22 \}$,& $\{ 2, 9, 12, 20 \}$,& $\{ 4, 8, 12, 18 \}$,\\
  $\{ 4, 9, 16, 22 \}$,& $\{ 5, 9, 18, 23 \}$,& $\{ 1, 5, 11, 22 \}$,& $\{ 1, 7, 12, 23 \}$,& $\{ 1, 3, 18, 21 \}$,& $\{ 3, 7, 15, 16 \}$,\\
  $\{ 2, 7, 18, 22 \}$,& $\{ 11, 13, 16, 18 \}$,& $\{ 2, 13, 15, 23 \}$,& $\{ 2, 8, 16, 21 \}$,& $\{ 5, 12, 15, 21 \}$,& $\{ 8, 11, 15, 20 \}$.
\end{tabular}
\end{center}

$\cdot$ Type $3^1 9^1 12^1$: $\alpha=(0,1,2)(3,6,9)(4,7,10)(5,8,11)(12,16,20)(13,17,21)(14,18,22)(15,19,23)$.

\begin{center}
\begin{tabular}{llllll}
$\{ 0, 3, 5, 9 \}$,& $\{ 0, 6, 13, 16 \}$,& $\{ 3, 10, 16, 18 \}$,& $\{ 3, 14, 20, 23 \}$,& $\{ 0, 10, 12, 20 \}$,& $\{ 3, 7, 15, 19 \}$,\\
$\{ 3, 13, 17, 22 \}$,& $\{ 5, 8, 12, 22 \}$,& $\{ 5, 16, 21, 23 \}$,& $\{ 0, 4, 14, 18 \}$,& $\{ 0, 8, 19, 21 \}$,& $\{ 0, 7, 17, 23 \}$,\\
$\{ 0, 11, 15, 22 \}$,& $\{ 4, 7, 11, 21 \}$.
\end{tabular}
\end{center}

$\cdot$ Type $3^1 8^1 13^1$: $\alpha=(1)$.

\begin{center}
\begin{tabular}{llllll}
   $\{ 0, 3, 5, 7 \}$,& $\{ 1, 5, 9, 11 \}$,& $\{ 2, 5, 8, 13 \}$,& $\{ 5, 10, 12, 15 \}$,& $\{ 5, 14, 17, 19 \}$,& $\{ 5, 16, 20, 22 \}$,\\
  $\{ 5, 18, 21, 23 \}$,& $\{ 2, 3, 6, 11 \}$,& $\{ 3, 8, 17, 21 \}$,& $\{ 3, 12, 19, 22 \}$,& $\{ 1, 3, 14, 16 \}$,& $\{ 3, 9, 13, 18 \}$,\\
  $\{ 3, 15, 20, 23 \}$,& $\{ 2, 4, 15, 19 \}$,& $\{ 2, 7, 18, 22 \}$,& $\{ 2, 10, 17, 23 \}$,& $\{ 2, 9, 14, 20 \}$,& $\{ 2, 12, 16, 21 \}$,\\
  $\{ 1, 8, 15, 18 \}$,& $\{ 7, 11, 15, 17 \}$,& $\{ 0, 9, 15, 22 \}$,& $\{ 6, 13, 15, 21 \}$,& $\{ 7, 9, 19, 21 \}$,& $\{ 1, 4, 10, 21 \}$,\\
  $\{ 0, 11, 14, 21 \}$,& $\{ 1, 13, 19, 23 \}$,& $\{ 1, 6, 17, 22 \}$,& $\{ 1, 7, 12, 20 \}$,& $\{ 4, 7, 14, 23 \}$,& $\{ 7, 10, 13, 16 \}$,\\
  $\{ 4, 11, 13, 22 \}$,& $\{ 10, 11, 18, 20 \}$,& $\{ 8, 11, 16, 19 \}$,& $\{ 0, 6, 10, 19 \}$,& $\{ 6, 9, 16, 23 \}$,& $\{ 0, 13, 17, 20 \}$,\\
  $\{ 4, 9, 12, 17 \}$,& $\{ 0, 8, 12, 23 \}$,& $\{ 8, 10, 14, 22 \}$,& $\{ 0, 4, 16, 18 \}$,& $\{ 4, 6, 8, 20 \}$,& $\{ 6, 12, 14, 18 \}$.
\end{tabular}
\end{center}

$\cdot$ Type $3^1 7^3$: $\alpha=(0,1,2)(3,10,17)(4,11,18)(5,12,19)(6,13,20)(7,14,21)(8,15,22)(9,16,23)$.

\begin{center}
\begin{tabular}{llllll}
$\{ 0, 3, 5, 10 \}$,& $\{ 3, 6, 8, 12 \}$,& $\{ 0, 8, 11, 17 \}$,& $\{ 3, 7, 11, 22 \}$,& $\{ 3, 13, 16, 20 \}$,& $\{ 3, 14, 21, 23 \}$,\\
$\{ 0, 13, 18, 22 \}$,& $\{ 5, 8, 15, 23 \}$,& $\{ 0, 7, 15, 20 \}$,& $\{ 4, 6, 16, 18 \}$,& $\{ 0, 6, 19, 21 \}$,& $\{ 0, 12, 16, 23 \}$,\\
$\{ 4, 7, 12, 19 \}$,& $\{ 0, 4, 9, 14 \}$.
\end{tabular}
\end{center}

$\cdot$ Type $3^1 6^1 15^1$: $\alpha=(0,1,2)(3,5,7)(4,6,8)(9,14,19)(10,15,20)(11,16,21)(12,17,22)(13,18,23)$.

\begin{center}
\begin{tabular}{llllll}
$\{ 0, 3, 9, 14 \}$,& $\{ 3, 6, 10, 19 \}$,& $\{ 0, 8, 12, 19 \}$,& $\{ 9, 12, 20, 22 \}$,& $\{ 4, 11, 19, 21 \}$,& $\{ 9, 13, 18, 21 \}$,\\
$\{ 4, 6, 13, 15 \}$,& $\{ 0, 4, 18, 22 \}$,& $\{ 0, 6, 17, 21 \}$,& $\{ 0, 7, 15, 23 \}$,& $\{ 0, 10, 16, 20 \}$,& $\{ 0, 5, 11, 13 \}$,\\
$\{ 3, 5, 16, 22 \}$,& $\{ 3, 12, 15, 18 \}$.
\end{tabular}
\end{center}

$\cdot$ Type $3^1 6^2 9^1$: $\alpha=(0,1,2)(3,5,7)(4,6,8)(9,11,13)(10,12,14)(15,18,21)(16,19,22)(17,20,23)$.

\begin{center}
\begin{tabular}{llllll}
$\{ 0, 3, 15, 18 \}$,& $\{ 3, 6, 9, 21 \}$,& $\{ 4, 11, 16, 21 \}$,& $\{ 9, 12, 15, 17 \}$,& $\{ 4, 14, 15, 22 \}$,& $\{ 0, 14, 17, 21 \}$,\\
$\{ 0, 4, 8, 10 \}$,& $\{ 4, 17, 19, 23 \}$,& $\{ 0, 6, 11, 23 \}$,& $\{ 3, 5, 13, 20 \}$,& $\{ 0, 7, 12, 20 \}$,& $\{ 3, 10, 12, 22 \}$,\\
$\{ 0, 5, 19, 22 \}$,& $\{ 0, 9, 13, 16 \}$.
\end{tabular}
\end{center}

$\cdot$ Type $3^1 5^3 6^1$: $\alpha=(0,1,2)(3,8,13)(4,9,14)(5,10,15)(6,11,16)(7,12,17)(18,20,22)(19,21,23)$.

\begin{center}
\begin{tabular}{llllll}
$\{ 0, 3, 5, 8 \}$,& $\{ 3, 6, 10, 18 \}$,& $\{ 3, 9, 20, 23 \}$,& $\{ 3, 11, 17, 22 \}$,& $\{ 3, 12, 14, 21 \}$,& $\{ 0, 11, 13, 23 \}$,\\
$\{ 4, 6, 11, 19 \}$,& $\{ 0, 6, 9, 22 \}$,& $\{ 0, 10, 12, 16 \}$,& $\{ 0, 17, 18, 20 \}$,& $\{ 0, 15, 19, 21 \}$,& $\{ 0, 4, 7, 14 \}$,\\
$\{ 4, 10, 15, 22 \}$,& $\{ 5, 12, 17, 19 \}$.
\end{tabular}
\end{center}

$\cdot$ Type $3^1 5^2 11^1$: $\alpha=(1)$.

\begin{center}
\begin{tabular}{llllll}
$\{ 0, 3, 5, 8 \}$,& $\{ 1, 3, 9, 11 \}$,& $\{ 2, 3, 8, 13 \}$,& $\{ 3, 10, 12, 15 \}$,& $\{ 3, 14, 17, 20 \}$,& $\{ 3, 16, 20, 23 \}$,\\
$\{ 3, 19, 21, 23 \}$,& $\{ 2, 5, 9, 13 \}$,& $\{ 5, 12, 13, 20 \}$,& $\{ 5, 8, 14, 16 \}$,& $\{ 5, 10, 18, 21 \}$,& $\{ 5, 10, 15, 21 \}$,\\
$\{ 1, 5, 20, 22 \}$,& $\{ 1, 7, 14, 19 \}$,& $\{ 1, 13, 18, 21 \}$,& $\{ 1, 4, 12, 17 \}$,& $\{ 1, 6, 16, 22 \}$,& $\{ 1, 7, 18, 22 \}$,\\
$\{ 0, 12, 16, 19 \}$,& $\{ 6, 10, 15, 22 \}$,& $\{ 2, 7, 12, 19 \}$,& $\{ 8, 11, 13, 21 \}$,& $\{ 6, 10, 12, 17 \}$,& $\{ 2, 15, 18, 20 \}$,\\
$\{ 0, 9, 16, 18 \}$,& $\{ 4, 9, 16, 18 \}$,& $\{ 9, 13, 16, 20 \}$,& $\{ 0, 11, 14, 23 \}$,& $\{ 7, 9, 20, 23 \}$,& $\{ 7, 10, 14, 17 \}$,\\
$\{ 4, 14, 17, 20 \}$,& $\{ 4, 10, 13, 22 \}$,& $\{ 2, 11, 16, 22 \}$,& $\{ 8, 10, 18, 21 \}$,& $\{ 0, 4, 8, 19 \}$,& $\{ 0, 6, 11, 23 \}$,\\
$\{ 0, 12, 21, 23 \}$,& $\{ 2, 4, 6, 18 \}$,& $\{ 4, 7, 11, 21 \}$,& $\{ 2, 10, 19, 23 \}$,& $\{ 6, 8, 14, 17 \}$,& $\{ 6, 13, 16, 21 \}$.
\end{tabular}
\end{center}

$\cdot$ Type $3^1 5^1 16^1$: $\alpha=(1)$.

\begin{center}
\begin{tabular}{llllll}
    $\{ 0, 3, 5, 8 \}$,& $\{ 0, 4, 6, 10 \}$,& $\{ 0, 7, 9, 12 \}$,& $\{ 0, 11, 14, 16 \}$,& $\{ 0, 13, 15, 18 \}$,& $\{ 0, 17, 20, 22 \}$,\\
  $\{ 0, 19, 21, 23 \}$,& $\{ 1, 3, 15, 20 \}$,& $\{ 2, 12, 15, 23 \}$,& $\{ 6, 8, 15, 19 \}$,& $\{ 7, 10, 15, 17 \}$,& $\{ 4, 11, 15, 22 \}$,\\
  $\{ 5, 9, 15, 21 \}$,& $\{ 1, 5, 13, 17 \}$,& $\{ 2, 5, 11, 20 \}$,& $\{ 5, 7, 14, 23 \}$,& $\{ 5, 10, 18, 22 \}$,& $\{ 5, 12, 16, 19 \}$,\\
  $\{ 2, 10, 13, 19 \}$,& $\{ 1, 7, 11, 19 \}$,& $\{ 4, 9, 17, 19 \}$,& $\{ 3, 14, 19, 22 \}$,& $\{ 1, 6, 12, 22 \}$,& $\{ 2, 9, 16, 22 \}$,\\
  $\{ 7, 8, 13, 22 \}$,& $\{ 1, 9, 18, 23 \}$,& $\{ 6, 9, 14, 20 \}$,& $\{ 3, 9, 11, 13 \}$,& $\{ 6, 11, 17, 23 \}$,& $\{ 8, 11, 18, 21 \}$,\\
  $\{ 3, 10, 16, 23 \}$,& $\{ 3, 12, 17, 21 \}$,& $\{ 6, 13, 16, 21 \}$,& $\{ 4, 13, 20, 23 \}$,& $\{ 7, 16, 18, 20 \}$,& $\{ 1, 4, 8, 16 \}$,\\
  $\{ 1, 10, 14, 21 \}$,& $\{ 4, 12, 14, 18 \}$,& $\{ 2, 3, 6, 18 \}$,& $\{ 2, 4, 7, 21 \}$,& $\{ 2, 8, 14, 17 \}$,& $\{ 8, 10, 12, 20 \}$.
\end{tabular}
\end{center}

$\cdot$ Type $3^1 5^1 8^1 8^1$: $\alpha=(1)$.

\begin{center}
\begin{tabular}{llllll}
   $\{ 0, 3, 5, 8 \}$,& $\{ 0, 4, 6, 10 \}$,& $\{ 0, 7, 9, 12 \}$,& $\{ 0, 11, 14, 16 \}$,& $\{ 0, 13, 15, 18 \}$,& $\{ 0, 17, 20, 22 \}$,\\
  $\{ 0, 19, 21, 23 \}$,& $\{ 1, 4, 8, 12 \}$,& $\{ 2, 4, 14, 19 \}$,& $\{ 4, 7, 18, 22 \}$,& $\{ 4, 9, 16, 21 \}$,& $\{ 4, 11, 13, 20 \}$,\\
  $\{ 4, 15, 17, 23 \}$,& $\{ 3, 6, 14, 22 \}$,& $\{ 8, 13, 16, 22 \}$,& $\{ 8, 11, 17, 19 \}$,& $\{ 6, 8, 18, 20 \}$,& $\{ 7, 8, 14, 21 \}$,\\
  $\{ 2, 8, 10, 23 \}$,& $\{ 1, 6, 13, 23 \}$,& $\{ 5, 7, 11, 23 \}$,& $\{ 7, 10, 13, 17 \}$,& $\{ 3, 9, 13, 19 \}$,& $\{ 2, 5, 13, 21 \}$,\\
  $\{ 3, 12, 20, 23 \}$,& $\{ 1, 3, 17, 21 \}$,& $\{ 6, 11, 15, 21 \}$,& $\{ 2, 3, 11, 18 \}$,& $\{ 3, 10, 15, 16 \}$,& $\{ 1, 9, 11, 22 \}$,\\
  $\{ 2, 12, 15, 22 \}$,& $\{ 2, 7, 16, 20 \}$,& $\{ 1, 7, 15, 19 \}$,& $\{ 1, 5, 16, 18 \}$,& $\{ 1, 10, 14, 20 \}$,& $\{ 2, 6, 9, 17 \}$,\\
  $\{ 5, 12, 14, 17 \}$,& $\{ 5, 10, 19, 22 \}$,& $\{ 5, 9, 15, 20 \}$,& $\{ 6, 12, 16, 19 \}$,& $\{ 9, 14, 18, 23 \}$,& $\{ 10, 12, 18, 21 \}$.
\end{tabular}
\end{center}

$\cdot$ Type $3^1 5^1 7^1 9^1$: $\alpha=(1)$.

\begin{center}
\begin{tabular}{llllll}
     $\{ 0, 3, 5, 8 \}$,& $\{ 0, 4, 6, 10 \}$,& $\{ 0, 7, 9, 12 \}$,& $\{ 0, 11, 13, 15 \}$,& $\{ 0, 14, 17, 19 \}$,& $\{ 0, 16, 20, 22 \}$,\\
  $\{ 0, 18, 21, 23 \}$,& $\{ 1, 4, 8, 12 \}$,& $\{ 4, 14, 15, 22 \}$,& $\{ 4, 9, 18, 20 \}$,& $\{ 2, 4, 11, 17 \}$,& $\{ 4, 13, 16, 21 \}$,\\
  $\{ 4, 7, 19, 23 \}$,& $\{ 8, 11, 20, 23 \}$,& $\{ 8, 10, 16, 19 \}$,& $\{ 6, 8, 15, 21 \}$,& $\{ 7, 8, 17, 22 \}$,& $\{ 2, 8, 13, 18 \}$,\\
  $\{ 5, 7, 13, 20 \}$,& $\{ 7, 10, 15, 18 \}$,& $\{ 2, 7, 14, 16 \}$,& $\{ 1, 7, 11, 21 \}$,& $\{ 2, 3, 15, 20 \}$,& $\{ 5, 12, 15, 17 \}$,\\
  $\{ 1, 9, 15, 19 \}$,& $\{ 10, 12, 14, 20 \}$,& $\{ 1, 6, 17, 20 \}$,& $\{ 3, 10, 17, 21 \}$,& $\{ 9, 13, 17, 23 \}$,& $\{ 1, 10, 13, 22 \}$,\\
  $\{ 2, 5, 10, 23 \}$,& $\{ 2, 12, 19, 21 \}$,& $\{ 2, 6, 9, 22 \}$,& $\{ 3, 12, 18, 22 \}$,& $\{ 6, 12, 16, 23 \}$,& $\{ 6, 11, 14, 18 \}$,\\
  $\{ 1, 3, 14, 23 \}$,& $\{ 1, 5, 16, 18 \}$,& $\{ 3, 9, 11, 16 \}$,& $\{ 3, 6, 13, 19 \}$,& $\{ 5, 9, 14, 21 \}$,& $\{ 5, 11, 19, 22 \}$.
\end{tabular}
\end{center}

$\cdot$ Type $3^1 5^1 6^1 10^1$: $\alpha=(1)$.

\begin{center}
\begin{tabular}{llllll}
      $\{ 0, 3, 8, 11 \}$,& $\{ 0, 5, 7, 9 \}$,& $\{ 0, 4, 6, 14 \}$,& $\{ 0, 10, 12, 16 \}$,& $\{ 0, 13, 15, 18 \}$,& $\{ 0, 17, 20, 22 \}$,\\
  $\{ 0, 19, 21, 23 \}$,& $\{ 1, 5, 11, 13 \}$,& $\{ 3, 5, 16, 19 \}$,& $\{ 5, 8, 14, 22 \}$,& $\{ 2, 5, 18, 23 \}$,& $\{ 5, 10, 15, 20 \}$,\\
  $\{ 5, 12, 17, 21 \}$,& $\{ 8, 10, 17, 23 \}$,& $\{ 7, 13, 20, 23 \}$,& $\{ 9, 11, 16, 23 \}$,& $\{ 1, 4, 12, 23 \}$,& $\{ 3, 6, 15, 23 \}$,\\
  $\{ 4, 11, 18, 20 \}$,& $\{ 2, 6, 11, 17 \}$,& $\{ 7, 11, 14, 21 \}$,& $\{ 11, 15, 19, 22 \}$,& $\{ 1, 7, 17, 19 \}$,& $\{ 4, 9, 15, 17 \}$,\\
  $\{ 1, 8, 15, 21 \}$,& $\{ 2, 7, 12, 15 \}$,& $\{ 3, 13, 14, 17 \}$,& $\{ 7, 8, 16, 18 \}$,& $\{ 4, 7, 10, 22 \}$,& $\{ 2, 4, 8, 19 \}$,\\
  $\{ 4, 13, 16, 21 \}$,& $\{ 2, 3, 10, 21 \}$,& $\{ 2, 9, 13, 22 \}$,& $\{ 2, 14, 16, 20 \}$,& $\{ 1, 6, 16, 22 \}$,& $\{ 1, 3, 9, 20 \}$,\\
  $\{ 1, 10, 14, 18 \}$,& $\{ 9, 12, 14, 19 \}$,& $\{ 6, 9, 18, 21 \}$,& $\{ 3, 12, 18, 22 \}$,& $\{ 6, 8, 12, 20 \}$,& $\{ 6, 10, 13, 19 \}$.
\end{tabular}
\end{center}

$\cdot$ Type $3^1 4^1 17^1$: $\alpha=(1)$.

\begin{center}
\begin{tabular}{llllll}
   $\{ 0, 3, 5, 7 \}$,& $\{ 0, 4, 6, 9 \}$,& $\{ 0, 7, 9, 12 \}$,& $\{ 0, 11, 14, 23 \}$,& $\{ 0, 12, 15, 19 \}$,& $\{ 0, 16, 20, 22 \}$,\\
  $\{ 0, 18, 21, 23 \}$,& $\{ 1, 3, 15, 21 \}$,& $\{ 5, 7, 14, 21 \}$,& $\{ 10, 14, 17, 21 \}$,& $\{ 2, 6, 8, 20 \}$,& $\{ 4, 11, 16, 18 \}$,\\
  $\{ 8, 13, 19, 23 \}$,& $\{ 6, 12, 15, 21 \}$,& $\{ 5, 9, 17, 23 \}$,& $\{ 9, 19, 21, 23 \}$,& $\{ 2, 8, 15, 23 \}$,& $\{ 4, 7, 15, 22 \}$,\\
  $\{ 1, 8, 10, 13 \}$,& $\{ 4, 8, 17, 20 \}$,& $\{ 3, 7, 13, 20 \}$,& $\{ 9, 12, 18, 22 \}$,& $\{ 2, 5, 17, 20 \}$,& $\{ 2, 4, 12, 19 \}$,\\
  $\{ 1, 4, 13, 21 \}$,& $\{ 4, 9, 14, 21 \}$,& $\{ 2, 11, 13, 22 \}$,& $\{ 5, 13, 20, 22 \}$,& $\{ 5, 11, 14, 17 \}$,& $\{ 1, 5, 11, 15 \}$,\\
  $\{ 3, 9, 11, 17 \}$,& $\{ 1, 9, 16, 19 \}$,& $\{ 6, 7, 13, 15 \}$,& $\{ 6, 10, 13, 22 \}$,& $\{ 1, 6, 11, 19 \}$,& $\{ 1, 8, 11, 18 \}$,\\
  $\{ 3, 12, 18, 20 \}$,& $\{ 2, 7, 10, 16 \}$,& $\{ 2, 3, 14, 20 \}$,& $\{ 3, 10, 19, 21 \}$,& $\{ 6, 13, 19, 23 \}$,& $\{ 8, 10, 13, 16 \}$.
\end{tabular}
\end{center}

$\cdot$ Type $3^1 4^4 5^1$: $\alpha=(1)$.

\begin{center}
\begin{tabular}{llllll}
    $\{ 0, 3, 5, 7 \}$,& $\{ 1, 3, 9, 11 \}$,& $\{ 2, 3, 8, 13 \}$,& $\{ 3, 10, 12, 15 \}$,& $\{ 3, 14, 17, 19 \}$,& $\{ 3, 16, 21, 23 \}$,\\
  $\{ 3, 18, 20, 22 \}$,& $\{ 2, 5, 9, 16 \}$,& $\{ 1, 5, 12, 19 \}$,& $\{ 5, 8, 14, 22 \}$,& $\{ 5, 10, 18, 23 \}$,& $\{ 5, 13, 17, 20 \}$,\\
  $\{ 5, 11, 15, 21 \}$,& $\{ 0, 8, 15, 17 \}$,& $\{ 1, 8, 20, 23 \}$,& $\{ 4, 8, 12, 16 \}$,& $\{ 8, 10, 19, 21 \}$,& $\{ 6, 8, 11, 18 \}$,\\
  $\{ 2, 10, 11, 20 \}$,& $\{ 4, 11, 17, 22 \}$,& $\{ 1, 6, 10, 17 \}$,& $\{ 0, 4, 10, 14 \}$,& $\{ 10, 13, 16, 22 \}$,& $\{ 1, 7, 15, 22 \}$,\\
  $\{ 1, 14, 16, 18 \}$,& $\{ 1, 4, 13, 21 \}$,& $\{ 2, 6, 12, 22 \}$,& $\{ 0, 9, 19, 22 \}$,& $\{ 0, 6, 16, 20 \}$,& $\{ 0, 11, 13, 23 \}$,\\
  $\{ 7, 11, 16, 19 \}$,& $\{ 7, 9, 13, 18 \}$,& $\{ 4, 9, 15, 20 \}$,& $\{ 2, 4, 18, 19 \}$,& $\{ 4, 6, 7, 23 \}$,& $\{ 6, 9, 14, 21 \}$,\\
  $\{ 9, 12, 17, 23 \}$,& $\{ 0, 12, 18, 21 \}$,& $\{ 2, 7, 17, 21 \}$,& $\{ 2, 14, 15, 23 \}$,& $\{ 6, 13, 15, 19 \}$,& $\{ 7, 12, 14, 20 \}$.
\end{tabular}
\end{center}

$\cdot$ Type $3^1 4^3 9^1$: $\alpha=(0,1,2)(3,7,11)(4,8,12)(5,9,13)(6,10,14)(15,18,21)(16,19,22)(17,20,23)$.

\begin{center}
\begin{tabular}{llllll}
   $\{ 0, 3, 15, 18 \}$,& $\{ 3, 5, 8, 21 \}$,& $\{ 3, 7, 13, 20 \}$,& $\{ 0, 7, 14, 16 \}$,& $\{ 0, 8, 10, 11 \}$,& $\{ 3, 16, 19, 23 \}$,\\
  $\{ 0, 6, 12, 22 \}$,& $\{ 4, 8, 15, 22 \}$,& $\{ 0, 5, 9, 19 \}$,& $\{ 5, 14, 18, 22 \}$,& $\{ 0, 4, 17, 20 \}$,& $\{ 0, 13, 21, 23 \}$,\\
  $\{ 6, 10, 18, 23 \}$,& $\{ 4, 9, 14, 23 \}$.
\end{tabular}
\end{center}

$\cdot$ Type $3^1 4^2 13^1$: $\alpha=(1)$.

\begin{center}
\begin{tabular}{llllll}
   $\{ 0, 3, 5, 7 \}$,& $\{ 1, 3, 9, 11 \}$,& $\{ 2, 3, 8, 13 \}$,& $\{ 3, 10, 12, 15 \}$,& $\{ 3, 14, 17, 19 \}$,& $\{ 3, 16, 20, 22 \}$,\\
  $\{ 3, 18, 21, 23 \}$,& $\{ 2, 5, 9, 14 \}$,& $\{ 5, 12, 17, 20 \}$,& $\{ 1, 5, 16, 23 \}$,& $\{ 5, 11, 15, 18 \}$,& $\{ 5, 13, 19, 21 \}$,\\
  $\{ 5, 8, 10, 22 \}$,& $\{ 1, 4, 7, 19 \}$,& $\{ 1, 8, 12, 21 \}$,& $\{ 1, 10, 14, 18 \}$,& $\{ 1, 6, 13, 20 \}$,& $\{ 1, 15, 17, 22 \}$,\\
  $\{ 4, 8, 16, 18 \}$,& $\{ 7, 9, 18, 20 \}$,& $\{ 0, 13, 18, 22 \}$,& $\{ 2, 6, 12, 18 \}$,& $\{ 7, 11, 13, 16 \}$,& $\{ 2, 10, 16, 21 \}$,\\
  $\{ 2, 7, 17, 23 \}$,& $\{ 2, 11, 19, 22 \}$,& $\{ 2, 4, 15, 20 \}$,& $\{ 9, 13, 15, 23 \}$,& $\{ 4, 10, 13, 17 \}$,& $\{ 7, 12, 14, 22 \}$,\\
  $\{ 6, 7, 15, 21 \}$,& $\{ 0, 8, 15, 19 \}$,& $\{ 6, 8, 11, 17 \}$,& $\{ 8, 14, 20, 23 \}$,& $\{ 0, 10, 11, 20 \}$,& $\{ 6, 10, 19, 23 \}$,\\
  $\{ 9, 12, 16, 19 \}$,& $\{ 0, 4, 12, 23 \}$,& $\{ 4, 11, 14, 21 \}$,& $\{ 0, 6, 14, 16 \}$,& $\{ 0, 9, 17, 21 \}$,& $\{ 4, 6, 9, 22 \}$.
\end{tabular}
\end{center}

$\cdot$ Type $3^1 4^2 6^1 7^1$: $\alpha=(1)$.

\begin{center}
\begin{tabular}{llllll}
  $\{ 0, 3, 11, 14 \}$,& $\{ 0, 5, 7, 9 \}$,& $\{ 0, 4, 6, 17 \}$,& $\{ 0, 8, 10, 19 \}$,& $\{ 0, 12, 15, 22 \}$,& $\{ 0, 13, 16, 21 \}$,\\
  $\{ 0, 18, 20, 23 \}$,& $\{ 1, 4, 8, 11 \}$,& $\{ 4, 7, 13, 20 \}$,& $\{ 4, 9, 15, 19 \}$,& $\{ 2, 4, 18, 22 \}$,& $\{ 4, 14, 16, 23 \}$,\\
  $\{ 4, 10, 12, 21 \}$,& $\{ 5, 11, 15, 18 \}$,& $\{ 1, 10, 14, 18 \}$,& $\{ 3, 8, 13, 18 \}$,& $\{ 6, 9, 18, 21 \}$,& $\{ 7, 12, 16, 18 \}$,\\
  $\{ 9, 14, 17, 20 \}$,& $\{ 1, 9, 16, 22 \}$,& $\{ 3, 9, 12, 23 \}$,& $\{ 2, 9, 11, 13 \}$,& $\{ 10, 11, 20, 22 \}$,& $\{ 6, 7, 11, 23 \}$,\\
  $\{ 11, 17, 19, 21 \}$,&\ $\{ 6, 10, 13, 15 \}$,& $\{ 2, 5, 10, 23 \}$,& $\{ 3, 10, 16, 17 \}$,& $\{ 2, 6, 16, 19 \}$,& $\{ 5, 8, 16, 20 \}$,\\
  $\{ 5, 13, 17, 22 \}$,& $\{ 1, 6, 12, 20 \}$,& $\{ 5, 12, 14, 19 \}$,& $\{ 1, 3, 5, 21 \}$,& $\{ 1, 13, 19, 23 \}$,& $\{ 1, 7, 15, 17 \}$,\\
  $\{ 2, 3, 15, 20 \}$,& $\{ 8, 15, 21, 23 \}$,& $\{ 2, 7, 14, 21 \}$,& $\{ 2, 8, 12, 17 \}$,& $\{ 3, 7, 19, 22 \}$,& $\{ 6, 8, 14, 22 \}$.
\end{tabular}
\end{center}

$\cdot$ Type $3^1 4^2 5^1 8^1$: $\alpha=(1)$.

\begin{center}
\begin{tabular}{llllll}
   $\{ 0, 3, 5, 7 \}$,& $\{ 1, 3, 9, 11 \}$,& $\{ 2, 3, 8, 13 \}$,& $\{ 3, 10, 12, 16 \}$,& $\{ 3, 14, 18, 20 \}$,& $\{ 3, 15, 17, 22 \}$,\\
  $\{ 3, 19, 21, 23 \}$,& $\{ 2, 5, 9, 15 \}$,& $\{ 1, 5, 12, 20 \}$,& $\{ 5, 8, 18, 23 \}$,& $\{ 5, 13, 16, 22 \}$,& $\{ 5, 10, 17, 21 \}$,\\
  $\{ 5, 11, 14, 19 \}$,& $\{ 2, 7, 17, 19 \}$,& $\{ 1, 4, 6, 17 \}$,& $\{ 0, 8, 12, 17 \}$,& $\{ 9, 14, 17, 23 \}$,& $\{ 11, 13, 17, 20 \}$,\\
  $\{ 1, 7, 13, 23 \}$,& $\{ 4, 12, 15, 23 \}$,& $\{ 6, 9, 12, 19 \}$,& $\{ 2, 12, 18, 21 \}$,& $\{ 7, 12, 14, 22 \}$,& $\{ 1, 15, 16, 18 \}$,\\
  $\{ 1, 8, 14, 21 \}$,& $\{ 1, 10, 19, 22 \}$,& $\{ 0, 13, 15, 19 \}$,& $\{ 4, 8, 16, 19 \}$,& $\{ 6, 10, 13, 18 \}$,& $\{ 6, 8, 11, 22 \}$,\\
  $\{ 8, 10, 15, 20 \}$,& $\{ 6, 7, 15, 21 \}$,& $\{ 0, 11, 16, 21 \}$,& $\{ 4, 7, 11, 18 \}$,& $\{ 0, 6, 20, 23 \}$,& $\{ 7, 9, 16, 20 \}$,\\
  $\{ 2, 4, 20, 22 \}$,& $\{ 2, 10, 11, 23 \}$,& $\{ 0, 4, 10, 14 \}$,& $\{ 0, 9, 18, 22 \}$,& $\{ 2, 6, 14, 16 \}$,& $\{ 4, 9, 13, 21 \}$.
\end{tabular}
\end{center}

$\cdot$ Type $3^1 4^1 8^1 9^1$: $\alpha=(1)$.

\begin{center}
\begin{tabular}{llllll}
   $\{ 0, 3, 5, 7 \}$,& $\{ 1, 3, 9, 11 \}$,& $\{ 2, 3, 8, 13 \}$,& $\{ 3, 10, 12, 15 \}$,& $\{ 3, 14, 17, 19 \}$,& $\{ 3, 16, 20, 22 \}$,\\
  $\{ 3, 18, 21, 23 \}$,& $\{ 2, 5, 9, 15 \}$,& $\{ 5, 11, 14, 22 \}$,& $\{ 1, 5, 18, 20 \}$,& $\{ 5, 10, 17, 23 \}$,& $\{ 5, 13, 16, 21 \}$,\\
  $\{ 5, 8, 12, 19 \}$,& $\{ 4, 10, 13, 20 \}$,& $\{ 2, 6, 20, 23 \}$,& $\{ 7, 11, 15, 20 \}$,& $\{ 0, 8, 17, 20 \}$,& $\{ 9, 12, 14, 20 \}$,\\
  $\{ 4, 6, 11, 17 \}$,& $\{ 8, 11, 16, 18 \}$,& $\{ 2, 11, 19, 21 \}$,& $\{ 0, 11, 13, 23 \}$,& $\{ 2, 12, 17, 22 \}$,& $\{ 7, 9, 17, 21 \}$,\\
  $\{ 1, 13, 15, 17 \}$,& $\{ 1, 8, 14, 23 \}$,& $\{ 4, 8, 15, 21 \}$,& $\{ 6, 8, 10, 22 \}$,& $\{ 0, 10, 14, 21 \}$,& $\{ 6, 14, 15, 18 \}$,\\
  $\{ 0, 15, 19, 22 \}$,& $\{ 1, 4, 7, 22 \}$,& $\{ 1, 10, 16, 19 \}$,& $\{ 1, 6, 12, 21 \}$,& $\{ 7, 12, 16, 23 \}$,& $\{ 0, 4, 12, 18 \}$,\\
  $\{ 0, 6, 9, 16 \}$,& $\{ 6, 7, 13, 19 \}$,& $\{ 2, 4, 14, 16 \}$,& $\{ 2, 7, 10, 18 \}$,& $\{ 4, 9, 19, 23 \}$,& $\{ 9, 13, 18, 22 \}$.
\end{tabular}
\end{center}

$\cdot$ Type $3^1 4^1 7^1 10^1$: $\alpha=(1)$.

\begin{center}
\begin{tabular}{llllll}
   $\{ 0, 3, 5, 7 \}$,& $\{ 1, 3, 9, 11 \}$,& $\{ 2, 3, 8, 14 \}$,& $\{ 3, 10, 12, 16 \}$,& $\{ 3, 13, 15, 18 \}$,& $\{ 3, 17, 20, 22 \}$,\\
  $\{ 3, 19, 21, 23 \}$,& $\{ 2, 5, 9, 15 \}$,& $\{ 5, 12, 18, 21 \}$,& $\{ 2, 4, 12, 20 \}$,& $\{ 2, 10, 13, 21 \}$,& $\{ 2, 11, 17, 19 \}$,\\
  $\{ 2, 7, 18, 22 \}$,& $\{ 2, 6, 16, 23 \}$,& $\{ 1, 5, 13, 17 \}$,& $\{ 5, 11, 20, 23 \}$,& $\{ 5, 8, 10, 22 \}$,& $\{ 5, 14, 16, 19 \}$,\\
  $\{ 4, 9, 13, 23 \}$,& $\{ 1, 10, 18, 23 \}$,& $\{ 7, 15, 17, 23 \}$,& $\{ 0, 8, 12, 23 \}$,& $\{ 4, 8, 17, 21 \}$,& $\{ 0, 6, 10, 17 \}$,\\
  $\{ 9, 12, 14, 17 \}$,& $\{ 4, 10, 15, 19 \}$,& $\{ 7, 10, 14, 20 \}$,& $\{ 0, 11, 15, 21 \}$,& $\{ 1, 6, 14, 21 \}$,& $\{ 6, 8, 13, 19 \}$,\\
  $\{ 6, 12, 15, 22 \}$,& $\{ 0, 9, 19, 22 \}$,& $\{ 0, 13, 16, 20 \}$,& $\{ 0, 4, 14, 18 \}$,& $\{ 1, 8, 15, 20 \}$,& $\{ 1, 7, 12, 19 \}$,\\
  $\{ 1, 4, 16, 22 \}$,& $\{ 4, 6, 7, 11 \}$,& $\{ 6, 9, 18, 20 \}$,& $\{ 7, 9, 16, 21 \}$,& $\{ 8, 11, 16, 18 \}$,& $\{ 11, 13, 14, 22 \}$.
\end{tabular}%
\end{center}